\documentclass[12pt]{article}
\usepackage{authblk}
\usepackage{bbm}
\usepackage{url}
\usepackage{soul}

\usepackage[nottoc,numbib]{tocbibind}

\usepackage{hyperref}
\hypersetup{
    colorlinks=true,
    linkcolor=blue,
    filecolor=magenta,      
    urlcolor=cyan,
    citecolor=blue,
    }

\usepackage[refpage]{nomencl}

\usepackage{framed}
\usepackage[normalem]{ulem}
\usepackage{amsmath}

\usepackage[utf8]{inputenc}
\usepackage{scalerel}
\usepackage{enumerate}
\usepackage{authblk}
\usepackage{wrapfig}
\usepackage{tikz-cd}
\usepackage[new]{old-arrows}
\usepackage{amsmath,amscd}
\usepackage{amsthm}
\usepackage{amsfonts}
\usepackage{color}
\usepackage{thmtools}
\usepackage{thm-restate}
\usepackage{amssymb}
\usepackage{stmaryrd}
\usepackage[margin=1in]{geometry}
\usepackage{enumitem}
\usepackage{cleveref}
\usepackage{yhmath} 

\newtheorem*{theorem*}{Main Theorem}

\newtheorem{theorem}{Theorem}
\newtheorem{proposition}{Proposition}[section]
\newtheorem{corollary}{Corollary}[section]
\newtheorem{lemma}{Lemma}[section]
\newtheorem{remark}{Remark}[section]
\newtheorem{claim}{Claim}[section]
\newtheorem{question}{Question}
\newtheorem{conjecture}{Conjecture}

\theoremstyle{definition}
\newtheorem{definition}{Definition}
\newtheorem{example}{Example}[section]
\newcommand{\Z}{\mathbb{Z}}

\newcommand{\R}{\mathbb{R}}
\newcommand{\C}{\mathbb{C}}
\newcommand{\Sp}{\mathbb{S}}
\newcommand{\SpE}{\mathbb{S}_{\mathrm{E}}}

\newcommand{\Pfin}{\mathcal{P}_{\mathrm{fin}}}
\newcommand{\diam}{\mathrm{diam}}
\newcommand{\dis}{\mathrm{dis}}

\newcommand{\cx}{\mathbf{\mathfrak{i}}}

\newcommand{\codis}{\mathrm{codis}}
\newcommand{\dgh}{d_\mathrm{GH}}

\newcommand{\filrad}{\mathrm{FillRad}}

\newcommand{\h}{\mathrm{H}}
\newcommand{\gl}{\mathrm{GL}}
\newcommand{\supp}{\mathrm{supp}}

\newcommand{\toG}{\stackrel{G}{\longrightarrow}}

\newcommand{\hooktoG}{\stackrel{G}{\longhookrightarrow}}
\newcommand{\toZtwo}{\stackrel{\mathbb{Z}_2}{\longrightarrow}}
\newcommand{\Gsimeq}{\stackrel{G}{\simeq}}

\newcommand{\TS}{\mathbf{E}}

\newcommand{\Hom}{\mathrm{H}}

\newcommand{\vr}{\mathrm{VR}}
\newcommand{\vrm}{\mathrm{VR}^{\mathrm{m}}}

\newcommand{\dht}{d_\mathrm{HT}}
\newcommand{\di}{d_\mathrm{I}}

\makenomenclature

\begin{document}

\title{$G$-Gromov-Hausdorff Distances\\ and\\ Equivariant Topology}	
   
\author[1]{Sunhyuk Lim}
\author[2]{Facundo M\'emoli}

\affil[1]{Department of Mathematics\\ Sungkyunkwan University (SKKU)\\
 	\texttt{lsh3109@skku.edu}}
 \affil[2]{Department of Mathematics\\
 	Rutgers University\\ 	\texttt{facundo.memoli@rutgers.edu}}

\date{\today}
\maketitle

\begin{abstract}
For each arbitrary finite group $G$, we consider a suitable notion of Gromov–Hausdorff distance between compact $G$-metric spaces and derive lower bounds based on equivariant topology methods. As applications, we prove equivariant rigidity and finiteness theorems, and obtain sharp bounds on the Gromov–Hausdorff distance between spheres.
\end{abstract}
  
\newpage

\printnomenclature[2cm]
\addcontentsline{toc}{section}{Nomenclature}

\tableofcontents

\section{Introduction.}

In this paper, we prove the following equivariant stability results pertaining to invariants of \emph{$G$-metric spaces}: metric spaces equipped with actions of a group $G$ by isometries.

\label{page:main-theorem}

\begin{theorem*}[Stability results]
For any finite group $G$, any two compact $G$-metric spaces $X$ and $Y$, $p\in[1,\infty]$, and every integer $k\geq 0$, we have the following inequalities:
$$
2\,\dgh^G(X,Y)\geq\dht^G\big((\Pfin(X),\diam_p^X),(\Pfin(Y),\diam_p^Y)\big) \geq \left\{\begin{array}{ll} \di^G\big(\h_k(\vrm_p(X;\bullet)),\h_k(\vrm_p(Y;\bullet))\big)\\
    \\
    \sup_{\varepsilon\geq 0}c_p^G(X,Y;\varepsilon)
\end{array}
\right. 
$$
\noindent
where, for each $\varepsilon>0$,

$$c_p^G(X,Y;\varepsilon):=\inf\{\eta>0:\text{there is a }G\text{-map }\vrm_p(Y;\varepsilon)\toG \vrm_p(X;\varepsilon+\eta)\}$$
and, for $\epsilon=0$,
$$\hspace{-.89in}c_p^G(X,Y;0):=\inf\{\eta>0:\text{there is a }G\text{-map }Y\toG \vrm_p(X;\eta)\}.$$
\end{theorem*}

In the statement above, all $G$-actions are by isometries and the notation is as follows.
\begin{itemize}
\item $\dgh^G$ is a variant of the \emph{Gromov-Hausdorff} distance designed to interact well with given $G$-equivariant isometric actions on compact metric spaces; see \Cref{sec:quotient} for a discussion of how it relates with the standard Gromov-Hausdorff distance between quotient spaces and \Cref{rem:other-G-GH} for a comparison with other extant notions of Gromov-Hausdorff distance between $G$-metric spaces.
\item $\Pfin(X)$ is the collection of all finitely supported probability measures $\mu$ on $X$ with a $G$-action induced by push-forward via the action on $X$, 
\item $\diam_p^X(\mu):=\big(\iint d_X^p(x,x')\,d\mu(x)d\mu(x')\big)^{1/p}$ is the \emph{$p$-diameter} of $\mu$, and 
\item  $\vrm_p(X;\bullet)$ is the \emph{$p$-Vietoris-Rips metric thickening filtration of $X$}: for $\varepsilon>0$,  $\vrm_p(X;\varepsilon)$ is the collection of all  $\mu\in\Pfin(X)$  such that $\diam_p^X(\mu)<\varepsilon$. These spaces are nested under inclusion as $\varepsilon$ increases;
see \Cref{def:VRm}.
\item $\dht^G$ and $\di^G$ are suitable $G$-equivariant variants of the  \emph{homotopy type distance} from \cite{frosini2017persistent} and the \emph{interleaving distance} \cite{chazal2009proximity} between \emph{persistent modules}, respectively.

\end{itemize}
The theorem above is obtained from \Cref{Gdhtstability}, \Cref{coro:dHTG-dIG-Katetov}, and \Cref{prop:strongercGpXYlbdd}.  

\medskip
The quantity $c_p^G$ is an instance of the concept of \emph{$G$-persistent index} which we describe in \Cref{def:clowerbdd}. In general, the two rightmost lower bounds are incomparable; see~\Cref{ex:compatibility}. Furthermore, we consider the following variants of $c_p^G(X,Y;\varepsilon)$ (see \Cref{def:clowerbdd}):

\begin{itemize}
    \item $c_{\vr}^G(X,Y;\varepsilon):=\inf\{\eta>0:\text{there is a }G\text{-map }\vert\vr(Y;\varepsilon)\vert\toG \vert\vr(X;\varepsilon+\eta)\vert\}$, and

    \item $c_\TS^G(X,Y;\varepsilon):=\inf\{\eta>0:\text{there is a }G\text{-map }B_\varepsilon(Y,\TS(Y))\toG B_{\varepsilon+\eta}(X,\TS(X))\}$,
\end{itemize}
where $\vr(X;\bullet)$ denotes the \emph{Vietoris-Rips} filtration of $X$ (see \Cref{def:VR}) and $\TS(X)$ denotes its \emph{tight span} (see \Cref{def:TS}). Precise connections between the three quantities $c_\infty^G$, $c_\vr^G$, and $c_\TS^G$, as well as applications, can be found in \Cref{prop:strongercGpXYlbddalernatives} and in \Cref{sec:GdGHresults-sph}. In particular, under mild conditions, we have that $$c_\infty^G(X,Y;0)=c_{\vr}^G(X,Y;0)=2c_\TS^G(X,Y;0).$$ Finally, we have that $c_\infty^{\Z_2}(\Sp^m,\Sp^{m+1};0)=\zeta_m:=\arccos\left(\frac{-1}{m+1}\right)$; see \Cref{rmk:cinftyZ2SmSm+1}. The invariant $c_p^G$ generalizes\footnote{When $X$ and $Y$ are spheres and   $G=\Z_2$, the case $(\varepsilon=0, p=\infty)$ of the $c_p^G$  boils down to the invariant considered in \cite{adams2022gromov}; see \Cref{rmk:referpolymath} and \Cref{sec:rel-pis}.} and improves\footnote{See \Cref{rmk:referpolymath} and \Cref{sec:lb-dght-pi}.} upon an invariant introduced in \cite{adams2022gromov}. 

We employ the stability results in the Main Theorem to prove novel $G$-equivariant rigidity (\Cref{thm:rigidity} and \Cref{cor:rigidity})  and finiteness  (\Cref{cor:Riemfinite})  theorems. In \Cref{sec:GdGHresults}, we consider topological spheres with different metric realizations: the $\Z_2$-metric spaces $\Sp^n$, $\SpE^n$, $\Sp^n_\infty$, and $\square^n_\infty$ equipped with their canonical $\Z_2$-actions (see  \Cref{ex:spheres-Z2}) and, using the Main Theorem, we establish lower bounds—and in some cases, exact values—for the $\Z_2$-equivariant Gromov–Hausdorff distances $\dgh^{\Z_2}(\Sp^m,\Sp^n)$, $\dgh^{\Z_2}(\SpE^m,\SpE^n)$, $\dgh^{\Z_2}(\Sp^m_\infty,\Sp^n_\infty)$, and $\dgh^{\Z_2}(\square^m_\infty,\square^n_\infty)$. On a different vein, in \Cref{sec:qbu} we show how our techniques can be used to obtain a $G$-equivariant \emph{quantitative Borsuk-Ulam theorem} that generalizes \cite{dubins1981equidiscontinuity,lim2021gromov,adams2022gromov}. See also \cite{adams2025quantifying} for a detailed study of geometric applications of this circle of ideas (see  \Cref{sec:qbu} for additional context).

Our constructions and results offer a unified framework that both extends and generalizes recent uses of $\Z_2$-equivariant topology, which have been instrumental in obtaining sharp bounds on the Gromov–Hausdorff distance between spheres \cite{lim2021gromov,adams2022gromov}; see also \cite{harrison2023quantitative,martin2024some,memoli2024embedding}.

\paragraph{Acknowledgements}
This work was partially supported by NSF DMS \#2301359, NSF CCF \#2310412, NSF CCF \#2112665, NSF CCF \#2217058, and NSF RI \#1901360. The genesis of this project goes back to a polymath-type activity that was run between UF, CMU, OSU, FU-Berlin and led to \cite{adams2022gromov}.

\section{Preliminaries.}\label{sec:preliminary}

Here we collect all basic terminology, definitions and the notation we will frequently employ in this paper.

\begin{itemize}\label{notations}
    \item A \emph{function} $f:X\longrightarrow Y$ is any, possibly discontinuous, function.
    
    \item A \emph{map} $f:X\longrightarrow Y$ is a continuous function.    

    \item For a bounded metric space $(X,d_X)$, $\diam(X):=\sup\limits_{x,x'\in X}d_X(x,x')$ is the \emph{diameter} of $X$. \nomenclature[01]{$\diam(X)$}{Diameter of $X$}

    \item For a nonnegative integer $n\geq 0$, let
    $$\Sp^n:=\left\{(x_0,x_1,\cdots,x_n)\in\R^{n+1}:\sum_{i=0}^{n}x_i^2=1\right\}$$
    be the usual round unit $n$-sphere and let $d_{\Sp^n}$ denote the geodesic distance on $\Sp^n$. i.e., $d_{\Sp^n}(u,v):=\arccos(\langle u,v \rangle)$ for all $u,v\in\Sp^n$. Alternatively, one may consider the Euclidean metric $d_{\SpE^n}$ on $\Sp^n$ inherited from $\R^{n+1}$, i.e. $d_{\SpE^n}(u,v):=\Vert u-v \Vert$ for all $u,v\in\Sp^n$ where $\Vert\cdot\Vert$ denotes the Euclidean norm in $\R^{n+1}$. Finally, we will use the notation $\SpE^n$ instead of $\Sp^n$ whenever the underlying $n$-sphere is equipped with the Euclidean metric $d_{\SpE^n}$ instead of the geodesic metric $d_{\Sp^n}$. Finally, let $\zeta_n$ denote the geodesic distance between two different vertices of a regular $(n+1)$-simplex inscribed in $\Sp^n$; i.e. $\zeta_n:=\arccos\left(\frac{-1}{n+1}\right)$. Note that $\zeta_1=\tfrac{2\pi}{3}$, $\zeta_2\approx 1.9106$ and $\tfrac{\pi}{2}\leq \zeta_n\leq\pi$, in general.\nomenclature[05]{$\zeta_n$}{The edge length of a regular $(n+1)$-simplex inscribed in $\Sp^n$}

    \item For an nonnegative integer $n\geq 0$, let
    \begin{itemize}
    \item $\R^n_\infty=(\R^n,\ell^\infty)$.

    \item $\mathbb{D}^n_\infty:=(\mathbb{D}^n,\ell^\infty)$ where $\mathbb{D}^n$ is the usual closed unit ball in $\R^n$ with the Euclidean distance.

    \item $\Sp^n_\infty:=(\Sp^n,\ell^\infty)$ where $\Sp^n = \partial \mathbb{D}^{n+1}$ is the usual unit (round)  sphere.
    
    \item $\blacksquare^n_\infty:=(\blacksquare^n,\ell^\infty)$ where $$\blacksquare^n:=\{(x_0,\dots,x_{n-1})\in\R^n:x_i\in [-1,1]\textrm{ for all }i=1,\dots,n\}$$ is the closed unit ball in $\R_\infty^n$.
  
    \item $\square^n_\infty:=(\square^n,\ell^\infty)$ where 
    \begin{align*}\square^n &:= \partial \blacksquare^{n+1}\\
    &= \{(x_0,\dots,x_n)\in\R^{n+1}:(x_0,\dots,x_n)\in\blacksquare^{n+1}_\infty \textrm{ and }x_i=\pm 1\textrm{ for some }i=0,\dots,n\}
    \end{align*}
    is the unit sphere in $\R^{n+1}_\infty$.
\end{itemize}
\end{itemize}

\nomenclature[02]{$\Sp^n$}{$n$-dimensional unit sphere with the geodesic metric}
\nomenclature[03]{$\SpE^n$}{$n$-dimensional unit sphere with the Euclidean metric}
\nomenclature[04]{$\Sp^n_\infty$}{$n$-dimensional unit sphere with the $\ell^\infty$ metric}
\nomenclature[05]{$\square^n_\infty$}{$n$-dimensional cube with the $\ell^\infty$ metric}

\subsection{Filtrations.}

\begin{definition}[Filtrations] \label{def:pers-fams}
A \emph{filtration} of a topological space $Z$ is a collection $U_\bullet = \big(U_r,\iota_{r,s} \big)_{r\leq s \in \R}$ where:
\begin{itemize}

\item for each $r\in \R$, $U_r$ is a subspace of $Z$, and

\item for each $r\leq s$, $U_r\subseteq U_s$ and
 $\iota_{r,s}: U_r \hookrightarrow U_s$ is the canonical inclusion map. 
\end{itemize}
\end{definition}

When there is no risk of confusion, we will simply say that $U_\bullet$ is a filtration without mentioning the topological space $Z$, with the understanding that $Z$ can be recovered as the colimit of $U_\bullet$.

\medskip
The following two constructions of filtrations will be considered in the rest of the paper.
\begin{definition}[Sub-level set filtrations]\label{def:sub-level} 
Suppose the pair $(Z,\beta_Z)$ is given where $Z$ is a topological space and $\beta_Z:Z\rightarrow\R$ is a real-valued map. Then, the family $\big(\beta_Z^{-1}((-\infty,r))\big)_{r\in \R}$ together with the canonical inclusions is a filtration (of $Z$). Filtrations of this type are often called \emph{sub-level set filtrations}.
\end{definition}

\begin{definition}[Vietoris-Rips filtration]\label{def:VR}
Let $X$ be a metric space and $r>0$. The \emph{(open) Vietoris-Rips complex} $\vr(X;r)$ of $X$ at scale $r$ is the simplicial complex whose vertices are the points of $X$ and whose simplices are those finite subsets of $X$ with diameter strictly less then $r$. For $r\leq 0$ we define $\vr(X;r):=\emptyset$. Note that $\vr(X;r)\subseteq \vr(X;s)$ whenever $r\leq s$.

\medskip
Note that, at the level of the geometric realization, 
{\small$$|\vr(X;r)|=\left\{\sum_{i=1}^n u_i x_i: n\in \mathbb{N},x_i\in X \text{ s.t. $\diam(\{x_1,\ldots,x_n\})<r$ and $u_i\geq 0$ s.t. $\sum_{i=1}^n u_i = 1$}\right\}.$$}

The above implies that there exists a canonical inclusion $\iota_{r,s}:\vert\vr(X;r)\vert\longhookrightarrow\vert\vr(X;s)\vert$ for  $r\leq s$. Also, $\vert\vr(X;r)\vert$ is a subspace of $\vert\vr(X;\infty)\vert$ for any $r\in\R$ where $\vert\vr(X;\infty)\vert$ is the colimit of $\big(\vert\vr(X;r)\vert,\iota_{r,s} \big)_{r\leq s}$.
The family $\vert\vr(X;\bullet)\vert$ of nested geometric realizations is a filtration, called the \emph{(open) Vietoris-Rips filtration} of $X$.
\end{definition}

Note that the Vietoris-Rips filtration $\vert\vr(X;\bullet)\vert$ of $X$ is identical to the sub-level set filtration induced by the pair $\big(\vert\vr(X;\infty)\vert,\diam\big)$ where
\begin{align*}
    \diam:\vert\vr(X;\infty)\vert&\longrightarrow\R\\
    \sum_{i=1}^nu_ix_i&\longmapsto\diam(\{x_1,x_2,\dots,x_n\}).
\end{align*}

Under suitable assumptions on $X$, Hausmann proved in \cite{hausmann1994vietoris} that $|\vr(X;r)|$ is homotopy equivalent to $X$ for small enough $r>0$.  In \Cref{thm:Z2hausmannsphere} we provide an optimal $\Z_2$-equivariant Hausmann-type theorem for spheres.

\subsection{The Vietoris-Rips metric thickening of a metric space.}

Let $X$ be a metric space and $\mathcal{P}(X)$ be the space of all Borel probability measures on $X$ with the weak topology \cite{dudley2018real}. Given a map $\varphi:X\to Y$ where $Y$ is a metric space,  the \emph{pushforward} $\varphi_\sharp\alpha$ of $\alpha\in\mathcal{P}(X)$ via $\varphi$ is the probability measure on $Y$ defined by $\varphi_\sharp\alpha(U):=\alpha\big(\varphi^{-1}(U)\big)$ for all Borel subsets $U\subset Y$.\nomenclature[06]{$\mathcal{P}(X)$}{Space of all Borel probability measures on $X$ with the weak topology.}

For each $r> 0$ and $p\in[1,\infty]$, the \emph{$p$-Vietoris-Rips metric thickening} $\vrm_p(X;r)$ of $X$ at scale $r$ is the subspace of probability measures~$\mu$ on $X$ whose \emph{support} $\supp[\mu]$ is finite and has $p$-diameter smaller than $r$, equipped with the inherited weak topology. Precise definitions are as follows; see \cite{adamaszek2018metric,adams2024persistent} for other aspects.

\begin{definition}[$p$-diameter]\label{def:pdiam}
Let $(X,d_X)$ be a bounded metric space and $p\in[1,\infty]$. Then, the $p$-diameter map $\diam_p^X:\mathcal{P}(X)\rightarrow [0,\infty)$ is defined as follows: 
$$\diam_p^X(\mu):=\begin{cases}\left(\iint_{X\times X}d_X^p(x,x')\,d\mu(x)\,d\mu(x')\right)^{\frac{1}{p}} &\text{if }p<\infty,\\
\\\diam(\supp[\mu]) &\text{if }p=\infty.\end{cases}$$\nomenclature[08]{$\diam_p^X(\mu)$}{$p$-diameter of the probability measure $\mu \in \mathcal{P}(X)$}
\end{definition}
Note that if $1\leq p \leq q \leq \infty$, then we have $\diam_p(\mu)\leq\diam_q(\mu)$ for every $\mu\in\mathcal{P}(X)$.

From now on, the set of all probability measures on $X$ with finite support will be denoted by $\Pfin(X)$. That is
$$\Pfin(X):=\left\{\mu := \sum_{i=0}^n u_i\delta_{x_i}:u_i \geq 0, \sum_i u_i = 1, x_1,\ldots,x_n \in X, n\in\mathbb{N}\right\}$$
where $\delta_x$ denotes the Dirac delta measure supported on the point $x\in X.$ We view $\Pfin(X)\subset \mathcal{P}(X)$ as a topological space by endowing it with the inherited (weak) topology.\nomenclature[07]{$\Pfin(X)$}{Subspace of $\mathcal{P}(X)$ with the finite support}

\begin{definition}[$p$-Vietoris-Rips metric thickening filtrations]\label{def:VRm}
Let $(X,d_X)$ be a metric space, $r> 0$, and $p\in[1,\infty]$. The \emph{$p$-Vietoris-Rips metric thickening} $\vrm_p(X;r)\subseteq \Pfin(X)$ of $X$ at scale $r$ is defined in the following way: 
$$\vrm_p(X;r):=\left\{\mu \in\Pfin(X): \diam_p^X(\mu)<r\right\}.$$
Furthermore, we regard $\vrm_p(X;r)$ as a topological space by endowing it with the weak topology inherited from $\mathcal{P}(X)$. For $r\leq 0$ we define $\vrm_p(X;r):=\emptyset$. Note that if $r\leq s$, then $\vrm_p(X;r)$ is contained in $\vrm_p(X;s)$. Hence, the family $\vrm_p(X;\bullet)$ is a filtration, called the \emph{$p$-Vietoris-Rips metric thickening filtration} of $X$. \nomenclature[09]{$\vrm_p(X)$}{$p$-Vietoris-Rips metric thickening of $X$}
\end{definition}

Observe that, for each $r>0$ we have that 
$$\vrm_p(X;r):=\left(\diam_p^X\vert_{\Pfin(X)}\right)^{-1}\big((-\infty,r)\big).$$

In other words, the filtration $\vrm_p(X;\bullet)$ can be regarded as a sub-level set filtration. This representation will be used in subsequent sections to express the stability of the assignment $X\mapsto \vrm_p(X;\bullet).$

Note that one may metrize $\vrm_p(X;r)$ with the $q$-Wasserstein distance \cite[Theorem 7.3]{villani2021topics} for any $q\in[1,\infty]$. Furthermore, if $X$ is compact and $q<\infty$, the weak topology on $\vrm_p(X;r)$ and the metric topology induced by the $q$-Wasserstein distance coincide \cite{dudley2018real}. Hence, the superscript $\mathrm{m}$ denotes ``metric'', since the metric thickening $\vrm_p(X;r)$ can be seen as a metric space, whereas the simplicial complex $\vr(X;r)$ may not be metrizable if $X$ is not discrete \cite{adamaszek2018metric}.
By identifying each point $x \in X$ with the Dirac measure~$\delta_{x}$, there is a natural isometric (i.e. distance preserving) embedding from $X$ into~$\vrm_p(X;r)$, via the injective map $x\mapsto \delta_x$.
Furthermore, note that the underlying set of the metric thickening $\vrm_\infty(X;r)$ is equal to the underlying set of (the geometric realization of) the simplicial complex $\vr(X;r)$, although the topology of these two spaces may differ~\cite{adamaszek2018metric}.
Also note that for each $r>0$ there is a distinguished map $\mathrm{id}_r:|\vr(X;r)|\to \vr^m_\infty(X;r)$ given by  $$\sum_{i=1}^n u_i x_i \mapsto \sum_{i=1}^n u_i\delta_{x_i}.$$
Recently, in \cite{gillespie2024vietoris} Gillespie proved that $\mathrm{id}_r$ is for each $r>0$ a weak homotopy equivalence so that his results imply that the Vietoris-Rips metric thickening filtration is functorially \emph{weakly} homotopy equivalent to the Vietoris-Rips filtration of $X$; see \cite[Section 7]{gillespie2024vietoris}.

Under suitable assumptions on $X$, \cite{adamaszek2018metric} establishes a Hausmann-type result for the Vietoris-Rips metric thickening; $\vrm_\infty(X;r)$ is homotopy equivalent to $X$ for small enough $r>0$.  \Cref{thm:vrmHausmann-G},  extends this result to the $G$-equivariant setting. In \Cref{thm:Z2hausmannspherevrm} we specialize this to the case of spheres with $\Z_2$ actions with an optimal range for $r$.

\subsection{Tight span of a metric space.}\label{sec:katetov}

In his seminal paper \cite{gromov1983filling} Gromov introduced the notion of \emph{filling radius} $\filrad(M)$ of an orientable $m$-dimensional Riemannian manifold $M$ by considering its isometric embedding into $L^\infty(M)$ via the Kuratowski map  $\kappa:M\to L^\infty(M)$ given by $x\mapsto d_M(x,\cdot).$ Then, $\filrad(M)$ is defined as the infimal $\delta>0$ such that $\Hom_m(\iota_\delta;\Z)([M])=0$ where $\iota_\delta: M\to B_\delta(M; L^\infty(M))$ is the map induced by the inclusion $\kappa(M)\subseteq B_\delta(M,L^\infty(M))$ into the $\delta$-neighborhood of $\kappa(M)$ in $L^\infty(M)$.
In this paper, we will employ a closely related but ``smaller" space.

\begin{definition}[{Tight span \cite{dress1984trees,isbell1964six}}]\label{def:TS}
The \emph{tight span} $\TS(X)$ of a compact metric space $(X,d_X)$ is defined as follows:
$$\TS(X):=\{f\in\Delta(X):\text{if }g\in\Delta(X)\text{ and }g\leq f,\text{ then }g=f\,(\text{i.e., }f\text{ is minimal})\}$$

where $\Delta(X):=\{f\in L^\infty(X):f(x)+f(x')\geq d_X(x,x')\text{ for all }x,x'\in X\}$. We endow $\TS(X)$ with the metric induced by $\ell^\infty$-norm.
\end{definition}

It is known that both $\TS(X)$ and $L^\infty(X)$ are metric spaces that satisfy the property called \emph{injectivity} \cite{dress1984trees,isbell1964six}, or equivalently \emph{hyperconvexity} \cite{aronszajn1956extension}. In particular, $\TS(X)$ is known to be the smallest injective metric space into which $X$ can be embedded, and it is unique up to isometry. Furthermore, the tubular neighborhoods of a given compact metric space $(X,d_X)$ inside of its tight span  are functorially homotopy equivalent to the Vietoris-Rips complexes of $X$.  More precisely, for $r>0$ let 
$$B_r(X,\TS(X)):=\{f\in\TS(X):\Vert f-d_X(x,\cdot)\Vert<r\text{ for some }x\in X\}.$$
Let \emph{tight span neighborhood filtration} of the compact metric space $X$ be the filtration
$$B_\bullet(X,\TS(X)):=\big(B_r(X,\TS(X))\subseteq B_s(X,\TS(X))\big)_{r<s}$$
with the provision that $B_r(X,\TS(X)):=\emptyset$ for $r\leq 0$. 

\begin{theorem}[{\cite[Theorem 4.1]{lim2020vietoris}}]\label{thm:vr-neigh}
There exist homotopy equivalences $\varphi_{r}:\vr_{2r}(X)\rightarrow B_r(X,\TS(X))$ for each $r>0$ such that for each $t>s>0$ the following diagram commutes up to homotopy:
$$\begin{tikzcd}
\vr_{2r}(X) \arrow[r, hook] \arrow[d, "\varphi_{r}"' , rightarrow]
& \vr_{2s}(X) \arrow[d, "\varphi_{s}"]\\
B_r(X, \TS(X)) \arrow[r,hook] & B_s(X, \TS(X))
\end{tikzcd}$$
\end{theorem}

\section{$G$-equivariant distances.}\label{sec:Gdistances}

Throughout this paper, $G$ will  denote a group and $G_0$ will denote the trivial group. Given a group homomorphism $h:\widetilde{G}\to G$ and any $G$-group action $\alpha_X:G\times X \to X$ on a set $X$, by $h^\sharp\alpha_X:\widetilde{G}\times X\to X$ we will denote the \emph{pullback $\widetilde{G}$-action on X} defined through $$(h^\sharp\alpha_X)(\tilde{g},x):=\alpha_X(h(\tilde{g}),x)\,\,\text{for $\tilde{g}\in \widetilde{G}$ and $x\in X$.}$$

\begin{definition}[$G$-set, $G$-topological space, $G$-function, $G$-map, $G$-metric space, and $G$-isometry]
Suppose that a group $G$ is given. We define the following concepts:

\begin{itemize}

    \item A pair $(X,\alpha_X)$ is a \emph{$G$-set} if $X$ is a set and $\alpha_X:G\times X\to X$ is a (left) $G$-group action on $X$.

    \item A pair $(X,\alpha_X)$ is a \emph{$G$-topological space} if $X$ is a topological space and $\alpha_X:G\times X\rightarrow X$ is a (left) $G$-group action on $X$ such that $\alpha_X(g,\cdot):X\rightarrow X$ is a homeomorphism for all $g\in G$.

    \item A \emph{$G$-function} (resp. \emph{$G$-map}) between two $G$-sets (resp. $G$-spaces) $X$ and $Y$ is any function (resp. map) $\varphi:X\rightarrow Y$ such that $$\varphi(\alpha_X(g,x)) = \alpha_Y(g,\varphi(x))$$ for all $x\in X$ and all $g\in G$. We will use the notation $\varphi:X\toG Y$ to denote that $\varphi$ is a $G$-function or a $G$-map.

    \item A triple $(X,d_X,\alpha_X)$ is said to be a \emph{$G$-metric space} if $(X,d_X)$ is a metric space and $\alpha_X:G\times X\rightarrow X$ is a (left) $G$-group action on $X$ such that $\alpha_X(g,\cdot):X\rightarrow X$ is an isometry for all $g\in G$. 

    \item If a given $G$-map $\varphi:X\toG Y$ between two $G$-metric spaces $X$ and $Y$ is  distance preserving, then $\varphi$ is called a \emph{$G$-isometric embedding}. Finally, if a given $G$-map $\varphi:X\toG Y$ is an isometry, then $\varphi$ is called a \emph{$G$-isometry}.
\end{itemize}
\end{definition}

Note that all $G$-objects defined above boil down to their usual counterparts (sets, topological spaces, etc) whenever $G$ is the trivial group $G_0$. We will henceforth denote by $\mathcal{M}^G$ the collection of all compact $G$-metric spaces. Note that $G$-isometry is the natural notion of isomorphism on $\mathcal{M}^G$; see \Cref{thm:dghGmetric}.\nomenclature[10]{$\mathcal{M}^G$}{The collection of all compact $G$-metric spaces}

\begin{example}\label{ex:spheres-Z2}
Here are some simple examples of $G$-(sets, topological spaces, metric spaces):
\begin{enumerate}
    \item Let $G$ be an arbitrary group and $X$ be an arbitrary (set, topological space, metric space). If we consider the $G$-group action $\alpha_X$ on $X$ defined as follows:
$$\alpha_X(g,\cdot):=\mathrm{id}_X\,\,\text{for all}\,\,g\in G,$$
    then $X$ can be promoted to a  $G$-(set, topological space, metric space). This $\alpha_X$ is called the \emph{trivial} $G$-group action.

    \item Here are a few interesting choices of groups $G$ and group actions $\alpha_X$ on $n$-spheres $(X,d_X)=(\Sp^n,d_{\Sp^n})\text{ or }(\SpE^n,d_{\SpE^n})$ so that $(X,d_X,\alpha_X)$ becomes a $G$-metric space:
    \begin{enumerate}
        \item Let $G=\Z_2:=\{-1,1\}$ and $\alpha_X(-1,u):=-u$ for all $u\in X$ (the antipodal  \emph{involution}). 

        \item For odd $n=2k-1$, let $G=\Sp^1$ where we identify $\Sp^1$ with the real line modulo the equivalence relation $t\sim s$ iff $t-s = 2\pi m$ for some integer $m$. Also, consider the following $G$-group action by isometries,
        $$\alpha_X(\theta,\cdot):=T_\theta\,\,\text{for all}\,\,\theta\in\Sp^1$$
        where $T_\theta$ is a rotation given, in matrix form, by
        $$\begin{bmatrix}
        M(\theta) & 0 & 0 & 0 &\cdots & 0\\
        0 & M(\theta) & 0 & 0 &\cdots & 0\\
        0 & 0 & M(\theta) & 0 &\cdots & 0\\
        0 & 0 & 0 & M(\theta) &\cdots & 0\\
        \vdots & \vdots & \vdots & \vdots &\ddots & \vdots\\
        0 & 0 & 0 & 0 &\cdots & M(\theta)\\
\end{bmatrix} \in \R^{2k\times 2k}$$
 where
$$M(\theta):=\begin{bmatrix}
\cos\theta & -\sin\theta\\
\sin\theta & \cos\theta
\end{bmatrix}.$$
    \end{enumerate}

    \item Let $G=\Z_2$, $(X,d_X)=(\square^n_\infty,\ell^\infty)\text{ or }(\Sp^n_\infty,\ell^\infty)$, and $\alpha_X(\pm1,u):=\pm u$ for all $u\in X$. Then, $(X,d_X,\alpha_X)$ becomes a $\Z_2$-metric space.
\end{enumerate}
\end{example}

\subsection{The $G$-Gromov-Hausdorff distance.}\label{sec:sec:ggh}

Recall that a \emph{correspondence} between two sets $X$ and $Y$ is any surjective relation $R\subset X\times Y$; see \cite[Chapter 7]{burago2022course}. The set of all correspondences between $X$ and $Y$ is denoted by $\mathcal{R}(X,Y).$
\begin{definition}[$G$-correspondence]
For any two $G$-sets  $(X,\alpha_X)$ and $(Y,\alpha_Y)$, a correspondence $R$ between $X$ and $Y$ is said to be a \emph{$G$-correspondence} if
$$(\alpha_X(g,x),\alpha_Y(g,y))\in R$$ for all $(x,y)\in R$ and all $g\in G$. The set of all $G$-correspondences between $X$ and $Y$ will be denoted by $\mathcal{R}_G(X,Y)$.
\end{definition}

\begin{remark}\label{rmk:R-order} Note that
\begin{enumerate}

    \item if $G=G_0$ then $G$-correspondences are just (standard) correspondences.

    \item  for any $G$-sets $(X,\alpha_X)$, $(Y,\alpha_Y)$, a $G$-function $\varphi:X\toG Y$, and for any group homomorphism $h:\widetilde{G}\to G$ we have that $\varphi$ is also a $\widetilde{G}$-function between the $\widetilde{G}$-sets $(X,h^\sharp\alpha_X)$ and $(Y,h^\sharp\alpha_Y)$.

    \item for any $G$-sets $(X,\alpha_X)$ and $(Y,\alpha_Y)$, for any $G$-correspondence $R$ between them, and for any group homomorphism $h:\widetilde{G}\to G$ we have that $R$ is also a $\widetilde{G}$-correspondence between the $\widetilde{G}$-sets $(X,h^\sharp\alpha_X)$ and $(Y,h^\sharp\alpha_Y)$.    
    
    \item If $\varphi:X\toG Y$ and $\psi:Y\toG X$ are $G$-functions, then $R(\varphi,\psi):=\mathrm{graph}(\varphi)\cup \mathrm{graph}(\psi)$ is a $G$-correspondence.
\end{enumerate}
\end{remark}

\begin{definition}[Distortion and codistortion]
Suppose two bounded metric spaces $(X,d_X)$ and $(Y,d_Y)$ are given. Then, for any non-empty relation $R\subseteq X\times Y$, its distortion is defined by $$\dis(R):=\sup_{(x,y),(x',y')\in R}\big|d_X(x,x')-d_Y(y,y')\big|.$$
Also, for any (not necessarily continuous) functions $\varphi:X\rightarrow Y$ and $\psi:Y\rightarrow X$, their distortion and codistortion are defined by
\begin{align*}
    \dis(\varphi)&:=\sup_{x,x'\in X}\big|d_X(x,x')-d_Y(\varphi(x),\varphi(x'))\big|,\\
    \codis(\varphi,\psi)&:=\sup_{x\in X,y\in Y}\big|d_X(x,\psi(y))-d_Y(\varphi(x),y)\big|.
\end{align*}
\end{definition}
Note that, with the above definitions,
    $$\dis(R(\varphi,\psi))=\max\{\dis(\varphi),\dis(\psi),\codis(\varphi,\psi)\}.$$

\begin{definition}[The $G$-Gromov-Hausdorff distance]
For any two $G$-metric spaces $X$ and $Y$, the $G$-Gromov-Hausdorff distance is defined as:
$$\dgh^G(X,Y):=\frac{1}{2}\inf\{\dis(R):R\in\mathcal{R}_G(X,Y)\}.$$ \nomenclature[11]{$\dgh^G$}{$G$-Gromov-Hausdorff distance}
\end{definition}

\begin{remark}
    Note that one has the following equivalent expression $$\dgh^G(X,Y) = \frac{1}{2}\inf\left\{\sup_{g,g'\in G} \dis\big(R_{g,g'}\big)\,:\, R\in \mathcal{R}(X,Y) \right\}$$
    where, for $g,g'\in G$,  $R_{g,g'}:=(\alpha_X(g,\cdot),\alpha_Y(g',\cdot)(R).$ Notice that in this expression, the infimum is taken over the set of all (standard) correspondences between $X$ and $Y$.
\end{remark}

The following proposition is a generalization of \cite[Theorem 2.1]{kalton1999distances} to the $G$-equivariant setting. See \Cref{app:proofs-GHG} for its proof.

\begin{proposition}\label{prop:GGHeqmaps}
For any $G$-metric spaces $X$ and $Y$ we have
$$\dgh^G(X,Y)=\frac{1}{2}\inf\Big\{\dis(R(\varphi,\psi)):\varphi:X\to Y \text{and } \psi:Y\to X\text{ are }G\text{-functions}\Big\}.$$
\end{proposition}

The following proposition is a generalization of \cite[Theorem 7.3.25]{burago2022course} to the $G$-equivariant setting. See \Cref{app:proofs-GHG} for its proof.

\begin{proposition}\label{prop:GGHeqembeds}
For any $G$-metric spaces $X$ and $Y$ we have
$$\dgh^G(X,Y)=\inf_{Z,\iota_X,\iota_Y}d_{\mathrm{H}}^Z(\iota_X(X),\iota_Y(Y))$$
where the infimum is taken over all $G$-metric spaces $Z$ and $G$-isometric embeddings $\iota_X:X\hooktoG Z$ and $\iota_Y:Y\hooktoG Z$.
\end{proposition}

The following theorem is a generalization of \cite[Theorem 7.3.30]{burago2022course} to the $G$-equivariant setting. See \Cref{app:proofs-GHG} for its proof.

\begin{theorem}\label{thm:dghGmetric}
Assume that $G$ is a finite or countably infinite group. Then, the $G$-Gromov-Hausdorff distance $\dgh^G$ is a metric on the collection of $G$-isometry classes of compact $G$-metric spaces.
\end{theorem}
 
\begin{remark}\label{rmk:G-dGH}
We make several remarks:
\begin{enumerate}
\item Note that $\mathcal{R}_G(X,Y)$ is nonempty since $X\times Y$ always belongs to $\mathcal{R}_G(X,Y)$. Moreover, this implies that
$$\dgh^G(X,Y)\leq\frac{1}{2}\max\{\diam(X),\diam(Y)\}$$
for any two $G$-metric spaces $(X,d_X,\alpha_X)$ and $(Y,d_Y,\alpha_Y)$.

\item By item 1 of \Cref{rmk:R-order}, if $G$ is the trivial group $G_0$, then $\dgh^G$ reduces to the standard Gromov-Hausdorff distance.

\item Assume a group homomorphism $h:\widetilde{G}\to G$ is given.  Then, any two $G$-metric spaces, $(X,d_X,\alpha_X)$ and $(Y,d_Y,\alpha_Y)$ induce the $\widetilde{G}$-metric spaces $(X,d_X,h^\sharp\alpha_X)$ and $(Y,d_Y,h^\sharp\alpha_Y)$, respectively. Then, via  item 3 of \Cref{rmk:R-order} we see that 
$$\dgh^G\big((X,d_X,\alpha_X),(Y,d_Y,\alpha_Y)\big)\geq \dgh^{\widetilde{G}}\big((X,d_X,h^\sharp\alpha_X),(Y,d_Y,h^\sharp\alpha_Y)\big).$$
In particular, $\dgh^G$ is always an upper bound to the standard Gromov-Hausdorff distance.

\item  Let $(X,d_X)$ be a metric space and let $\alpha_1$, $\alpha_2$ be two $G$-actions on $X$ and $R$ be a $G$-correspondence $R$ between $(X,d_X,\alpha_1)$ and $(X,d_X,\alpha_2)$. Then, it is easy to verify that 
$$\dis(R)\geq\inf_{(x_1,x_2)\in X\times X}\sup_{g\in G}\big\vert d_X(x_1,\alpha_1(g,x_1))-d_X(x_2,\alpha_2(g,x_2))\big\vert.$$
Hence, we have that
\begin{align*}\dgh^G\big((X,d_X,\alpha_1),(X,d_X,\alpha_2)\big)&\geq\frac{1}{2}\inf_{(x_1,x_2)\in X\times X}\sup_{g\in G}\big\vert d_X(x_1,\alpha_1(g,x_1))-d_X(x_2,\alpha_2(g,x_2))\big\vert\\
&\geq \frac{1}{2}\inf_{(x_1,x_2)\in X\times X}\big|e_1(x_1)-e_2(x_2)\big|,
\end{align*}
where $e_i:X\to \R_+$ is given by $x\mapsto\sup_{g\in G}d_X(x,\alpha_i(g,x))$, for $i=1,2$. For example, if $X=\Sp^n$, $G=\Z_2$, $\alpha_1$ is the canonical antipodal action and $\alpha_2$ is trivial, then $e_1 \equiv \pi$ whereas $e_2\equiv 0.$ This implies that

$$\dgh\big((\Sp^n,d_{\Sp^n},\alpha_1),(\Sp^n,d_{\Sp^n},\alpha_2)\big)=\tfrac{\pi}{2}.$$
\end{enumerate}
\end{remark}

\begin{remark}[Other equivariant versions of the Gromov-Hausdorff distance]\label{rem:other-G-GH} We make several historical remarks.

\begin{enumerate}
\item There are several  existing definitions of equivariant Gromov–Hausdorff topology that are related to, but distinct from, the metric topology induced by our $G$-Gromov–Hausdorff distance. 
\begin{itemize}
    \item In \cite[Definitions 1–2]{fukaya1986theory}, Fukaya defines an equivariant version of the Gromov–Hausdorff distance between ``pointed group metric space triples", which can be seen as a pointed version of $G$-metric spaces. This definition is well-suited for non-compact spaces and, in contrast to our  version, allows the groups acting on the two spaces to be different. Using this distance, the author proves rigidity and finiteness-type results with respect to equivariant diffeomorphisms for Riemannian orbifolds. Moreover, the same distance is employed to show that the fundamental groups of almost nonnegatively curved manifolds are almost nilpotent \cite[Main Theorem 0.1]{fukaya1992fundamental}.

    \item By  altering the distance in \cite{fukaya1986theory}, Harvey \cite[Definition 2.3]{harvey2016equivariant} introduces an  equivariant Gromov–Hausdorff topology for compact ``unpointed" $G$-metric spaces, where the acting groups $G$ can still be different. The author employs this topology in order to prove a stability theorem for $G$-Alexandrov spaces (see \cite[Theorem A]{harvey2016equivariant}). On $\mathcal{M}^G$, the collection of all compact $G$-metric spaces, one can verify that the metric topology induced by our $G$-Gromov–Hausdorff distance is strictly finer than that induced by Harvey’s definition.
    
    \item In \cite[Definition 3]{paulin1988topologie}, Paulin introduces an alternative notion of  equivariant Gromov–Hausdorff topology for the space of (not necessarily compact) $G$-metric spaces. Similarly to the case of Harvey's equivariant topology discussed above, one can verify that the  topology induced by our $G$-Gromov–Hausdorff distance is strictly finer than that induced by Paulin’s definition.
\end{itemize}
Note that the definitions of Fukaya and Harvey allow comparison between spaces with actions by \emph{different} groups. By contrast, Paulin’s definition, like ours, assumes a single fixed group $G$.

    \item The case when $G=\Z_2$ of $\dgh^G$  was previously considered in \cite[Section 4.4.2]{polterovich2020topological}. In that work, the authors computed a nontrivial lower bound for $\dgh^{\Z_2}(X,Y)$ for specific $\Z_2$-metric spaces $X$ and $Y$, where the $\Z_2$-action on $Y$ arises from a $\Z_4$-action. They achieved this by applying a stability inequality (stated without proof) between the $\Z_2$-Gromov-Hausdorff distance and the $\Z_2$-interleaving distance (see Definition \ref{def:gint-dist}), in conjunction with \cite[Theorem 4.4.11]{polterovich2020topological}; see also Example \ref{ex:dIGdIstrictineq}.
\end{enumerate}
\end{remark}

\subsubsection{$\dgh^G(X,Y)$ versus $\dgh(X/G,Y/G)$}\label{sec:quotient}

It is natural to try to relate the $G$-Gromov-Hausdorff distance between two $G$-metric spaces to the standard Gromov-Hausdorff distance between their quotients.
Let $X\in\mathcal{M}^G$. Recall that the quotient space $X/G$ can be endowed with the canonical quotient metric $d_{X/G}(\rho,\upsilon):=d_X(\rho,\upsilon)$ for any $\rho,\upsilon\in X/G$. Note that, if $\rho=[x]$ and $\upsilon=[x']$ for some $x,x'\in X$, then we have $d_{X/G}(\rho,\upsilon)=\inf_{g\in G}d_X(x,\alpha_X(g,x'))$. In general, $d_{X/G}$ defines only a pseudometric, but it becomes a genuine metric when all $G$-orbits are closed in $X$, for instance when $G$ is finite. 

\begin{proposition}\label{prop:dghGvsquotientdgh}
Let $X,Y\in\mathcal{M}^G$. Then we have that $\dgh^G(X,Y)\geq\dgh(X/G,Y/G)$.

\end{proposition}
\begin{proof} 
Let $\pi_X:X\rightarrow X/G$ and $\pi_Y:Y\rightarrow Y/G$ be the canonical projection maps to the quotient spaces. Fix an arbitrary $G$-correspondence $R\in\mathcal{R}_G(X,Y)$. Now, let $[R]:=(\pi_X\times\pi_Y)(R)\subseteq X/G \times Y/G$. We claim that $\dis([R])\leq\dis(R)$, which completes the proof.

Fix arbitrary $(\rho,\upsilon),(\rho',\upsilon')\in [R]$ and $\varepsilon>0$. Then, one can choose $(x,y),(x',y')\in R$ such that $\rho=[x]$, $\upsilon=[y]$, $\rho'=[x']$, and $\upsilon'=[y']$.  By the definition of the quotient metric, there exists $g_\varepsilon\in G$ such that $d_Y(y,\alpha_Y(g_\varepsilon,y'))<d_{Y/G}(\upsilon,\upsilon')+\varepsilon$. Let $x'':=\alpha_X(g_\varepsilon,x')$ and $y'':=\alpha_Y(g_\varepsilon,y')$. Note that $(x'',y'')\in R$ since $R$ is a $G$-correspondence. Hence, we have that
$$d_{X/G}(\rho,\rho')-d_{Y/G}(\upsilon,\upsilon') < d_X(x,x'')-d_Y(y,y'')+\varepsilon\leq \dis(R)+\varepsilon.$$
Since $(\rho,\upsilon)$, $(\rho',\upsilon')$, and $\varepsilon$ were arbitrary, it follows that $\dis([R])\leq\dis(R)$ as desired.
\end{proof}

\begin{example}
By Proposition \ref{prop:dghGvsquotientdgh}, we have $\dgh^{\Z_2}(\Sp^m,\Sp^n)\geq\dgh(\R\mathbb{P}^m,\R\mathbb{P}^n)$. This inequality can also be established in an alternative way. Indeed, by item (1) of Remark \ref{rmk:G-dGH} we have $\dgh(\R\mathbb{P}^m,\R\mathbb{P}^n)\leq\frac{1}{2}\max\{\diam(\R\mathbb{P}^m),\diam(\R\mathbb{P}^n)\}=\frac{\pi}{4}$. Moreover, by Theorem \ref{Gdhtstability}, Proposition \ref{prop:Z2dGHSmSn}, and Corollary \ref{cor:SmSnlbddvr}, we have $\dgh^{\Z_2}(\Sp^m,\Sp^n)\geq\frac{\pi}{4}$ for $m\neq n$. This, in fact, gives the (improved) bound:
$$\dgh^{\Z_2}(\Sp^m,\Sp^n)\geq \frac{\pi}{4}\geq \dgh(\R\mathbb{P}^m,\R\mathbb{P}^n)\,\,\mbox{for all $m\neq n$}.$$
\end{example}

Furthermore, the following example demonstrates that the inequality in Proposition \ref{prop:dghGvsquotientdgh} can be strict. In fact, there are non $G$-isometric $G$-spaces whose quotients are, however, isometric.

\begin{example}\label{ex:ghGvsquotientdgh}
Consider the following $\Z_3$-metric spaces. Let $X:=\{x_0,x_1,x_2\}$ and $Y:=\{y_0,y_1,y_2\}$ be three-point spaces with the obvious $\Z_3$-actions and metrics defined by $d_X(x_i,x_j):=1-\delta_{i,j}$ and $d_Y(y_i,y_j):=2(1-\delta_{i,j})$. Then, it is straightforward to verify that $\dgh^{\Z_3}(X,Y)=\frac{1}{2}$, while $\dgh(X/\Z_3,Y/\Z_3)=0$ since both of $X/\Z_3$ and $Y/\Z_3$ are singletons.
\end{example}

Example \ref{ex:ghGvsquotientdgh} shows that a large ``gap" can exist between the LHS and the RHS of the  inequality in \Cref{prop:dghGvsquotientdgh}.
The following proposition and example explain the nature of this gap. Note that it is easy to verify that $$d_{\mathrm{H}}^{Z}(\iota_X(X),\iota_Y(Y))=d_{\mathrm{H}}^{Z/G}\big([\iota_X](X/G),[\iota_Y](Y/G)\big)$$ for arbitrary $G$-metric space $Z$ and $G$-isometric embeddings $\iota_X:X\hooktoG Z$ and $\iota_Y:Y\hooktoG Z$. Then, by invoking Proposition \ref{prop:GGHeqembeds}, we have the following result.

\begin{proposition}\label{prop:GGHeqembedsquotient}
For any $G$-metric spaces $X$ and $Y$ we have
$$\dgh^G(X,Y)=\inf_{Z,\iota_X,\iota_Y}d_{\mathrm{H}}^{Z/G}\big([\iota_X](X/G),[\iota_Y](Y/G)\big),$$
where the infimum is taken over all $G$-metric spaces $Z$ and $G$-isometric embeddings $\iota_X:X\hooktoG Z$ and $\iota_Y:Y\hooktoG Z$, and $[\iota_X]:X/G\longhookrightarrow Z/G$, $[\iota_Y]:Y/G\longhookrightarrow Z/G$ are the induced isometric embeddings.
\end{proposition}

Observe that Proposition \ref{prop:GGHeqembedsquotient} provides an alternative proof of Proposition \ref{prop:dghGvsquotientdgh}. This follows from the fact that the infimum of $d_{\mathrm{H}}^{Z/G}([\iota_X](X/G),[\iota_Y](Y/G))$ over all $G$-metric spaces $Z$ and $G$-isometric embeddings $\iota_X:X\hooktoG Z$ and $\iota_Y:Y\hooktoG Z$ is greater than or equal to the infimum of $d_{\mathrm{H}}^{Z/G}(\iota_{X/G}(X/G),\iota_{Y/G}(Y/G))$ over all $G$-metric spaces $Z$ and isometric embeddings $\iota_{X/G}:X/G\longhookrightarrow Z/G$ and $\iota_{Y/G}:Y/G\longhookrightarrow Z/G$. Furthermore, Proposition \ref{prop:GGHeqembedsquotient} sheds new light on Example \ref{ex:ghGvsquotientdgh} as follows.

\begin{figure}
\begin{center}
    \includegraphics[width = 0.8\linewidth]{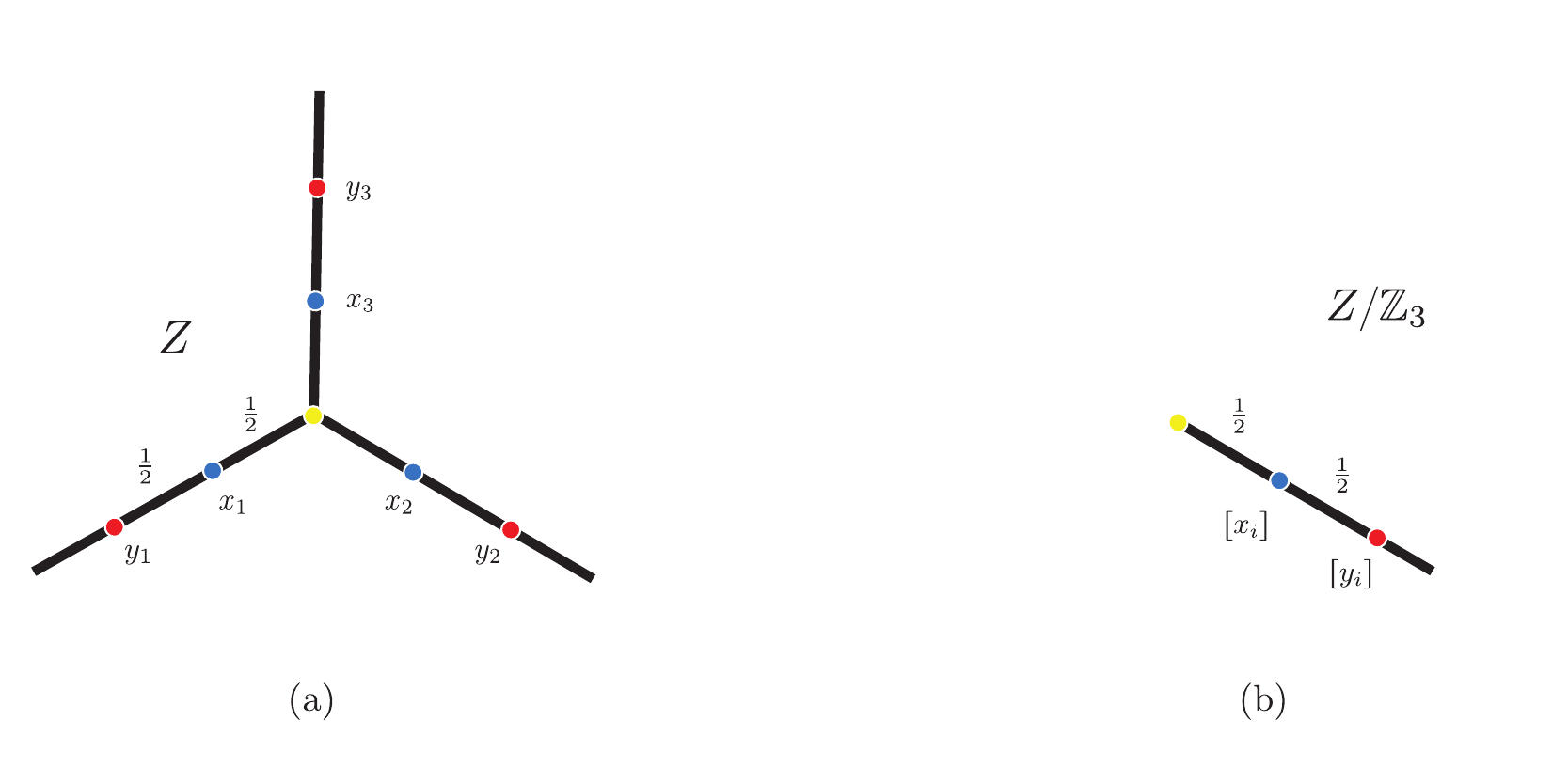}
\end{center}
    \caption{Description of $Z$ and $Z/\Z_3$ from \Cref{ex:ghGvsquotientdghHD}.}
    \label{fig:ghGvsquotientdghHD}
\end{figure}

\begin{example}\label{ex:ghGvsquotientdghHD}
Let $X$ and $Y$ be the metric spaces from Example \ref{ex:ghGvsquotientdgh}. One can easily verify that the tripod metric space $Z$ and the $\Z_3$ isometric embeddings $\iota_X:X\stackrel{\Z_3}{\longhookrightarrow}Z$, $\iota_Y:Y\stackrel{\Z_3}{\longhookrightarrow}Z$ illustrated in Figure \ref{fig:ghGvsquotientdghHD}(a), realize the optimal distance, i.e.  $d_{\mathrm{H}}^{Z}(\iota_X(X),\iota_Y(Y))=d_{\mathrm{H}}^{Z/G}([\iota_X](X/G),[\iota_Y](Y/G))=\tfrac{1}{2}$. Moreover, the associated quotient space $Z/\Z_3$ and the isometric embeddings $[\iota_X],[\iota_Y]$ are illustrated in Figure \ref{fig:ghGvsquotientdghHD}(b). This figure therefore illustrates why $\dgh^{\Z_3}(X,Y)=\frac{1}{2}$ whereas $\dgh(X/\Z_3,Y/\Z_3)=0$.
\end{example}

\subsection{The $G$-homotopy type distance.}\label{sec:dHT-G}

We extend the notion of \emph{(persistent) homotopy type distance} from \cite{frosini2017persistent} to the setting of $G$-spaces. First we need the following definitions on the equivariant topology. See \cite[Section I.1]{tom1987transformation} and \cite[Section I.1]{may1996equivariant}.

\begin{definition}[$G$-homotopy equivalence]
Suppose that two $G$-spaces $X$ and $Y$ are given. We say that two $G$-maps $\varphi:X\toG Y$ and $\varphi':X\toG Y$ are \emph{$G$-homotopic} if there is a map $H:X\times [0,1]\longrightarrow Y$ such that
\begin{enumerate}
    \item $H(\cdot,0)=\varphi$
    \item $H(\cdot,1)=\varphi'$
    \item $H(\cdot,t):X\toG Y$ is a $G$-map for each $t\in [0,1]$.
\end{enumerate}
Any such map will be called a  \emph{$G$-homotopy} between $\varphi$ and $\varphi'$. We write $\varphi\Gsimeq\varphi'$ whenever $\varphi$ and $\varphi'$ are $G$-homotopic. Also, we say that $X$ and $Y$ are \emph{$G$-homotopy equivalent} if there are $G$-maps $\varphi:X\toG Y$ and $\psi:Y\toG X$ such that $\psi\circ\varphi$ and $\mathrm{id}_X$ are $G$-homotopic and $\varphi\circ\psi$ and $\mathrm{id}_Y$ are $G$-homotopic. Finally, if $X$ and the one-point space $\{\ast\}$ are $G$-homotopy equivalent, then $X$ is said to be \emph{$G$-contractible}.
\end{definition}

We now define the concept of filtered $G$-topological spaces.

\begin{definition}[Filtered $G$-topological spaces]
Given a $G$-topological space $(X,\alpha_X)$, a map $\beta_X:X\to \R$
    is said to be \emph{$G$-invariant} if $\beta_X(\alpha_X(g,x))=\beta_X(x)$ for all $x\in X$ and $g\in G$.
    
    Any triple $(X,\alpha_X,\beta_X)$ where $(X,\alpha_X)$ is a $G$-topological space and $\beta_X:X\to \R$ is a $G$-invariant map is said to be a \emph{filtered $G$-topological space}.
\end{definition}

The following proposition shows that a $G$-group action on a metric space $X$ can be transferred to a $G$-group action on $\Pfin(X)$, the space of all probability measures on $X$ with finite support. Furthermore, with this action, $\Pfin(X)$ becomes a filtered $G$-topological space.

\begin{proposition}\label{prop:PfinGaction}
Let $(X,d_X,\alpha_X)$ be a $G$-metric space. Then the function  
\begin{align*}
\alpha_{\Pfin(X)}:G\times \Pfin(X)&\longrightarrow \Pfin(X)\\
\left(g,\mu\right)&\longmapsto  \alpha_X(g,\cdot)_\sharp \mu
\end{align*}
is a well defined $G$-group action\footnote{More explicitly, for $\mu=\sum_{i=1}^m u_i \delta_{x_i}$, $\alpha_{\Pfin(X)}(g,\mu) = \sum_{i=1}^mu_i\delta_{\alpha_X(g,x_i)}.$} on the topological space $\Pfin(X)$. Furthermore, for each $p\in [1,\infty]$, the map $\diam_p^X:\Pfin(X)\rightarrow\R$ is $G$-invariant so that $$\big(\Pfin(X),\alpha_{\Pfin(X)},\diam_p^X\big)$$ becomes a filtered $G$-topological space.
\end{proposition}
\begin{proof}
First, let's verify that $(\Pfin(X),\alpha_{\Pfin(X)})$ is indeed a $G$-topological space. It is clear that $\alpha_{\Pfin(X)}(e_G, \mu)=\mu$ where $e_G$ is the identity element of $G$ and $\alpha_{\Pfin(X)}(g_2,\alpha_{\Pfin(X)}(g_1, \mu))=\alpha_{\Pfin(X)}(g_2g_1, \mu)$ for any $g_1,g_2\in G$ and $\mu\in\Pfin(X)$. So, we only need to show that $\alpha_{\Pfin(X)}(g, \cdot):X\rightarrow X$ is a homeomorphism for any $g\in G$. Note that $\alpha_{\Pfin(X)}(g, \cdot)$ is obviously a bijection with  inverse $\alpha_{\Pfin(X)}(g^{-1}, \cdot)$. Hence, we only need to show the continuity with respect to the weak topology. Assume that a sequence $\{\mu_k\}_{k\geq 1}\subseteq\Pfin(X)$ weakly converges to $\mu\in\Pfin(X)$. i.e., for any continuous and bounded function $f:X\rightarrow\R$, we have $\int_X f\,d\mu_k\rightarrow\int_X f\,d\mu$. Then, for any $g\in G$, we have that
$$\int_X f\,d\alpha_{\Pfin(X)}(g,\mu_k)=\int_X f\circ\alpha_X(g,\cdot)\,d\mu_k\rightarrow\int_X f\circ\alpha_X(g,\cdot)\,d\mu=\int_X f\,d\alpha_{\Pfin(X)}(g,\mu).$$
Hence, $\alpha_{\Pfin(X)}(g,\mu_k)$ weakly converges to $\alpha_{\Pfin(X)}(g,\mu)$ so that $\alpha_{\Pfin(X)}(g, \cdot)$ is continuous as we required.

Second, let's prove that $\diam_p^X$ is $G$-invariant. For $p<\infty$, observe that
\begin{align*}
&\iint_{X\times X}d_X^p(x,x')\,d\alpha_{\Pfin(X)}(g,\mu)(x)\,d\alpha_{\Pfin(X)}(g,\mu)(x')&=\\
&\iint_{X\times X}d_X^p(\alpha_X(g,x),\alpha_X(g,x'))\,d\mu(x)\,d\mu(x')
&=\\
&\iint_{X\times X}d_X^p(x,x')\,d\mu(x)\,d\mu(x') 
\end{align*}
for any $g\in G$ and $\mu\in\Pfin(X)$ since $\alpha_X(g,\cdot)$ is an isometry. Therefore, $\diam_p^X(\alpha_{\Pfin(X)}(g,\mu))=\diam_p^X(\mu)$. For $p=\infty$, $\diam_\infty^X(\alpha_{\Pfin(X)}(g,\mu))=\diam_\infty^X(\mu)$ since $\supp[\alpha_{\Pfin(X)}(g,\mu)]=\alpha_X(g,\supp[\mu])$ and $\alpha_X(g,\cdot)$ is an isometry.
\end{proof}

\begin{example}\label{ex:GVRfilteredspace}
For a  metric space $X$, recall that $\vert\vr(X;\infty)\vert$ is the colimit of $\big(\vert\vr(X;r)\vert,\iota_{r,s} \big)_{r\leq s}$ (see \Cref{def:VR}). Furthermore, if $(X,d_X,\alpha_X)$ is a $G$-metric space, then there exists a canonical $G$-group action $\alpha_{\vert\vr(X;\infty)\vert}$ on $\vert\vr(X;\infty)\vert$ defined as follows:
$$\alpha_{\vert\vr(X;\infty)\vert}\left(g,\sum_{i=1}^nu_ix_i\right):=\sum_{i=1}^nu_i\alpha_X(g,x_i)\,\forall\, g\in G\text{ and }\sum_{i=1}^nu_ix_i\in\vert\vr(X;\infty)\vert.$$

Finally, it is easy to verify that the map $\diam:\vert\vr(X;\infty)\vert\longrightarrow\R$ is $G$-invariant. Hence, the triple $$\big(\vert\vr(X;\infty)\vert,\alpha_{\vert\vr(X;\infty)\vert},\diam\big)$$ is a filtered $G$-topological space.
\end{example}

\begin{example}\label{ex:spheres-Z2-vrvrm} Recall  items (2) and (3) of \Cref{ex:spheres-Z2}. By \Cref{prop:PfinGaction} and \Cref{ex:GVRfilteredspace}, we have that $(\Pfin(X),\alpha_{\Pfin(X)},\diam_p^X)$ and $(\vert\vr(X;\infty)\vert,\alpha_{\vert\vr(X;\infty)\vert},\diam)$ are filtered $\Z_2$-topological spaces for $X=\Sp^n,\SpE^n,\Sp^n_\infty,$ and $\square^n_\infty$ with the canonical $\Z_2$-actions.
\end{example}
    
\begin{definition}[$(G,\delta)$-maps]
Consider two filtered $G$-topological spaces $(X,\alpha_X,\beta_X)$ and $(Y,\alpha_Y,\beta_Y)$. For a given $\delta\geq 0$, we say that a $G$-map $\Phi:X\toG Y$ is a  \emph{$(G,\delta)$-map} if
$$\beta_Y\circ\Phi(x)\leq \beta_X(x)+\delta\text{ for all } x\in X.$$
\end{definition}

\begin{remark}\label{rmk:G-delta-map}
It is easy to verify that for any $(G,\delta)$-map $\Phi:X\rightarrow Y$ and any group homomorphism $h:\widetilde{G}\to G$, we have that $\Phi$ is also a $(\widetilde{G},\delta)$-map; cf. the item (2) of \Cref{rmk:R-order}.
\end{remark}

\begin{definition}[$(G,\delta)$-homotopies]
Suppose two filtered $G$-topological spaces $(X,\alpha_X,\beta_X)$ and $(Y,\alpha_Y,\beta_Y)$ together with $(G,\delta)$-maps $\Phi_0,\Phi_1:X\rightarrow Y$ are given. Then, we say that $\Phi_0$ and $\Phi_1$ are \emph{$(G,\delta)$-homotopic}  if there exists a map $H:X\times [0,1]\rightarrow Y$ such that:
\begin{enumerate}
    \item $H(\cdot,0)=\Phi_0$,
    
    \item $H(\cdot,1)=\Phi_1$, and
    
    \item $H(\cdot,t)$ is $(G,\delta)$-map for all $t\in [0,1]$.
\end{enumerate}
Any such $H$ map will be referred to as a \emph{$(G,\delta)$-homotopy}.
\end{definition}

\begin{definition}[The $G$-homotopy type distance]
We say that two filtered $G$-topological spaces $(X,\alpha_X,\beta_X)$ and $(Y,\alpha_Y,\beta_Y)$ are \emph{$(G,\delta)$-homotopy equivalent} if there are $(G,\delta)$-maps $\Phi:X\rightarrow Y$ and $\Psi:Y\rightarrow X$ such that
\begin{enumerate}
    \item $\Psi\circ\Phi$ is $(G,2\delta)$-homotopic  to $\mathrm{id}_X$, and
    
    \item $\Phi\circ\Psi$ is $(G,2\delta)$-homotopic  to $\mathrm{id}_Y$.
\end{enumerate}

Finally, we  define the \emph{$G$-homotopy type distance} between $(X,\alpha_X,\beta_X)$ and $(Y,\alpha_Y,\beta_Y)$ as follows:
$$\dht^G(X,Y):=\inf\{\delta\geq 0:(X,\alpha_X,\beta_X)\text{ and }(Y,\alpha_Y,\beta_Y)\text{ are }(G,\delta)\text{-homotopy equivalent}\}.$$
We set $\dht^G(X,Y)=\infty$ whenever $X$ and $Y$ fail to be $(G,\delta)$-homotopy equivalent. \nomenclature[12]{$\dht^G$}{$G$-homotopy type distance}
\end{definition}

The following proposition is a generalization of \cite[Proposition 2.10]{frosini2017persistent} to the $G$-equivariant setting.

\begin{proposition}\label{prop:dHTG}
$\dht^G$ is an extended pseudometric on the collection of all filtered $G$-topological spaces.
\end{proposition}

\begin{remark}\label{rmk:dhtGmonotone}
Note that:
\begin{enumerate}
    \item When $G=G_0$, $\dht^G$ boils down to the original homotopy type distance from \cite{frosini2017persistent}.

    \item If $h:\widetilde{G}\longrightarrow G$   is a group homomorphism, then any two filtered $G$-topological spaces $(X,\alpha_X,\beta_X)$ and $(Y,\alpha_Y,\beta_Y)$ induce the filtered $\widetilde{G}$-topological space $(X,h^\sharp\alpha_X,\beta_X)$ and $(Y,h^\sharp\alpha_Y,\beta_Y)$, respectively. Furthermore, by \Cref{rmk:G-delta-map}, for any  $(G,\delta)$-map $\Phi:X\rightarrow Y$ we have that $\Phi$ is also a $(\widetilde{G},\delta)$-map. Hence, if $(X,\alpha_X,\beta_X)$ and $(Y,\alpha_Y,\beta_Y)$ are $(G,\delta)$-interleaved, then they are also $(\widetilde{G},\delta)$-interleaved. Finally, we obtain that
    $$\dht^G\big((X,\alpha_X,\beta_X),(Y,\alpha_Y,\beta_Y)\big)\geq\dht^{\widetilde{G}}\big((X,h^\sharp\alpha_X,\beta_X),(Y,h^\sharp\alpha_Y,\beta_Y)\big).$$
\end{enumerate}
\end{remark}

\subsubsection{The proof of \Cref{prop:dHTG}}
\begin{lemma}\label{lemma:Gdeltacomp}
If $\Phi:X\toG Y$ is a $(G,\delta)$-map and $\Phi':Y\toG Z$ is a $(G,\delta')$-map, then $\Phi'\circ\Phi$ is a $(G,\delta+\delta')$-map.
\end{lemma}
\begin{proof}
$\Phi'\circ\Phi$ is obviously a $G$-map. Also,
\begin{align*}
    \beta_Z(\Phi'(\Phi(x)))\leq\beta_Y(\Phi(x))+\delta'\leq\beta_X(x)+\delta+\delta'
\end{align*}
for all $x\in X$.
\end{proof}

\begin{lemma}\label{lemma:Gdeltahtpycomp}
If $\Phi_0,\Phi_1:X\rightarrow Y$ are $(G,\delta)$-homotopic and $\Psi:Y\rightarrow Z$ (resp. $\Psi:Z\rightarrow X$) is a $(G,\delta')$-map, then $\Psi\circ\Phi_0$ (resp. $\Phi_0\circ\Psi$) and $\Psi\circ\Phi_1$ (resp. $\Phi_1\circ\Psi$) are $(G,\delta+\delta')$-homotopic.
\end{lemma}
\begin{proof}
First, note that $\Psi\circ\Phi_0$ and $\Psi\circ\Phi_1$ are $(G,\delta+\delta')$-maps by Lemma \ref{lemma:Gdeltacomp}. By  assumption, there is a map $H:X\times [0,1]\rightarrow Y$ such that:
\begin{enumerate}
    \item $H(\cdot,0)=\Phi_0$,
    
    \item $H(\cdot,1)=\Phi_1$, and
    
    \item $H(\cdot,t)$ is $(G,\delta)$-map for all $t\in [0,1]$.
\end{enumerate}

Then, it is easy to verify that $\Psi\circ H:X\times [0,1]\rightarrow Z$ is a $(G,\delta+\delta')$ homotopy between $\Psi\circ\Phi_0$ and $\Psi\circ\Phi_1$.
\end{proof}

\begin{proof}[\textbf{Proof of \Cref{prop:dHTG}}]
It is easy to verify that $\dht^G(X,X)=0$ and $\dht^G(X,Y)=\dht^G(Y,X)$ from the definition. Hence, we only need to show the triangle inequality. Consider three triples $(X,\alpha_X,\beta_X)$, $(Y,\alpha_Y,\beta_Y)$, and $(Z,\alpha_Z,\beta_Z)$. We need to show the following:
$$\dht^G(X,Z)\leq\dht^G(X,Y)+\dht^G(Y,Z).$$
Fix arbitrary $\varepsilon>\dht^G(X,Y)$ and $\varepsilon'>\dht^G(Y,Z)$. Then, there are $(G,\varepsilon)$-maps $\Phi:X\rightarrow Y$, $\Psi:Y\rightarrow X$ and $(G,\varepsilon')$-maps $\Phi':Y\rightarrow Z$, $\Psi':Z\rightarrow Y$. By Lemma \ref{lemma:Gdeltacomp}, $\Phi'\circ\Phi:X\rightarrow Z$ and $\Psi\circ\Psi':Z\rightarrow X$ are $(G,\varepsilon+\varepsilon')$-maps. Also, since $\Psi'\circ\Phi'$ is $(G,2\varepsilon')$-homotopic to $\mathrm{id}_Y$ and $\Psi\circ\Phi$ is $(G,2\varepsilon)$-homotopic to $\mathrm{id}_X$, one can deduce that
$$\Psi\circ\Psi'\circ\Phi'\circ\Phi\stackrel{(G,2(\varepsilon+\varepsilon'))}{\simeq}\Psi\circ\Phi\stackrel{(G,2\varepsilon)}{\simeq}\mathrm{id}_X$$
where the first equivalence is by Lemma \ref{lemma:Gdeltahtpycomp}. Similarly,$\Phi'\circ\Phi\circ\Psi\circ\Psi'$ is also $(G,2(\varepsilon+\varepsilon'))$-homotopic to $\mathrm{id}_Z$. Therefore. $X$ and $Z$ are $(G,\varepsilon+\varepsilon')$-homotopy equivalent so that $\dht^G(X,Z)\leq\varepsilon+\varepsilon'$. Since the choice of $\varepsilon$ and $\varepsilon'$ are arbitrary, one can establish that
$$\dht^G(X,Z)\leq\dht^G(X,Y)+\dht^G(Y,Z)$$
as required.
\end{proof}

\subsection{The $G$-interleaving distance.}\label{sec:sec:gdi}

We first recall the concept of persistence module and associated notions; see \cite{carlsson2010zigzag,bubenik2014categorification}.

\begin{definition}[Persistence module]
A \emph{persistence module} $V_\bullet=(V_r,v_{r,s})_{r\leq s\in \R}$  is a family of $\mathbb{F}$-vector spaces $V_r$ for some given field $\mathbb{F}$ with linear maps $v_{r,s}:V_r \to V_s$ for each $r \leq s$ such that 
\begin{itemize}
\item $v_{r,r}=\mathrm{id}_{V_r}$,
\item $v_{s,t}\circ v_{r,s}=v_{r,t}$ for each $r\leq s \leq t$.
\end{itemize}
In other words, a persistence module is a functor from the poset $(\R,\leq)$ to the category of vector spaces. The maps $v_{\bullet,\bullet}$ are referred to as the \emph{structure maps} of $V_\bullet$.
\end{definition}

By $0_\bullet$ we will denote the zero persistence module.

\begin{definition}[Persistence morphism, isomorphism, and automorphism]
Suppose two persistence modules $V_\bullet=(V_r,v_{r,s})_{r\leq s}$ and $W_\bullet=(W_r,w_{r,s})_{r\leq s}$ are given. Then, 
\begin{itemize}
    \item A family of linear transformations $T_\bullet=(T_r:V_r\rightarrow W_r)_{r}$ is said to be a \emph{persistence morphism} from $V_\bullet$ to $W_\bullet$ if they satisfy $w_{r,s}\circ T_r=T_s\circ v_{r,s}$ for all $r\leq s$. 
    
    \item We define composition of persistence morphisms in the obvious way. 
    
    \item We say that a persistence morphism $T_\bullet:V_\bullet\rightarrow W_\bullet$ as a \emph{persistence isomorphism} if there exists an inverse persistence morphism $T'_\bullet:W_\bullet\rightarrow V_\bullet$ such that the composition $T'_\bullet\circ T_\bullet$ is equal to the identity persistence morphism on $V_\bullet$ and the composition $T_\bullet\circ T'_\bullet$ is equal to the identity persistence morphism on $W_\bullet$. 
    
    \item If there is a persistence isomorphism between $V_\bullet$ and $W_\bullet$, then we say that the two persistence modules $V_\bullet$ and $W_\bullet$ are \emph{isomorphic}.
    
    \item A persistence isomorphism between $V_\bullet$ and $V_\bullet$ itself is called a \emph{persistence automorphism} on $V_\bullet$. The \emph{group} of persistence automorphisms on $V_\bullet$ is denoted by $\gl(V_\bullet)$.
    
    \item Given a persistence module $V_\bullet=(V_r,v_{r,s})_{r\leq s}$ and $\delta\in \R$, by $V_{\bullet+\delta}$ we will denote the \emph{$\delta$-shifted} persistence module $(V_{r+\delta},v_{r+\delta,s+\delta})_{r\leq s}$.\footnote{In other words, if $V'_\bullet = V_{\bullet+\delta}$ then, $V'_r = V_{r+\delta}$ for all $r$ and $v'_{r,s}:=v_{r+\delta,s+\delta}$ for all $r\leq s$.}
    \end{itemize}
\end{definition}

For an integer $k\geq 0$, applying the degree $k$ homology functor (with coefficients in a given field $\mathbb{F}$) to a filtration $U_\bullet = \big(U_r,\iota_{r,s} \big)_{r\leq s}$ produces the persistence module $\mathrm{H}_k(U_\bullet;\mathbb{F})$ where the structure maps are those induced by the inclusions $\big(\iota_{r,s}\big)_{r\leq s}$. Following the existing literature, we will use the term \emph{persistent homology} of a  filtration to refer to the persistence module obtained  upon applying the  homology functor to the given filtration. 
\begin{center}\textbf{From now on, we will fix a field $\mathbb{F}$ and omit it from our notation.}
\end{center}

\begin{example}\label{ex:PMod-sub-level} 
 The sub-level set filtration from \Cref{def:sub-level} gives rise to the persistence module $\h_k(\beta_X\big((-\infty,\bullet)\big)$.
 \end{example}

\begin{example}\label{ex:PMod-VRm}
For any $p\in[1,\infty]$, the $p$-Vietoris-Rips metric thickening filtration $\vrm_p(X;\bullet)$ from \Cref{def:VRm} gives rise to the persistence module $\h_k(\vrm_p(X;\bullet))$.
\end{example}

\begin{example}\label{ex:PMod-VR}
The Vietoris-Rips filtration from 
\Cref{def:VR} gives rise to the persistence module $\h_k(|\vr(X;\bullet)|)$.
\end{example}

\paragraph{Introducing $G$-group actions.} $G$-equivariant formulations of persistence modules and the interleaving distance have been already defined and studied in \cite[Section 4.4]{polterovich2020topological}. Recall that a representation of a group $G$ is a pair $(V,\nu)$ where $V$ is a vector space and $\nu$ is a homomorphism from $G$ to $\gl(V)$. 

\begin{definition}[{\cite[Definition 4.4.1]{polterovich2020topological}}]
A \emph{$G$-persistence module} is a pair $(V_\bullet,\nu_\bullet)$ where $V_\bullet=(V_r,v_{r,s})_{r\leq s}$ is a persistence module and  $\nu_\bullet$  is  a homomorphism from $G$ to $\gl(V_\bullet)$. We will refer to $\nu_\bullet$ as a \emph{persistent representation of $G$ (on $V_\bullet$).}\nomenclature[13]{$(V_\bullet,\nu_\bullet)$}{$G$-persistence module}
\end{definition}

\begin{remark}
Note that:
\begin{enumerate}
    \item The defining condition of $\nu_\bullet$ means that for all $r\leq s$ and all $g\in G$ the following diagram commutes:
$$\begin{tikzcd}
V_r\ar[r,"v_{r,s}"]\ar[loop, ->, in=170,out=100,looseness=5,"\nu_r(g)"']&[10pt] V_s \ar[loop, ->, in=10,out=80,looseness=5,"\nu_s(g)"] 
\end{tikzcd}$$

\item When $G=G_0$, $G$-persistence modules boil down to  (standard) persistence modules.
\end{enumerate}
\end{remark}

\begin{example}
Suppose a filtered $G$-topological space $(Z,\alpha_Z,\beta_Z)$ is given.  Then, the persistence module from \Cref{ex:PMod-sub-level} can be promoted to a $G$-persistence module through the following persistent representation of $G$: $$\nu_\bullet:G\to \gl\left(\h_k\big(\beta_Z^{-1}(-\infty,\bullet)\big) \right)$$ such that for each $g\in G$ and $r\in \R$, 

$$\nu_r(g):=\h_k\Big({\beta_Z^{-1}\big((-\infty,r)\big)}\stackrel{\alpha_Z(g,\cdot)}{\xrightarrow{\hspace{1cm}}} {\beta_Z^{-1}\big((-\infty,r)\big)}\Big).$$

Note that this is well-defined since $\beta_Z$ is $G$-invariant. 

\begin{quote}
\textbf{For conciseness, whenever clear from context, we will omit specifying the associated persistence representation of the module $\h_k\big(\beta_Z^{-1}(-\infty,\bullet)\big)$.}
\end{quote}
\end{example}

In particular, when $X$ is a $G$-metric space, the filtered $G$-topological spaces $$\text{$(\Pfin(X),\alpha_{\Pfin(X)},\diam_p^X)$ and $(\vert\vr(X;\infty)\vert,\alpha_{\vert\vr(X;\infty)\vert},\diam)$}$$ (see \Cref{prop:PfinGaction} and \Cref{ex:GVRfilteredspace})  induce $G$-persistence modules. 

\begin{example}\label{ex:GvrmPH}
Assume that a $G$-metric space $(X,d_X,\alpha_X)$ is given. Then, for any $p\in [1,\infty]$, the persistence module from \Cref{ex:PMod-VRm} can be promoted to a $G$-persistence module via the following persistence representation of $G$:
$$\nu_\bullet:G\to \gl\big(\h_k(\vrm_p(X;\bullet))\big)$$
such that for each $g\in G$ and $r\in \R$, 

$$\nu_r(g):=\h_k\Big(\vrm_p(X;r) \stackrel{\alpha_{\Pfin(X)}(g,\cdot)}{\xrightarrow{\hspace{1cm}}} \vrm_p(X;r)\Big).$$

See \Cref{prop:PfinGaction} for the precise definition of $\alpha_{\Pfin(X)}$.
\end{example}

\begin{example}\label{ex:GvrPH}
Assume that a $G$-metric space $(X,d_X,\alpha_X)$ is given. Then, the persistence module from \Cref{ex:PMod-VR} can be promoted to a $G$-persistence module via the following persistence representation of $G$:
$$\nu_\bullet:G\to \gl\left(\h_k\big(|\vr(X;\bullet)|\big) \right)$$
such that for each $g\in G$ and $r\in \R$, 

$$\nu_r(g):=\h_k\Big(|\vr(X;r)| \stackrel{\alpha_{\vert\vr(X;\infty)\vert}(g,\cdot)}{\xrightarrow{\hspace{1cm}}} |\vr(X;r)|\Big).$$

See \Cref{ex:GVRfilteredspace} for the precise definition of $\alpha_{\vert\vr(X;\infty)\vert}$.
\end{example}

\begin{example}\label{ex:spheres-Z2-vrvrmPH}
Recall  items (2) and (3) of \Cref{ex:spheres-Z2}. By \Cref{ex:GvrmPH,ex:GvrPH}, we have that $\h_k(\vrm_p(X;\bullet))$ and $\h_k(\vert\vr(X;\bullet)\vert)$ are $\Z_2$-persistence modules for $X=\Sp^n,\SpE^n,\Sp^n_\infty,$ and $\square^n_\infty$ with the canonical $\Z_2$-actions.
\end{example}

\begin{definition}[$G$-persistence morphism]
Let $(V_\bullet,\nu_\bullet)$ and $(W_\bullet,\omega_\bullet)$ be two $G$-persistence modules. We say that a persistence morphism $T_\bullet:V_\bullet\rightarrow W_\bullet$ is a \emph{$G$-persistence morphism} if the following diagram commutes for all $r\leq s\in \R$ and all $g\in G$:

$$\begin{tikzcd}
V_r\ar[r,"v_{r,s}"]\ar[loop, ->, in=170,out=100,looseness=5,"\nu_r(g)"']\ar[dd,"T_r"']  &[10pt] V_s \ar[loop, ->, in=10,out=80,looseness=5,"\nu_s(g)"]\ar[dd,"T_s"] \\
\\
W_r\ar[r,"w_{r,s}"]\ar[loop, ->, in=-170,out=-100,looseness=5,"\omega_r(g)"] &[10pt] W_s \ar[loop, ->, in=-10,out=-80,looseness=5,"\omega_s(g)"'] 
\end{tikzcd}$$
\end{definition}

\begin{definition}[{$(G,\delta)$-interleavings and  $G$-interleaving distance \cite[Definition 4.4.7]{polterovich2020topological}}]\label{def:gint-dist}
Two $G$-persistence modules $(V_\bullet,\nu_\bullet)$ and $(W_\bullet,\omega_\bullet)$ are said to be \emph{$(G,\delta)$-interleaved} for some $\delta\geq 0$ if there are $G$-persistence morphisms  $T_\bullet:V_\bullet \to W_{\bullet+\delta}$ and $T'_\bullet: W_\bullet \to V_{\bullet+\delta}$ such that the following two diagrams commute for every $r\in \R$:

$$\begin{tikzcd}
V_r\ar[rr,"v_{r,r+2\delta}"]\ar[ddr,"T_r"']  &&  V_{r+2\delta}  \\
\\
&W_{r+\delta}\ar[uur,"T_{r+\delta}'"'] &  
\end{tikzcd}
\begin{tikzcd}
&V_{r+\delta} \ar[ddr,"T_{r+\delta}"]&
\\ 
\\
W_r\ar[rr,"w_{r,r+2\delta}"]\ar[uur,"T_r'"]  &&  W_{r+2\delta}
\end{tikzcd}
$$

The pair $(T_\bullet,T_\bullet')$ is said to be a \emph{$(G,\delta)$-interleaving} between $(V_\bullet,\nu_\bullet)$ and  $(W_\bullet,\nu_\omega)$. 
The \emph{$G$-interleaving distance} between $(V_\bullet,\nu_\bullet)$ and $(W_\bullet,\omega_\bullet)$ is defined as $$\di^G\big((V_\bullet,\nu_\bullet),(W_\bullet,\omega_\bullet)\big):=\inf\{\delta\geq 0:\,\mbox{$(V_\bullet,\nu_\bullet)$ and $(W_\bullet,\omega_\bullet)$ are $(G,\delta)$-interleaved}\}.$$
\nomenclature[14]{$\di^G$}{$G$-interleaving distance}
\end{definition}

The following proposition is a generalization of \cite[Proposition 5.3]{chazal2016structure} to the $G$-equivariant setting.\footnote{Since \cite{polterovich2020topological} does not state or give a proof of this fact, we provide one in \Cref{app:proofs-GHG}.}

\begin{proposition}\label{prop:diGpsdmtr}
$\di^G$ is an extended pseudometric on the collection of $G$-persistence modules.
\end{proposition}

\begin{remark}\label{rmk:diGmonotone}
Note that:
\begin{enumerate}
    \item when $G=G_0$, $\di^G$ boils down to the standard interleaving distance from \cite[Definition 4.2]{chazal2009proximity}.

    \item  Assume that a group homomorphism $h:\widetilde{G}\to G$ is given. Then, any $G$-persistence module $(V_\bullet,\nu_\bullet)$ induces the $\widetilde{G}$-persistence module $(V_\bullet,h^\sharp\nu_\bullet)$ where $h^\sharp\nu_\bullet:=\nu_\bullet\circ h$. Also, any $G$-persistence morphism $T_\bullet:(V_\bullet,\nu_\bullet)\rightarrow (W_\bullet,\omega_\bullet)$ induces the $\widetilde{G}$-persistence morphism $T_\bullet:(V_\bullet,h^\sharp\nu_\bullet)\rightarrow (W_\bullet,h^\sharp\omega_\bullet)$. Hence, if $(V_\bullet,\nu_\bullet)$ and $(W_\bullet,\omega_\bullet)$ are $(G,\delta)$-interleaved, then they are also $(\widetilde{G},\delta)$-interleaved and in particular 
    $$\di^G\big((V_\bullet,\nu_\bullet),(W_\bullet,\omega_\bullet)\big)\geq\di^{\widetilde{G}}\big((V_\bullet,h^\sharp\nu_\bullet),(W_\bullet,h^\sharp\omega_\bullet)\big).$$
    
\end{enumerate}
\end{remark}


\section{Hierarchy of  $G$-equivariant distances for $\vrm_p$.}

In this section we study the hierarchy of  $G$-equivariant distances for $p$-Vietoris-Rips metric thickening filtrations. 

\subsection{From metric spaces to Vietoris-Rips thickenings.}

The following theorem is a generalization of the rightmost inequality in \cite[Theorem A]{adams2024persistent} to the $G$-equivariant setting.

\begin{theorem}\label{Gdhtstability}
Let $G$ be a finite group. For any two compact $G$-metric spaces $X,Y$ and $p\in[1,\infty]$, we have that
$$2\,\dgh^G(X,Y)\geq\dht^G\big((\Pfin(X),\diam_p^X),(\Pfin(Y),\diam_p^Y)\big).$$
\end{theorem}

Example \ref{ex:dghdhtstrict} below demonstrates that the inequality in Theorem \ref{Gdhtstability} can be strict. 
\begin{definition}Let $X$ be a $G$-metric space and $A \subseteq X$ be a $G$-invariant subset. Motivated by \cite[Section 2]{hausmann1994vietoris}, a \emph{$G$-crushing} from $X$ to $A$ is defined as a $G$-equivariant, distance-nonincreasing, strong deformation retraction from $X$ to $A$. That is, a continuous map $H:X\times [0,1]\rightarrow X$ satisfying:
\begin{enumerate}[label=(\roman*)]
    \item $H(x,0)=x$, $H(x,1)\in A$, and $H(a,t)=a$ for all $x\in X$, $a\in A$, and $t\in[0,1]$,
    \item $d_X(H(x,t),H(x',t))\geq d_X(H(x,t'),H(x',t'))$ for all $x,x'\in X$ and $t\leq t'$,
    \item $H(\cdot,t):X\toG X$ is a $G$-map for all $t\in [0,1]$.
\end{enumerate}
\end{definition}
The following lemma is a generalization of \cite[Lemma B.1]{adamaszek2018metric} and can be verified with a similar proof.

\begin{lemma}\label{lemma:Gcrushinghmtp}
Let $X$ be a $G$-metric space and $A \subseteq X$ be a $G$-invariant subset. If there is a $G$-crushing from $X$ to $A$, then the canonical inclusion $\iota_r:\vrm_\infty(A;r)\hooktoG\vrm_\infty(X;r)$ is a $G$-homotopy equivalence for every $r>0$.
\end{lemma}

As a corollary of Lemma \ref{lemma:Gcrushinghmtp}, we have the following result. Note that in the trivial group case $G=G_0$, this result can also be obtained as a special case of \cite[Theorem G]{adams2024persistent}.

\begin{corollary}\label{cor:Gcrushingdht}
Let $X$ be a $G$-metric space and $A \subseteq X$ be a $G$-invariant subset. If there is a $G$-crushing from $X$ to $A$, then $\dht^G\big((\Pfin(X),\diam_\infty^X),(\Pfin(Y),\diam_\infty^Y)\big)=0$.
\end{corollary}

\begin{example}\label{ex:dghdhtstrict}
Let $X$ be a $G$-metric space and $A\subset X$ be a $G$-invariant subset satisfying $\dgh(A,G)>0$. Also, assume that there is a $G$-crushing from $X$ to $A$. Then, by Corollary \ref{cor:Gcrushingdht}, one can conclude that
$$2\,\dgh^G(X,Y)\geq 2\,\dgh(X,Y) > 0=\dht^G\big((\Pfin(X),\diam_\infty^X),(\Pfin(Y),\diam_\infty^Y)\big).$$
For example, one may choose the geodesic $\Z_3$-metric space $X$ shown in Figure \ref{fig:compatibility} and let $A:=\{x_0\}$ be its center point.
\end{example}

The following lemma is a generalization of \cite[Lemmas 4.6 and 5.15]{adams2024persistent} to the $G$-equivariant setting.

\begin{lemma}\label{lemma:G2deltamapexst}
Let $G$ be a finite group, $X$ be a $G$-metric space, $\delta>0$, $U$ be a finite $G$-invariant $\delta$-net of $X$, and $p\in [1,\infty]$. Then, there exists a $(G,2\delta)$-map $\Phi:(\Pfin(X),\diam_p^X)\toG (\Pfin(U),\diam_p^U)$.
\end{lemma}
\begin{proof}
Note that $\mathcal{U}:=\{B_\delta(u)\}_{u\in U}$ is a finite $G$-invariant open covering of $X$. Then, by \Cref{lemma:Gptofuty}, there is a $G$-partition of unity $\{\varphi_u:X\rightarrow [0,1]\}_{u\in U}$ subordinate to $\mathcal{U}$. In particular, $\varphi_{\alpha_X(g,u)}(\alpha_X(g,x))=\varphi_{u}(x)$ for all $g\in G$, $u\in U$, and $x\in X$. Now, consider the following map
\begin{align*}
    \Phi:\Pfin(X)&\rightarrow\Pfin(U)\\
    \mu&\longmapsto\sum_{u\in U}\left(\int_X\varphi_{u}(x)\,d\mu(x)\right)\delta_u.
\end{align*}

Then, we have that

\begin{align*}
\Phi(\alpha_{\Pfin(X)}(g,\mu))&=\Phi(\alpha_X(g,\cdot)_\sharp\mu)=\sum_{u\in U}\left(\int_X\varphi_{u}(x)\,d\alpha_X(g,\cdot)_\sharp\mu(x)\right)\delta_u\\&=\sum_{u\in U}\left(\int_X\varphi_{u}(\alpha_X(g,x))\,d\mu(x)\right)\delta_u
=\sum_{u\in U}\left(\int_X\varphi_{\alpha_X(g^{-1},u)}(x)\,d\mu(x)\right)\delta_u\\
&=\sum_{u\in U}\left(\int_X\varphi_{u}(x)\,d\mu(x)\right)\delta_{\alpha_X(g,u)}=\alpha_{\Pfin(U)}(g,\Phi(\mu))
\end{align*}
for any $g\in G$ and $\mu\in\Pfin(X)$. Also, $\Phi$ is continuous since $U$ is finite.  Hence, $\Phi$ is a $G$-map. Finally, we have that $\diam_p^U(\Phi(\mu))\leq\diam_p^X(\mu)+2\delta$ for any $\mu\in\Pfin(X)$ by \cite[Definition 4.3, Lemma 4.6, and Lemma 5.15]{adams2024persistent}. This completes the proof.
\end{proof}

\begin{lemma}[{\cite[Propositions 6.1 and 6.2]{adamaszek2018metric}}] \label{lemma:homeo}
    The canonical identity map \begin{align*}\mathrm{id}:\vert\vr(X;\eta)\vert&\to\vrm_\infty(X;\eta)  \\\sum_{i=0}^n u_ix_i&\mapsto\sum_{i=0}^n u_i\delta_{x_i}
    \end{align*}
    is continuous. Furthermore, if $X$ is a finite metric space, then $\mathrm{id}$ is a homeomorphism.
\end{lemma}

\begin{lemma}\label{lemma:Ginvdeltanetexst}
Let $G$ be a finite group and $X$ be a totally bounded $G$-metric space. Then, for any $\delta>0$, there exists a finite $G$-invariant $\delta$-net $U$ of $X$.\footnote{I.e., $\alpha_X(g,x)\in U$ for every $g\in G$ and $u\in U$.} 
\end{lemma}
\begin{proof}
Since $X$ is totally bounded, there exists a finite $\delta$-net $\widetilde{U}\subseteq X$. Then, one can choose $U:=\{\alpha_X(g,x):g\in G\text{ and }x\in U\}$.
\end{proof}

\begin{proof}[\textbf{Proof of  \Cref{Gdhtstability}}]
Our proof follows that of \cite[Theorem A]{adams2024persistent} with the necessary adjustments to inject $G$-equivariance. Fix an arbitrary $\delta>0$. Then, by \Cref{lemma:Ginvdeltanetexst}, one can choose $G$-invariant $\delta$-nets $U$ of $X$ and $V$ of $Y$. It is easy to verify that $\dgh^G(X,U)\leq\delta$ and $\dgh^G(Y,V)\leq\delta$. Also, by \Cref{lemma:G2deltamapexst}, there are $(G,2\delta)$-maps $\Phi:(\Pfin(X),\diam_p^X)\toG (\Pfin(U),\diam_p^U)$ and $\Psi:(\Pfin(Y),\diam_p^Y)\toG (\Pfin(V),\diam_p^V)$.

Now, fix an arbitrary $\eta>0$ such that $2\,\dgh^G(X,Y) < \eta$. Then, by the triangle inequality, we have that $\dgh^G(U,V)<\frac{\eta}{2}+2\delta$.

Then, there exists $G$-functions $\varphi:U\toG V$ and $\psi:V\toG U$ such that
$$\max\{\dis(\varphi),\dis(\psi),\codis(\varphi,\psi)\}<\eta+4\delta.$$
It is easy to verify that $$\varphi_\sharp:(\Pfin(U),\diam_p^U)\toG(\Pfin(V),\diam_p^V) \,\,\mbox{and} \,\,\psi_\sharp:(\Pfin(V),\diam_p^U)\toG(\Pfin(U),\diam_p^V)$$ are $(G,\eta+4\delta)$-maps. Finally, we define
\begin{align*}
    \widehat{\Phi}&:=\iota_V\circ\varphi_\sharp\circ\Phi,\text{ and}\\
    \widehat{\Psi}&:=\iota_U\circ\psi_\sharp\circ\Psi
\end{align*}
where $\iota_U:\Pfin(U)\hooktoG\Pfin(X)$ and $\iota_V:\Pfin(V)\hooktoG\Pfin(Y)$ are the canonical inclusions. Note that $\widehat{\Phi}$ and $\widehat{\Psi}$ are obviously $G$-maps. Finally, for any $\mu\in\Pfin(X)$, we have that
$$\diam_p^Y\big(\widehat{\Phi}(\mu)\big)=\diam_p^V\big(\varphi_\sharp\circ\Phi(\mu)\big)\leq (\eta+4\delta)+\diam_p^U\big(\Phi(\mu)\big)\leq(\eta+6\delta)+\diam_p^X(\mu).$$
Hence, $\widehat{\Phi}:\Pfin(X)\toG\Pfin(Y)$ is a $(G,\eta+6\delta)$-map. Similarly, $\widehat{\Psi}:\Pfin(Y)\toG\Pfin(X)$ is also a $(G,\eta+6\delta)$-map. Hence, $\dht^G\big((\Pfin(X),\diam_p^X),(\Pfin(Y),\diam_p^Y)\big)\leq\eta+6\delta$. Since $\eta$ and $\delta$ are arbitrary, finally one can conclude that 
$$2\,\dgh^G(X,Y)\geq\dht^G\big((\Pfin(X),\diam_p^X),(\Pfin(Y),\diam_p^Y)\big).$$
\end{proof}

\subsection{From Vietoris-Rips thickenings to persistence modules.}

The following proposition is a generalization of the first inequality in \cite[Theorem A]{adams2024persistent} to the $G$-equivariant setting.

\begin{proposition}\label{prop:GdHTGdI}
For any two filtered $G$-topological spaces $(X,\alpha_X,\beta_X)$ and $(Y,\alpha_Y,\beta_Y)$ and for all integers $k\geq 0$, we have that
$$\dht^G(X,Y)\geq\di^G\big(\h_k(\beta_X^{-1}(-\infty,\bullet)),\h_k(\beta_Y^{-1}(-\infty,\bullet))\big).$$
\end{proposition}
\begin{proof}
Choose an arbitrary $\varepsilon>\dht^G(X,Y)$. Then, there are $(G,\varepsilon)$-maps $\Phi:X\rightarrow Y$ and $\Psi:Y\rightarrow X$ such that $\Psi\circ\Phi$ is $(G,2\varepsilon)$-homotopic to $\mathrm{id}_X$ and $\Phi\circ\Psi$ is $(G,2\varepsilon)$-homotopic to $\mathrm{id}_Y$. Note that $\Phi_r:=\Phi\vert_{{\beta_X^{-1}(-\infty,r)}}$ sends any point in $\beta_X^{-1}(-\infty,r)$ to $\beta_Y^{-1}(-\infty,r+\varepsilon)$ for all $r\in\R$. Hence, this induces a morphism $(\Phi_r)_\ast:\Hom_k(\beta_X^{-1}(-\infty,r))\rightarrow\Hom_k(\beta_Y^{-1}(-\infty,r+\varepsilon))$ and similarly $(\Psi_r)_\ast:\Hom_k(\beta_Y^{-1}(-\infty,r))\rightarrow\Hom_k(\beta_X^{-1}(-\infty,r+\varepsilon))$ for all $r\in\R$. Furthermore, for every $r\in\R$, it is easy to verify that $(\Phi_r)_\ast$ and $(\Psi_r)_\ast$ are $G$-equivariant and $(\Psi_{r+\varepsilon})_\ast\circ(\Phi_r)_\ast$ and $(\Phi_{r+\varepsilon})_\ast\circ(\Psi_r)_\ast$ are equal to the structure maps respectively because of the choice of $\Phi$ and $\Psi$.

Therefore, $\varepsilon\geq\di^G(\h_k(\beta_X^{-1}(-\infty,\bullet)),\h_k(\beta_Y^{-1}(-\infty,\bullet)))$. Since the choice of $\varepsilon$ is arbitrary, we have
$$\dht^G(X,Y)\geq\di^G(\h_k(\beta_X^{-1}(-\infty,\bullet)),\h_k(\beta_Y^{-1}(-\infty,\bullet)))$$
as required.
\end{proof}

\begin{remark} \Cref{prop:GdHTGdI} generalizes \cite[Lemma 3.1]{frosini2017persistent}. The latter can be obtained from the former with the choice $G=G_0$.
\end{remark}

We now consider $\h_k(\vrm_p(X;\bullet))$, the $G$-persistence module corresponding to a $G$-metric space $X$ via the $p$-Vietoris-Rips metric thickening filtration; see \Cref{def:VRm}. From \Cref{prop:GdHTGdI} we obtain the following generalization of \cite[Theorem B]{adams2024persistent} to the $G$-equivariant setting:

\begin{corollary}\label{coro:dHTG-dIG-Katetov}
    For any two $G$-metric spaces $X,Y$, $p\in[1,\infty]$, and integer $k\geq 0 $ we have that
    $$\dht^G\big((\Pfin(X),\diam_p^X),(\Pfin(Y),\diam_p^Y)\big)\geq \di^G\big(\h_k(\vrm_p(X;\bullet)),\h_k(\vrm_p(Y;\bullet))\big).$$
\end{corollary}

In \cite[Corollary 5.10]{adams2024persistent}, the authors showed that $\di\big(\h_k(\vrm_\infty(X;\bullet)),\h_k(\vert\vr(X;\bullet)\vert)\big)=0$ for any totally bounded metric space $X$ and $k\geq 0$. The following proposition is a generalization of that result to the $G$-equivariant setting.

\begin{proposition}\label{prop:vrvrmdIGsame}
Let $G$ be a finite group. Then, for any totally bounded $G$-metric space $X$ and  $k\geq 0$, we have that
$$\di^G\big(\h_k(\vrm_\infty(X;\bullet)),\h_k(\vert\vr(X;\bullet)\vert)\big)=0.$$
\end{proposition}
\begin{proof}
Fix an arbitrary $\delta>0$. Then, by \Cref{lemma:Ginvdeltanetexst}, there exists a $G$-invariant $\delta$-net $U$ of $X$. Then, $\dgh^G(X,U)\leq\delta$. Hence, by \Cref{Gdhtstability} and \Cref{coro:dHTG-dIG-Katetov}, we have that $$\di^G\big(\h_k(\vrm_\infty(X;\bullet)),\h_k(\vrm_\infty(U;\bullet))\big)\leq 2\delta.$$ Furthermore, by \cite[Lemmas 4.11 and 4.12]{adams2024persistent}, we have that $\vrm_\infty(X;r)$ and $\vert\vr(X;r)\vert$ are $G$-homeomorphic for any $r>0$ via the canonical bijection $\sum_{i=0}^n u_i\delta_{x_i}\mapsto\sum_{i=0}^n u_ix_i$. Therefore, $\di^G\big(\h_k(\vrm_\infty(X;\bullet)),\h_k(\vert\vr(U;\bullet)\vert)\big)\leq 2\delta$. Hence, by the triangle inequality for $\di^G$, we have that
\begin{align*}
&\di^G\big(\h_k(\vrm_\infty(X;\bullet)),\h_k(\vert\vr(X;\bullet)\vert)\big)\\
&\leq\di^G\big(\h_k(\vrm_\infty(X;\bullet)),\h_k(\vert\vr(U;\bullet)\vert)\big)+\di^G\big(\h_k(\vr(X;\bullet)),\h_k(\vert\vr(U;\bullet)\vert)\big)\leq 4\delta.
\end{align*}
Since the choice of $\delta$ is arbitrary, one can conclude that
$$\di^G\big(\h_k(\vrm_\infty(X;\bullet)),\h_k(\vert\vr(X;\bullet)\vert)\big)=0.$$
\end{proof}

\begin{corollary}\label{cor:vrvrmdIGsame}
Let $G$ be a finite group. Then, for any two totally bounded $G$-metric spaces $X$ and $Y$, we have that
$$\di^G\big(\h_k(\vrm_\infty(X;\bullet)),\h_k(\vrm_\infty(Y;\bullet))\big)=\di^G\big(\h_k(\vert\vr(X;\bullet)\vert),\h_k(\vert\vr(Y;\bullet)\vert)\big).$$
\end{corollary}
\begin{proof}
Apply the triangle inequality of $\di^G$ and \Cref{prop:vrvrmdIGsame}.
\end{proof}

\subsubsection{An exact calculation for the case of spheres.} 

Fix arbitrary positive integers $n> m>0$. By \cite[Theorem 4.1 and Corollary 9.39]{lim2024vietoris}, we know that $$\di\big(\h_m(\vert\vr(\Sp^m;\bullet)\vert),\h_m(\vert\vr(\Sp^n;\bullet)\vert)\big)=\frac{\zeta_m}{2}.$$
Given the item (2) of \Cref{rmk:diGmonotone}, one  wonders whether the $\Z_2$-interleaving distance might be \emph{strictly better} than the standard one in the case at hand, i.e. whether
$$\di^{\Z_2}\big(\h_m(\vert\vr(\Sp^m;\bullet)\vert),\h_m(\vert\vr(\Sp^n;\bullet)\vert)\big)>\di\big(\h_m(\vert\vr(\Sp^m;\bullet)\vert),\h_m(\vert\vr(\Sp^n;\bullet)\vert)\big)=\frac{\zeta_m}{2}$$ holds when the persistence modules appearing in the LHS are endowed with the canonical $\Z_2$-persistent representation induced by the canonical $\Z_2$-action (i.e. the antipodal involution) on spheres with the geodesic metric; see the item (2) of \Cref{ex:spheres-Z2}. This is \emph{not} the case. Furthermore, the following proposition shows that  \emph{no group $G$ acting by isometries on the spheres has an impact on  the value of the resulting $G$-interleaving distance} $\di^{G}\big(\h_m(\vert\vr(\Sp^m;\bullet)\vert),\h_m(\vert\vr(\Sp^n;\bullet)\vert)\big)$.

\begin{proposition}\label{prop:Gdnmatter}
Let $n>m>0$ be positive integers and $G$ be a group. Consider any $G$-group actions on $\Sp^n$ and $\Sp^m$ by isometries. Then, we have that
$$\di^{G}\big(\h_m(\vert\vr(\Sp^m;\bullet)\vert),\h_m(\vert\vr(\Sp^n;\bullet)\vert)\big)=\di\big(\h_m(\vert\vr(\Sp^m;\bullet)\vert),\h_m(\vert\vr(\Sp^n;\bullet)\vert)\big)=\frac{\zeta_m}{2}.$$
\end{proposition}

The following example shows that there are $G$-spaces $X$ and $Y$ such that 
$$\di^{G}\big(\h_k(\vert\vr(X;\bullet)\vert),\h_k(\vert\vr(Y;\bullet)\vert)\big) > \di\big(\h_k(\vert\vr(X;\bullet)\vert),\h_k(\vert\vr(Y;\bullet)\vert)\big).$$

\begin{figure}
\begin{center}
    \includegraphics[width = 0.8\linewidth]{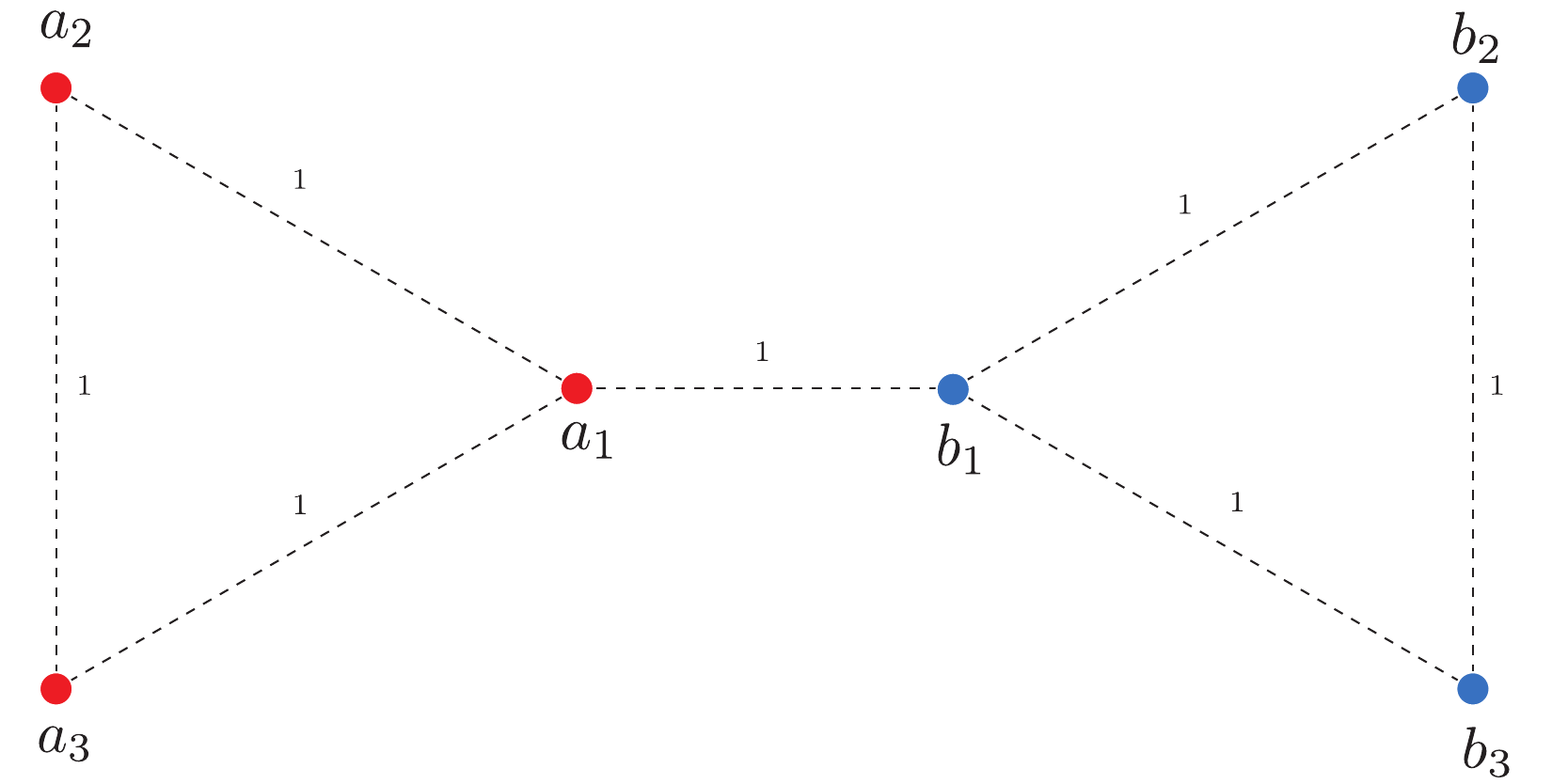}
\end{center}
    \caption{Description of $(X,d_X)$ from \Cref{ex:dIGdIstrictineq}.}
    \label{fig:dIGdIstrictexampleX}
\end{figure}

\begin{example}\label{ex:dIGdIstrictineq}
Let $(X,d_X)$ and $(Y,d_Y)$ be the following two six-points metric spaces: \begin{itemize}
\item $X:=\{a_1,a_2,a_3,b_1,b_2,b_3\}$ where $d_X$ is the shortest path distance on the graph having the points in $X$ as vertices and $\{(a_1,a_2),(a_2,a_3),(a_3,a_1),(b_1,b_2),(b_2,b_3),(b_3,b_1),(a_1,b_1)\}$ as the set of edges; see \Cref{fig:dIGdIstrictexampleX}. 
\item $Y:=\{u_1,u_2,u_3,u_4,v_1,v_2\}$ and $d_Y$ is the discrete metric: $d_Y(y,y'):=\begin{cases}1&\text{if }y\neq y'\\ 0&\text{if }y= y'.\end{cases}$
\end{itemize}

Let $V_\bullet:=\h_0(\vert\vr(X;\bullet)\vert)$ and $W_\bullet:=\h_0(\vert\vr(Y;\bullet)\vert)$ be the degree-$0$ Vietoris-Rips persistent homology of $X$ and $Y$ (see \Cref{ex:PMod-VR}). Then, it is easy to verify that, as standard persistence modules, $V_\bullet$  and $W_\bullet$ are isomorphic so that $\di(V_\bullet,W_\bullet)=0$.

Next, we will consider $\Z_2$-persistence representations for $V_\bullet$ and $W_\bullet$. We will view $\Z_2$ as $\{-1,1\}$ with  multiplication as the binary operation. Let $\alpha_X:\Z_2\times X\rightarrow X$ be s.t. $\alpha_X(-1,a_i):=b_i$ (equivalently, $\alpha_X(-1,b_i):=a_i$) for $i=1,2,3$. Then, $(X,d_X,\alpha_X)$ is a $\Z_2$-metric space. Furthermore, $\alpha_X$ naturally induces a persistent representation $\nu_\bullet:\Z_2\rightarrow\gl(V_\bullet)$ (see \Cref{ex:GvrPH}) so that $(V_\bullet,\nu_\bullet)$ is a $\Z_2$-persistence module. As for $W_\bullet$, we will first equip  $Y$ with a $\Z_4$-action. We will view $\Z_4$ as $\{\cx,-1,-\cx,1\}$, where $\cx=\sqrt{-1}$, with  multiplication as the binary group operation. Let $\alpha_Y:\Z_4\times Y\rightarrow Y$ be s.t. $\alpha_Y(\cx,u_1):=u_2$, $\alpha_Y(\cx,u_2):=u_3$, $\alpha_Y(\cx,u_3):=u_4$, $\alpha_Y(\cx,u_4):=u_1$, $\alpha_Y(\cx,v_1):=v_2$, and $\alpha_Y(\cx,v_2):=v_1$. Then, $(Y,d_Y,\alpha_Y)$ is a $\Z_4$-metric space. Also, $\alpha_Y$ naturally induces a  persistent representation $\omega_\bullet:\Z_4\rightarrow\gl(W_\bullet)$ so that $(W_\bullet,\omega_\bullet)$ is a $\Z_4$-persistence module. Furthermore, we define a $\Z_2$-persistence representation $\omega_\bullet^{\otimes 2}:\Z_2\rightarrow\gl(W_\bullet)$ s.t. $\omega_\bullet^{\otimes 2}(-1):=\omega_\bullet(-1)$. Then, $(W_\bullet,\omega_\bullet^{\otimes 2})$ is a $\Z_2$-persistence module.

By \cite[Theorem 4.4.11]{polterovich2020topological}, we have that
$$\di^{\Z_2}\big((V_\bullet,\nu_\bullet),(W_\bullet,\omega_\bullet^{\otimes 2})\big)\geq m_\mathrm{odd}(L^{V_\bullet})$$
where
\begin{itemize}
    \item $L^{V_\bullet}$ is the persistence module s.t. $(L^{V_\bullet})_r:=\mathrm{ker}\big(\nu_r(-1)-(-1)\cdot\mathrm{id}_{V_r}\big)$ (see \cite[Example 4.4.5 and pg.43]{polterovich2020topological}) for every $r\in\R$. i.e., $(L^{V_\bullet})_r$ is the eigenspace of $\nu_r(-1)$ associated to the eigenvalue $-1$;

    \item $m_\mathrm{odd}(L^{V_\bullet})$ is the maximum over ``multiplicity functions" $m_j(L^{V_\bullet})$ for all odd $j$ (see \cite[Definition 4.3.2]{polterovich2020topological}).
\end{itemize}
Then, it is easy to verify that
$$(L^{V_\bullet})_r=\begin{cases}\mathrm{span}(\{a_1-b_1,a_2-b_2,a_3-b_3\})&\text{if }r\in (0,1] \\ 0 &\text{otherwise}\end{cases}$$
so that $m_\mathrm{odd}(L^{V_\bullet})=m_3(L^{V_\bullet})=\frac{1}{4}$. In conclusion, we have that
$$\di^{\Z_2}\big((V_\bullet,\nu_\bullet),(W_\bullet,\omega_\bullet^{\otimes 2})\big)\geq\frac{1}{4}>0=\di(V_\bullet,W_\bullet).$$
\end{example}

\begin{proof}[\textbf{Proof of \Cref{prop:Gdnmatter}}]
First of all, observe that the map $$\h_k\big(\vert\vr(\Sp^m;r)\vert\big)\rightarrow\h_k\big(\vert\vr(\Sp^m;r+\delta)\vert\big)$$ induced by the inclusion is the zero map for any $\delta>\zeta_m$, $r\in\R$, and $k\in\Z_{\geq 0}$ by \cite[Remark 9.3 and Proof of Proposition 9.7]{lim2024vietoris}. Similarly, the map $\h_k\big(\vert\vr(\Sp^n;r)\vert\big)\rightarrow\h_k\big(\vert\vr(\Sp^n;r+\delta)\vert\big)$ induced by the inclusion is the zero map for any $\delta>\zeta_n$, $r\in\R$, and $k\in\Z_{\geq 0}$. Therefore, for an arbitrary $\delta>\max\{\zeta_m,\zeta_n\}=\zeta_m$, if we consider the following $G$-persistence morphisms

\begin{align*}
&T_\bullet = \Bigg(T_r: \h_m\big(\vert\vr(\Sp^m;r)\vert\big) \to \h_m\left(\left\vert\vr\left(\Sp^n;r+\frac{\delta}{2}\right)\right\vert\right)\Bigg)_{r\in\R}
\text{ and}\\
&T_\bullet' = \Bigg(T_r': \h_m\big(\vert\vr(\Sp^n;r)\vert\big) \to \h_m\left(\left\vert\vr\left(\Sp^m;r+\frac{\delta}{2}\right)\right\vert\right)\Bigg)_{r\in\R}
\end{align*}

where every $T_r$ and $T'_r$ is the zero map, one can easily verify that $\h_m(\vert\vr(\Sp^m;\bullet)\vert)$ and $\h_m(\vert\vr(\Sp^n;\bullet)\vert)$ are $(G,\delta/2)$-interleaved through $T_\bullet$ and $T_\bullet'$; see \Cref{def:gint-dist}. Hence, $$\frac{\delta}{2}\geq\di^{G}\big(\h_m(\vert\vr(\Sp^m;\bullet)\vert),\h_m(\vert\vr(\Sp^n;\bullet)\vert)\big).$$ Since $\delta>\zeta_m$ is arbitrary, we have that
$$\frac{\zeta_m}{2}\geq\di^{G}\big(\h_m(\vert\vr(\Sp^m;\bullet)\vert),\h_m(\vert\vr(\Sp^n;\bullet)\vert)\big).$$

Furthermore, by \cite[Theorem 4.1 and Corollary 9.39]{lim2024vietoris}, we know that
$$\di\big(\h_m(\vert\vr(\Sp^m;\bullet)\vert),\h_m(\vert\vr(\Sp^n;\bullet)\vert)\big)=\frac{\zeta_m}{2}.$$
Finally, by  item (2) of \Cref{rmk:diGmonotone} and choosing $\widetilde{G}=G_0$, one can conclude that 
$$\di^G\big(\h_m(\vert\vr(\Sp^m;\bullet)\vert),\h_m(\vert\vr(\Sp^n;\bullet)\vert)\big)=\di\big(\h_m(\vert\vr(\Sp^m;\bullet)\vert),\h_m(\vert\vr(\Sp^n;\bullet)\vert)\big)=\frac{\zeta_m}{2}.$$
\end{proof}

\section{$G$-equivariant precompactness, rigidity and finiteness.}\label{sec:finiteness}

In a series of works including \cite{weinstein1967homotopy,cheeger1970finiteness,grove1988bounding,yamaguchi1988homotopy,peter1990finiteness,grove1990geometric} it was shown that suitably constrained families of metric spaces admit only finitely many equivalence classes—whether up to homeomorphism, diffeomorphism, or homotopy equivalence. A number of these finiteness results typically combine precompactness in the Gromov–Hausdorff topology with appropriate rigidity theorems. In this section, building on results from earlier sections, we establish  finiteness results for the $G$-homotopy equivalence.

\medskip
Let $G$ be an arbitrary group. Recall that $\mathcal{M}^G$  denotes the collection of all compact $G$-metric spaces. 

\begin{definition}[The class $\mathcal{H}^G$]\label{def:HG}
Let $\mathcal{H}^G\subseteq\mathcal{M}^G$ be the subcollection of all compact $G$-metric spaces $(X,d_X,\alpha_X)$ that satisfy the following condition: there exists $r^G(X)>0$ such that: 

\begin{enumerate}
    \item For each $\varepsilon\in(0,r^G(X)]$, there exists a $G$-homotopy equivalence $p_\varepsilon^X:\vrm_\infty(X;\varepsilon)\toG X$ that admits as a $G$-homotopy inverse the canonical isometric embedding \begin{align*}j_\varepsilon^X:X&\hooktoG \vrm_\infty(X;\varepsilon)\\
    x&\longmapsto \delta_x.\end{align*}

    \item The following diagram commutes up to $G$-homotopy, for any $0<\varepsilon\leq\varepsilon'\leq r^G(X)$:
    $$\begin{tikzcd}
\vrm_\infty(X;\varepsilon) \arrow[r, hook, "\iota_{\varepsilon,\varepsilon'}^X"] \arrow[dr, "p_\varepsilon^X"]
& \vrm_\infty(X;\varepsilon') \arrow[d, "p_{\varepsilon'}^X"]\\ & X
\end{tikzcd}$$
where $\iota_{\varepsilon,\varepsilon'}^X:\vrm_\infty(X;\varepsilon)\hooktoG \vrm_\infty(X;\varepsilon')$ denotes the canonical inclusion.\footnote{I.e. the condition requires that $p_{\varepsilon'}^X\circ\iota_{\varepsilon,\varepsilon'}^X$ and $p_{\varepsilon}^X$ be $G$-homotopic.}
\end{enumerate}
\nomenclature[15]{$\mathcal{H}^G$}{The class $\mathcal{H}^G$ of compact $G$-metric spaces}
\end{definition}

\subsection{Complete $G$-Riemannian manifolds are in $\mathcal{H}^G$.}
In \cite[Theorem 4.2]{adamaszek2018metric}, the authors proved a Hausmann-type of theorem for Vietoris-Rips metric thickenings via the notion of \emph{Karcher mean} or \emph{Riemannian center of mass} \cite{karcher1977riemannian}. Using the fact that Karcher means are actually well behaved under isometries (which introduces equivariance)---as pointed out in \cite[Section 1.4.1]{karcher1977riemannian}---in \Cref{thm:vrmHausmann-G} below we obtain a generalization of  \cite[Theorem 4.2]{adamaszek2018metric} to the $G$-equivariant setting.

More precisely, the Riemannian center of mass map $p_\varepsilon^M$ is defined as follows: fix an arbitrary $\mu\in\vrm_\infty(M;\varepsilon)$. Then, $\supp[\mu]$ is a finite subset of $M$ contained in a ball of radius $\varepsilon\leq r(M)$ (see \Cref{thm:vrmHausmann-G} for the definition of $r(M)$). Hence, by \cite[Theorem 1.2 and Definition 1.3]{karcher1977riemannian}, the map $x\longmapsto\frac{1}{2}\int_{M}d_M(x,y)^2\,d\mu(y)$ has a unique minimizer, where $d_M$ is the geodesic metric. Then,  $p_\varepsilon^M(\mu)$ is defined as this unique minimizer.

Assume that a group $G$ and a complete Riemannian manifold $M$ are given. We say that $(M,\alpha_M)$ is a \emph{complete $G$-Riemannian manifold} if $\alpha_M:G\times M\longrightarrow M$ is a $G$-group action such that $(M,d_M,\alpha_M)$ becomes a $G$-metric space.\footnote{That is, $\alpha_M$ acts by isometries on $(M,d_M)$.}

\begin{theorem}[$G$-equivariant Hausmann theorem for $\vrm_\infty$]\label{thm:vrmHausmann-G} Let $G$ be a group and $M$ be a complete $G$-Riemannian manifold. Let  $$r(M):=\min\Big\{\mathrm{Conv}(M),\tfrac{1}{4}\pi\Delta_M^{-1/2}\Big\},$$ where $\mathrm{Conv}(M)$ is the convexity radius of $M$ and $\Delta_M>0$ is a uniform upper bound on the sectional curvature of $M$. Then, for each $\varepsilon\in(0,r(M)]$, the Riemannian center of mass (or Karcher mean) map $p_\varepsilon^M:\vrm_\infty(M;\varepsilon)\toG M$ is a $G$-homotopy equivalence that has the canonical isometric embedding $j_\varepsilon^M:M\hooktoG \vrm_\infty(M;\varepsilon)$ as a $G$-homotopy inverse.
\end{theorem}
Note that, thanks to \Cref{thm:vrmHausmann-G}, every complete $G$-Riemannian manifold $M$ satisfies the first condition in \Cref{def:HG} with $r^G(M) = r(M)$.  Additionally, the second condition in \Cref{def:HG} is also  satisfied since, obviously, $p_{\varepsilon'}^M\circ\iota_{\varepsilon,\varepsilon'}^M=p_{\varepsilon}^M$.  We immediately obtain:
\begin{corollary}
\label{prop:RMnice}
If $M$ is a complete $G$-Riemannian manifold, then $M$ belongs to $\mathcal{H}^G$.
\end{corollary}
\begin{proof}[\textbf{Proof of \Cref{thm:vrmHausmann-G}}]
As \cite[Theorem 4.2]{adamaszek2018metric} explains, for each $\varepsilon\in(0,r(M)]$, the Riemannian center of mass map $p_\varepsilon^M:\vrm_\infty(M;\varepsilon)\longrightarrow M$ is a (standard) homotopy equivalence that has the canonical isometric embedding $j_\varepsilon^M:M\longhookrightarrow \vrm_\infty(M;\varepsilon)$ as a homotopy inverse (which is certainly a $G$-map). Since $p_\varepsilon^M\circ j^M_\varepsilon = \mathrm{id}_M$ one only needs to verify that the identity on $\vrm_\infty(M;\varepsilon)$ and $j_\varepsilon^M\circ p^M_\varepsilon$ are homotopic. For this, the authors consider the homotopy $H:\vrm_\infty(M;\varepsilon)\times[0,1]\to \vrm_\infty(M;\varepsilon)$ given, for $\mu\in\vrm_\infty(M;\varepsilon)$,  by the \emph{straight-line} homotopy $$H(\mu,t) := (1-t)\mu + t\,j^M_\varepsilon\circ p^M_\varepsilon(\mu).$$
That $H$ is well defined and continuous is verified in \cite{adamaszek2018metric}. Given these facts, in order to obtain our claim that $p_\varepsilon^M:\vrm_\infty(M;\varepsilon)\to M$ is a $G$-homotopy equivalence with  $j^M_\varepsilon$ as a $G$-homotopy inverse,  it suffices to prove that:
\begin{itemize}
\item[(1)] $p^M_\varepsilon$ is $G$-equivariant.
\item[(2)] $t\mapsto H(\cdot,t)$ is a $G$-map for each $t\in[0,1].$
\end{itemize}

We first prove (1) which is equivalent to the condition that $\alpha_M(g,p_M^\varepsilon(\mu)) = p_\varepsilon^M(\alpha_M(g,\cdot)_\sharp\mu)$ for all $\mu\in \vrm_\infty(M;\varepsilon)$ and $g\in G$.  Observe that, for each $g\in G$ and $\mu\in\vrm_\infty(M;\varepsilon)$ we have that 
\begin{align*}\int_{M}d_M^2(x,y)\,d\big(\alpha_M(g,\cdot)_\sharp\mu\big)(y)=\int_{M}d_M^2(x,\alpha_M(g,y))\,d\mu(y)
= \int_M d_M^2(\alpha_M(g^{-1},x),y)\,d\mu(y).\end{align*}
Since both the support of $\mu$ and (therefore) that of $\alpha_{M}(g,\cdot)_\sharp \mu$ are contained in a ball of radius $\varepsilon\leq r(M)$, both measures admit a unique Riemannian center of mass. This, together with the equalities above then imply that 
$p_\varepsilon^M(\alpha_{M}(g,\cdot)_\sharp \mu)=\alpha_M(g,p_\varepsilon^M(\mu))$, as we wanted.

Now,  (2) it equivalent to proving that for all $\mu\in \vrm_\infty(M;\varepsilon)$ and $g\in G$ we have that 
\begin{equation}\label{eq:H}
\alpha_M(g,\cdot)_\sharp H(\mu,t) = H\big(\alpha_M(g,\cdot)_\sharp\mu,t\big)\,\,\mbox{for all $t\in[0,1]$}.
\end{equation}
Starting from the LHS, we have
\begin{align*}
    \alpha_M(g,\cdot)_\sharp H(\mu,t) &= (1-t)\alpha_M(g,\cdot)_\sharp \mu + t \,\alpha_M(g,\cdot)_\sharp j^M_\varepsilon \circ p_\varepsilon^M(\mu) \nonumber\\
    &= (1-t)\alpha_M(g,\cdot)_\sharp \mu + t\, \alpha_M(g,\cdot)_\sharp \delta_{p^M_\varepsilon(\mu)}\nonumber\\
    &= (1-t)\alpha_M(g,\cdot)_\sharp \mu + t\, \delta_{\alpha_M(g,p^M_\varepsilon(\mu))}\\
    &\stackrel{(*)}{=}(1-t)\alpha_M(g,\cdot)_\sharp \mu + t\, \delta_{p^M_\varepsilon(\alpha_M(g,\cdot)_\sharp \mu))}\nonumber\\
     &=(1-t)\alpha_M(g,\cdot)_\sharp \mu + t\, j_\varepsilon^M\circ p^M_\varepsilon(\alpha_M(g,\cdot)_\sharp\mu))\nonumber
     =H\big(\alpha_M(g,\cdot)_\sharp\mu,t\big)\nonumber
\end{align*}
where equality $(*)$ follows from item (1) above.
\end{proof}

\subsection{A precompactness theorem for $\dgh^G$.}
For each $X\in\mathcal{M}^G$ and $\varepsilon>0$, we define the \emph{$(G,\varepsilon)$-covering number of $X$} as follows: 

\begin{equation}\label{eq:def:NG}
N_X^G(\varepsilon):=\inf\{\vert U \vert: U\text{ is a }G\text{-invariant }\varepsilon\text{-net of }X\}. 
\end{equation}
\nomenclature[16]{$N_X^G$}{$(G,\varepsilon)$-covering number of $X$}
\nomenclature[16]{$N_{\mathcal{F}}^G$}{$(G,\varepsilon)$-covering number of  family $\mathcal{F}$}

Similarly, for a 
family $\mathcal{F}\subseteq\mathcal{M}^G$, we define \begin{equation}\label{eq:def"N-cov}N_{\mathcal{F}}^G(\varepsilon):=\sup_{X\in\mathcal{F}}N_X^G(\varepsilon)\end{equation} for $\varepsilon>0.$ 

\begin{remark}
Assume that $(X,d_X,\alpha_X)\in\mathcal{M}^G$, $\varepsilon>0$, and a group homomorphism $h:\widetilde{G}\to G$ are given. Then, it is easy to verify that any $G$-invariant $\varepsilon$-net $U$ of $(X,d_X,\alpha_X)$ is also $\widetilde{G}$-invariant $\varepsilon$-net of $(X,d_X,h^\sharp\alpha_X)$. Therefore, we have that $N_{(X,d_X,\alpha_X)}^G(\varepsilon)\geq N_{(X,d_X,h^\sharp\alpha_X)}^{\widetilde{G}}(\varepsilon)$ for each $\epsilon>0$.
\end{remark}

\begin{definition}[Uniformly $G$-totally bounded family]
We say that a family $\mathcal{F}\subseteq\mathcal{M}^G$ is \emph{uniformly $G$-totally bounded} if
\begin{enumerate}
    \item There exists $D>0$ such that $\diam(X)\leq D$ for all $X\in\mathcal{F}$, and

    \item For each $\varepsilon>0$, the number $N_{\mathcal{F}}^G(\varepsilon)$ is finite.
\end{enumerate}
\end{definition}

The following theorem is a generalization of Gromov's precompactness theorem (see \cite[Theorem 7.4.15]{burago2022course}) to the $G$-equivariant setting.

\begin{theorem}[$G$-equivariant  Gromov's precompactness theorem]\label{thm:Gequivprcpt}
Assume that $G$ is a finite group. If $\mathcal{F}\subseteq\mathcal{M}^G$ is uniformly $G$-totally bounded, then $\mathcal{F}$ is precompact with respect to $\dgh^G$.\footnote{I.e., any sequence in $\mathcal{F}$ contains a converging subsequence.}
\end{theorem}

The proof of the theorem is relegated to \Cref{app:proofs-finiteness}.

\subsection{An application: rigidity and finiteness results.}

In \Cref{sec:rigidity}, we work out rigidity results for both the homotopy type and Gromov-Hausdorff distances. In \Cref{sec:finite}, we combine these results with \Cref{thm:Gequivprcpt} to obtain a finiteness result for the number of $G$-homotopy class of suitably constrained Riemannian manifolds. 

\subsubsection{Rigidity}\label{sec:rigidity}
The following theorem provides a rigidity result with respect to the $G$-homotopy type distance and \Cref{cor:rigidity} below translates it into a rigidity claim with respect to the $G$-Gromov-Hausdorff distance.

\begin{theorem}\label{thm:rigidity}
Suppose that $X,Y\in\mathcal{H}^G$. If
$$\dht^G\big((\Pfin(X),\diam_\infty^X),(\Pfin(Y),\diam_\infty^Y)\big)<\tfrac{1}{2}\min\{r^G(X),r^G(Y)\},$$
then $X$ and $Y$ are $G$-homotopy equivalent.
\end{theorem}

From \Cref{Gdhtstability} and \Cref{thm:rigidity} we obtain:

\begin{corollary}\label{cor:rigidity}
Suppose that $X,Y\in\mathcal{H}^G$. If $\dgh^G(X,Y)<\tfrac{1}{4}\min\{r^G(X),r^G(Y)\}$, then $X$ and $Y$ are $G$-homotopy equivalent.
\end{corollary}

See \cite[Theorem A]{peter1990finiteness} for results of similar flavor. Indeed, on page 392 of \cite{peter1990finiteness}, Petersen concludes that any two $n$-dimensional\footnote{In \cite{peter1990finiteness}, Petersen uses the Lebesgue covering dimension.} compact metric spaces $X$ and $Y$ such that there exists $R>0$ such that balls of radius of at most $R$ are contractible within themselves, then $X$ and $Y$ are homotopy equivalent as soon as $\dgh(X,Y)<\tfrac{R}{32 n^2}$; see also \cite[Theorem 2]{yamaguchi1988homotopy}.

\paragraph{Using rigidity ideas for obtaining lower bounds.} \Cref{thm:rigidity} and \Cref{cor:rigidity} yield interesting consequences even when $G$ is the trivial group $G_0$. 

\begin{remark}\label{rem:strat}
If $X,Y\in\mathcal{H}^{G_0}$ are \emph{not} homotopy equivalent, then $$\dgh(X,Y)\geq \tfrac{1}{4}\min\{r^{G_0}(X),r^{G_0}(Y)\}.$$
\end{remark}

\begin{example}\label{rmk:trivialcase} Here we discuss three different applications of \Cref{rem:strat}.
\begin{itemize}
\item For $n\neq m$, via \Cref{thm:Z2hausmannspherevrm}, we find that $\dgh(\Sp^n,\Sp^m)\geq \tfrac{1}{4}\min(\zeta_n,\zeta_m)\geq \tfrac{\pi}{8}.$ 
\item Similarly, if $\mathbb{T}^d$ denotes the $d$-dimensional flat torus,\footnote{Since $r(\mathbb{T}^d)\geq \tfrac{\pi}{2}$.} then $\dgh(\Sp^n,\mathbb{T}^d)\geq \tfrac{\pi}{8}$ for all positive integers $n,d$ not simultaneously equal to $1$.  
\item For $d,d'$ different positive integers, we also have $\dgh(\mathbb{T}^d,\mathbb{T}^{d'})\geq \tfrac{\pi}{8}$.
\end{itemize}
\end{example}

Another family of interesting examples arise by considering \emph{Lens spaces} \cite{tietze1908topologischen}.  Let $p\geq 2$ and $q$ be coprime positive integers. Consider $\Sp^3$ as the unit sphere in $\C^2$. Now, consider the $\Z_p$-action $\alpha_{p,q}$ on $\Sp^3$ induced by the following map:
\begin{align*}
    \Sp^3&\longrightarrow\Sp^3\\
    (z_1,z_2)&\longmapsto \left(e^{\frac{2\pi i}{p}}z_1,e^{\frac{2\pi qi}{p}}z_2\right).
\end{align*}
The \emph{$(p,q)$-Lens space} $L(p;q)$ is defined as $L(p;q):=\Sp^3\slash\Z_p$, the quotient space induced by the aforementioned $\Z_p$-action $\alpha_{p,q}$. As an application of  \Cref{rem:strat}, we obtain  the following result.

\begin{proposition}\label{prop:Lens}
Let $p\geq 2$ and $q_1,q_2$ be positive integers that are coprime with $p$. Then, we have that
$$\dgh(L(p;q_1),L(p;q_2))\geq \tfrac{1}{4}\min\{r(L(p;q_1)),r(L(p;q_2))\}=\frac{\pi}{8p}$$
whenever $q_1q_2\not\equiv\pm n^2$ (mod $p$) for any positive integer $n$.
\end{proposition}
\begin{proof}
Let $q$ be an arbitrary positive integer that is coprime with $p$. Note that $L(p;q)$ is a compact Riemannian manifold with  constant sectional curvature $\Delta_{L(p;q)}\equiv 1$; see \cite[Proposition 2.32]{lee2018introduction}.

Next, we estimate $\mathrm{Conv}(L(p;q))$, the convexity radius of $L(p;q)$. By \cite[Section 3.2]{viana2019isoperimetric} we have $\mathrm{Inj}(L(p;q))\geq\frac{\pi}{p}$. By \cite[pg.177]{petersen2006riemannian}, we conclude that $$\mathrm{Conv}(L(p;q))\geq\min\left\{\frac{\pi}{2\sqrt{\Delta_{L(p;q)}}},\frac{\mathrm{Inj}(L(p;q))}{2}\right\} \geq \frac{\pi}{2p}.$$
Finally, one can conclude that $r(L(p;q))\geq\frac{\pi}{2p}$ (see Theorem \ref{thm:vrmHausmann-G}).

The classification of Lens spaces up to homotopy equivalence is well-known: by \cite[Theorem VI]{olum1953mappings}, two Lens spaces $L(p;q_1)$ and $L(p;q_2)$ are homotopy equivalent if and only if $q_1q_2\equiv\pm n^2$ (mod $p$) for some positive integer $n$. Therefore, by \Cref{rmk:trivialcase}, one can conclude that
$$\dgh(L(p;q_1),L(p;q_2))\geq \tfrac{1}{4}\min\{r(L(p;q_1)),r(L(p;q_2))\}=\frac{\pi}{8p}$$
whenever $q_1q_2\not\equiv\pm n^2$ (mod $p$) for every positive integer $n$.
\end{proof}

By Propositions \ref{prop:dghGvsquotientdgh} and \ref{prop:Lens}, we obtain the following corollary.

\begin{corollary}\label{cor:Lens}
Let $p\geq 2$ and $q_1,q_2$ be positive integers that are coprime with $p$ such that $q_1q_2\not\equiv\pm n^2$ (mod $p$) for any positive integer $n$. Then, we have that
$$\dgh^{\Z_p}\big((\Sp^3,\alpha_{p,q_1}),(\Sp^3,\alpha_{p,q_2})\big)\geq\frac{\pi}{8p}.$$
\end{corollary}

\begin{proof}[\textbf{Proof of \Cref{thm:rigidity}}]
Since $2\,\dht^G\big((\Pfin(X),\diam_\infty^X),(\Pfin(Y),\diam_\infty^Y)\big)<\min\{r^G(X),r^G(Y)\}$, one can chose $\delta\in\bigg(\dht^G\big((\Pfin(X),\diam_\infty^X),(\Pfin(Y),\diam_\infty^Y)\big),\frac{1}{2}\min\{r^G(X),r^G(Y)\}\bigg)$ . Hence, there are $(G,\delta)$-maps $\Phi:(\Pfin(X),\diam_\infty^X)\toG(\Pfin(Y),\diam_\infty^Y)$ and $\Psi:(\Pfin(Y),\diam_\infty^Y)\toG(\Pfin(X),\diam_\infty^X)$ such that $\Psi\circ\Phi$ is $(G,2\delta)$-homotopic to $\mathrm{id}_{\Pfin(X)}$ and $\Phi\circ\Psi$ is $(G,2\delta)$-homotopic to $\mathrm{id}_{\Pfin(Y)}$. Now, choose an arbitrary $\varepsilon>0$ such that $\varepsilon+2\delta<\min\{r^G(X),r^G(Y)\}$.  
Define the following maps:
\begin{itemize}
\item $\Phi_\varepsilon := \Phi|_{\vrm_\infty(X;\varepsilon)}$. Note that its image is contained in $\vrm_\infty(Y;\varepsilon+\delta)$. 
\item $\Psi_\varepsilon := \Psi|_{\vrm_\infty(Y;\varepsilon)}$. Note that its image is contained in $\vrm_\infty(X;\varepsilon+\delta)$. 
\item $\Phi_{\epsilon+\delta}:=\Phi|_{\vrm_\infty(X;\varepsilon+\delta)}$. Note that its image is contained in $\vrm_\infty(Y;\varepsilon+2\delta)$. 
\item $\Psi_{\varepsilon+\delta}:=\Psi|_{\vrm_\infty(Y;\varepsilon+\delta)}$. Note that its image is contained in $\vrm_\infty(X;\varepsilon+2\delta)$. 
\end{itemize}

Consider the following two diagrams (which commute up to $G$-homotopy):

\begin{center}
\begin{tikzcd}
X \arrow[r, hook, "j_\varepsilon^X", red] & \vrm_\infty(X;\varepsilon) \arrow[rr, hook, "\iota_{\varepsilon,\varepsilon+2\delta}^X"] \arrow[rd, "\Phi_\varepsilon",red] &                                                              & \vrm_\infty(X;\varepsilon+2\delta) \arrow[r, "p_{\varepsilon+2\delta}^X",blue,dotted] & X \\
                  &                                              & \vrm_\infty(Y;\varepsilon+\delta) \arrow[ru, "\Psi_{\varepsilon+\delta}",blue,dotted] \arrow[d, shift left=2, "p_{\varepsilon+\delta}^Y",red] &                                      &   \\
                  &                                              & Y \arrow[u, hook', shift left=2, "j_{\varepsilon+\delta}^Y",blue,dotted]                              &                                      &  
\end{tikzcd}
\end{center}

\begin{center}
\begin{tikzcd}
Y \arrow[r, hook, "j_\varepsilon^Y",blue] & \vrm_\infty(Y;\varepsilon) \arrow[rr, hook, "\iota_{\varepsilon,\varepsilon+2\delta}^Y"] \arrow[rd, "\Psi_{\varepsilon}",blue] &                                                              & \vrm_\infty(Y;\varepsilon+2\delta) \arrow[r, "p_{\varepsilon+2\delta}^Y",red,dotted] & Y \\
                  &                                              & \vrm_\infty(X;\varepsilon+\delta) \arrow[ru, "\Phi_{\varepsilon+\delta}",red,dotted] \arrow[d, shift left=2, "p_{\varepsilon+\delta}^X",blue] &                                      &   \\
                  &                                              & X \arrow[u, hook', shift left=2, "j_{\varepsilon+\delta}^X",red,dotted]                              &                                      &  
\end{tikzcd}
\end{center}

Based on these two diagrams, we define the following four $G$-maps: 

\begin{itemize}
\item $\widetilde{\Phi}:=p_{\varepsilon+\delta}^Y\circ\Phi_\varepsilon\circ j_{\varepsilon}^X:X\toG Y$ arising as the composition of the red arrows in the first diagram.

\item $\widetilde{\Psi}:=p_{\varepsilon+\delta}^X\circ\Psi_\varepsilon\circ j_{\varepsilon}^Y:Y\toG X$ arising as the composition of the blue arrows in the second diagram.

\item $\widehat\Phi:=p_{\varepsilon+2\delta}^Y\circ\Phi_{\varepsilon+\delta}\circ j_{\varepsilon+\delta}^X:X\toG Y$ arising as the composition of the  red dotted arrows in the second diagram above.  
\item $\widehat\Psi:=p_{\varepsilon+2\delta}^X\circ\Psi_{\varepsilon+\delta}\circ j_{\varepsilon+\delta}^Y:Y\toG X$ arising as the composition of the  blue dotted arrows in the first diagram above. 
\end{itemize}

\begin{claim}\label{claim:homotopy}
 $\widetilde{\Phi}$ is $G$-homotopic to $\widehat{\Phi}$ and  $\widetilde{\Psi}$ is $G$-homotopic to $\widehat{\Psi}$. 
\end{claim}

Using the claim, we have that 
\begin{align*}
    \widetilde{\Psi}\circ \widetilde{\Phi}&\Gsimeq \widehat\Psi\circ \widetilde{\Phi}&\mbox{(by \Cref{claim:homotopy})}\\
    &=p_{\varepsilon+2\delta}^X\circ\Psi_{\varepsilon+\delta}\circ j_{\varepsilon+\delta}^Y\circ p_{\varepsilon+\delta}^Y\circ\Phi_\varepsilon\circ j_{\varepsilon}^X&\mbox{(by definitions of $\widetilde{\Phi}$ and $\widehat{\Psi}$)}\\
    &\Gsimeq p_{\varepsilon+2\delta}^X\circ\Psi_{\varepsilon+\delta}\circ \mathrm{id}_{\vrm_\infty(Y;\varepsilon+\delta)}\circ \Phi_\varepsilon\circ j_{\varepsilon}^X&\mbox{(by item 1 of \Cref{def:HG})}\\
    &= p_{\varepsilon+2\delta}^X\circ\Psi_{\varepsilon+\delta}\circ\Phi_\varepsilon\circ j_{\varepsilon}^X\\
    &=p_{\varepsilon+2\delta}^X\circ\big(\Psi\circ\Phi\big)|_{\vrm_\infty(X;\varepsilon)}\circ j_{\varepsilon}^X\\ 
    &\Gsimeq p_{\varepsilon+2\delta}^X\circ \iota_{\varepsilon,\varepsilon+2\delta}^X\circ j_{\varepsilon}^X&\mbox{(since $\Psi\circ\Phi$ is $(G,2\delta)$-homotopic to $\mathrm{id}_{\Pfin(X)}$)}\\
    &=p_{\varepsilon+2\delta}^X\circ j_{\varepsilon+2\delta}^X\\&\Gsimeq\mathrm{id}_X &\mbox{(by item 1 of \Cref{def:HG})}.
\end{align*}

Similarly, we also have that $\widetilde{\Phi}\circ\widetilde{\Psi}\Gsimeq\mathrm{id}_Y$. Hence, $X$ and $Y$ are $G$-homotopy equivalent.

\begin{proof}[Proof of \Cref{claim:homotopy}]
We only verify that $\widetilde{\Psi}$ is $G$-homotopic to $\widehat{\Psi}$ as the other claim is analogous.

\medskip

Consider the following diagram: 

\begin{center}
\begin{tikzcd}
Y \arrow[r, hook, "j_{\varepsilon}^Y"] \arrow[rrr, bend left, "j_{\varepsilon+\delta}^Y"] & \vrm_\infty(Y;\varepsilon) \arrow[rr, hook, "\iota_{\varepsilon,\varepsilon+\delta}^X"] \arrow[rd, "\Psi_\varepsilon"] &                                     & \vrm_\infty(Y;\varepsilon+\delta) \arrow[rd, "\Psi_{\varepsilon+\delta}"] &                  \\
                  &                                    & \vrm_\infty(X;\varepsilon+\delta) \arrow[rr, hook, "\iota_{\varepsilon+\delta,\varepsilon+2\delta}^X"] \arrow[rrd, "p_{\varepsilon+\delta}^X"] &                   & \vrm_\infty(X;\varepsilon+2\delta) \arrow[d, "p_{\varepsilon+2\delta}^X"] \\
                  &                                    &                                     &                   & X               
\end{tikzcd}
\end{center}

In the diagram above, the parallelogram in the center commutes as does the triangle with a curved edge on the top left. The triangle in the bottom right commutes up to $G$-homotopy. Then, it is easy to check that
$$\widehat{\Psi}=p_{\varepsilon+2\delta}^X\circ\Psi_{\varepsilon+\delta}\circ j_{\varepsilon+\delta}^Y=p_{\varepsilon+2\delta}^X\circ\Psi_{\varepsilon+\delta}\circ\iota_{\varepsilon,\varepsilon+\delta}^Y\circ j_{\varepsilon}^Y=p_{\varepsilon+2\delta}^X\circ\iota_{\varepsilon+\delta,\varepsilon+2\delta}^X\circ\Psi_\varepsilon\circ j_{\varepsilon}^Y\Gsimeq p_{\varepsilon+\delta}^X\circ\Psi_\varepsilon\circ j_{\varepsilon}^Y=\widetilde{\Psi}.$$
\end{proof}
\end{proof}

\subsubsection{Finiteness for a class of Riemannian manifolds.}\label{sec:finite}

For each $n\in\Z_{>0}$, $C,D,K,t\in\R_{>0}$, and $\kappa\in\R$, let $\mathcal{F}_{n,C,D,\kappa,K,t}^G$ be the collection of all complete $n$-dimensional $G$-Riemannian manifolds $(M,g_M)$ with $\diam(M)\leq D$, $\mathrm{Conv}(M)\geq C$,  Ricci curvature  bounded below $\kappa (n-1) g_M$,  sectional curvatures upper bounded by $K$, and $\mathrm{Sep}^G(M)\geq t$. Then, for each $M\in\mathcal{F}_{n,C,D,\kappa,K,t}^G$, we have that $r^G(M)\geq\min\{C,\frac{1}{4}\pi K^{-1/2}\}=:r_{C,K}$.  Also let $\nu(n,\kappa,r)$ denote the volume of a ball of radius $r$ in the $n$-dimensional model space with  constant sectional curvature $\kappa$. We then have the following.

\begin{theorem}\label{cor:Riemfinite}
For each $n\in\Z_{>0}$ $C,D,K\in\R_{>0}$, and $\kappa\in\R$, the collection $\mathcal{F}_{n,C,D,\kappa,K,t}^G$ has at most
$$\sum_{j=1}^{M_{n,C,D,K,\kappa}}j^{\frac{j^2-j}{2}}\cdot\vert\hom(G,\mathcal{S}_j)\vert$$
many $G$-homotopy types, where $M_{n,C,D,K,\kappa}:=\Big\lceil\frac{\nu(n,\kappa,D)}{\nu\left(n,\kappa,\min\{\frac{r_{C,K}}{24},\frac{t}{6}\}\right)}\Big\rceil$ and $\hom(G,\mathcal{S}_j)$ denotes the set of all homomorphisms from $G$ to the symmetric group $\mathcal{S}_j$ on $[j]:=\{1,2,\dots,j\}$.
\end{theorem}

The following lemma is a generalization to the $G$-equivariant setting of a result that has been implicitly used in several publications such as \cite{peter1990finiteness,yamaguchi1988homotopy} and \cite[page 274]{gromov1999metric}. Its proof is deferred to \Cref{app:proofs-finiteness}.

\begin{lemma}\label{lemma:prcptcoverbdd}
Let $\mathcal{F}\subseteq\mathcal{M}^G$ be a precompact family of connected compact $G$-metric spaces. Then, for any $\varepsilon>0$, $(\mathcal{F},\dgh^G)$ can be covered by
$$\sum_{j=1}^{N_{\mathcal{F}}^G(\varepsilon/3)}j^{\frac{j^2-j}{2}}\cdot\vert\hom(G,\mathcal{S}_j)\vert$$
many $\varepsilon$-balls (in the sense of the $G$-Gromov-Hausdorff distance).
\end{lemma}

\begin{remark}
If $G=G_0$ then $\vert\hom(G_0,\mathcal{S}_n)\vert=1$ for every positive integer $n$, and in that case \Cref{lemma:prcptcoverbdd} recovers the statement of \cite[Lemma 6.4]{memoli2024persistent}. It is natural to ask how $\vert\hom(G,\mathcal{S}_n)\vert$ changes as the complexity of the group $G$ increases. However, this question is subtle even when $G$ is a  cyclic group.  Let $a_{m,n}:=\vert\hom(\Z_m,\mathcal{S}_n)\vert$. The asymptotic growth rate of $a_{m,n}$ has been studied in \cite{chowla1951recursions,moser1955solutions,moser1956asymptotic,wilf1986asymptotics,muller1997finite}. For example:
\begin{itemize}
    \item when $G=\Z_2$, by \cite[Lemma 1 and Theorem 8]{chowla1951recursions}, $a_{2,n}$ satisfies $a_{2,n}=a_{2,n-1}+(n-1)a_{2,n-2}$ and
    $$a_{2,n}\sim 2^{-1/2}n^{n/2}\exp\left(-\frac{n}{2}+n^{1/2}-\frac{1}{4}\right)\text{ as }n\rightarrow\infty.$$

    \item when $G=\Z_p$ for any prime $p>2$, by \cite[Section 3]{moser1955solutions} and \cite[Example 2]{moser1956asymptotic},
    $$a_{p,n}\sim p^{-1/2}n^{(1-1/p)n}\exp\left(-\frac{p-1}{p}n+n^{1/p}\right)\text{ as }n\rightarrow\infty.$$
\end{itemize}

Hence in these cases, $a_{m,n}$ grows rapidly, far exceeding $1=\vert\hom(G_0,\mathcal{S}_n)\vert$. More generally, assume a surjective group homomorphism $h:\widetilde{G}\twoheadrightarrow G$ is given. Then, one can easily verify that $\vert\hom(\widetilde{G},\mathcal{S}_n)\vert\geq\vert\hom(G,\mathcal{S}_n)\vert$. Thus, in general, the quantity in \Cref{lemma:prcptcoverbdd} is expected to increase rapidly as the size of of the group $G$ increases.
\end{remark}

The following  lemma is a generalization of a classical result relating packing numbers and covering numbers; see e.g. \cite[pg.299]{petersen2006riemannian}.

\begin{lemma}\label{lemma:Gcoveringpacking}
Let $X$ be a compact $G$-metric space with $\mathrm{Sep}^G(X)>0$. Then, for any $\varepsilon\in(0,\mathrm{Sep}^G(X)]$, we have that
$$N_M^G(\varepsilon)\leq\sup\{\vert U \vert: U\subseteq X\text{ is a }G\text{-invariant set s.t. }\{B_{\varepsilon/2}(u)\}_{u\in U}\text{ is a packing in }X\}.$$
\end{lemma}
\begin{proof}
Let $l:=\sup\{\vert U \vert: U\subseteq X\text{ is a }G\text{-invariant set s.t. }\{B_{\varepsilon/2}(u)\}_{u\in U}\text{ is a packing in }X\}$. Then, there exists a $G$-invariant subset $U:=\{u_1,\dots,u_l\}\subseteq X$ such that $B_{\varepsilon/2}(u_i)$ and $B_{\varepsilon/2}(u_j)$ are disjoint whenever $i\neq j$. Now, we claim that $U$ is a $G$-invariant $\varepsilon$-net. Suppose not, then one can choose $x\in X\backslash\bigcup_{i=1}^l B_\varepsilon(u_i)$. Then, by the choice of $\varepsilon$, one can verify that $U\cup\{\alpha_X(g,x):g\in G\}$ is a $G$-invariant subset of $X$ where the $\varepsilon/2$-balls with centers in the set forms a packing in $X$. This contradicts  the maximality of $l$ and completes the  proof.
\end{proof}

\begin{proof}[\textbf{Proof of \Cref{cor:Riemfinite}}]
By \Cref{lemma:Gcoveringpacking} we have that $N_M^G(\varepsilon)\leq\frac{\nu(n,\kappa,D)}{\nu(n,\kappa,\varepsilon/2)}$ for any $\varepsilon\in (0,t]$ (see \cite[Lemma 36 and the proof of Corollary 31]{petersen2006riemannian}). The proof follows from \Cref{cor:rigidity} and \Cref{lemma:prcptcoverbdd}.
\end{proof}

\section{$G$-index and lower bounds for $\dht^G$.}\label{sec:Gindices}

We first introduce the notion of $G$-persistence index and then use this concept to establish a lower bound for the $G$-homotopy type distance.

\subsection{$G$-persistent indices}

\begin{definition}[$G$-index]
Let $Y$ be a $G$-space and $\{X_\xi\}_{\xi\in\Xi}$  be a collection of $G$-spaces indexed by a nonempty subset $\Xi\subseteq\R_{> 0}$. Then, we define the $G$-index $\mathrm{ind}_G(Y,\{X_\xi\}_{\xi\in\Xi})$ from $Y$ to $\{X_\xi\}_{\xi\in\Xi}$ in the following way:
$$\mathrm{ind}_G(Y,\{X_\xi\}_{\xi\in\Xi}):=\inf\{\xi\in\Xi:\text{there is a }G\text{-map }Y\toG X_\xi\}.$$
\end{definition}

\begin{remark}
When $G=\Z_2$, the $\Z_2$-index $\mathrm{ind}_{\Z_2}(Y,\{\Sp^n\}_{n\in\Z_{\geq 0}})$ from $Y$ to $\{\Sp^n\}_{n\in\Z_{\geq 0}}$  boils down to the following classical notion $\mathrm{ind}_{\Z_2}(Y)$, the $\Z_2$-index of $Y$ \cite{yang1955theorems}:
$$\mathrm{ind}_{\Z_2}(Y):=\min\{n\in\Z_{\geq 0}:\text{there is a }\Z_2\text{-map }Y\toZtwo \Sp^n\}.$$

This notion of $\Z_2$-index has been employed to obtain a certain generalization of the Borsuk-Ulam theorem \cite[Chapter 5]{matouvsek2003using}.
\end{remark}

We are now ready to introduce three (a priori distinct) quantities, defined via the notion of $G$-index, which will ultimately provide lower bounds for the $G$-homotopy type distance between filtered $G$-topological spaces, and, via \Cref{Gdhtstability}, for the $G$-Gromov–Hausdorff distance between $G$-metric spaces. Below, note that $\vrm_p(X;\bullet)$, $\vert\vr(X;\bullet)\vert$, and $B_\bullet(X,\TS(X))$ are all filtrations, in particular they are collections of $G$-spaces indexed by the whole real line $\R$ (see \Cref{def:pers-fams,def:VR,def:VRm} and \Cref{sec:katetov}).

\begin{definition}[$G$-persistent indices]\label{def:clowerbdd}
For any two $G$-metric spaces $X,Y$, $p\in[1,\infty]$, and $\varepsilon\in\R_{\geq 0}$, we define the following quantities which we call \emph{$G$-persistence indices}: 

\begin{enumerate}
    \item \emph{$G$-persistent index induced by the $p$-Vietoris-Rips thickening filtration:}\\ $c_p^G(X,Y;\varepsilon):=\begin{cases}\mathrm{ind}_G(Y,\vrm_p(X;\bullet))&\text{if }\varepsilon=0\\ \mathrm{ind}_G(\vrm_p(Y;\varepsilon),\vrm_p(X;\varepsilon+\bullet))&\text{if }\varepsilon>0,\end{cases}$

    \item \emph{$G$-persistent index induced by the Vietoris-Rips filtration: }\\$c_{\vr}^G(X,Y;\varepsilon):=\begin{cases}\mathrm{ind}_G(Y,\vert\vr(X;\bullet)\vert)&\text{if }\varepsilon=0\\ \mathrm{ind}_G(\vert\vr(Y;\varepsilon)\vert,\vert\vr(X;\varepsilon+\bullet)\vert)&\text{if }\varepsilon>0,\text{ and}\end{cases}$

    \item \emph{$G$-persistent index induced by the tight span neighborhood filtration:}\\$c_\TS^G(X,Y;\varepsilon):=\begin{cases}\mathrm{ind}_G\big(Y,B_\bullet(X,\TS(X))\big)&\text{if }\varepsilon=0\\ \mathrm{ind}_G\big(B_\varepsilon(X,\TS(X)),B_{\varepsilon+\bullet}(X,\TS(X))\big)&\text{if }\varepsilon>0.\end{cases}$
\end{enumerate}
\nomenclature[16]{$c_p^G(X,Y;\varepsilon)$}{The $G$-persistent index associated to $\vrm_p(X;\bullet)$}
\nomenclature[17]{$c_{\vr}^G(X,Y;\varepsilon)$}{The $G$-persistent index associated to $\vert\vr(X;\bullet)\vert$}
\nomenclature[18]{$c_\TS^G(X,Y;\varepsilon)$}{The $G$-persistent index associated to $B_\bullet(X,\TS(X))$}
\end{definition}

\begin{remark}\label{rmk:referpolymath}
The case when $G = \mathbb{Z}_2$ and $\varepsilon = 0$ of $c_{\vr}^G(X, Y; \varepsilon)$ agrees with an invariant that was considered in \cite[Definition 1.1]{adams2022gromov}, specifically when $X$ and $Y$ are spheres equipped with the standard antipodal action. Introducing the parameter $\varepsilon$, as we do in this paper,  allows, in general, for sharper lower bounds on the Gromov–Hausdorff distance compared to the $\varepsilon = 0$ case; see \Cref{ex:eps-good-1,ex:eps-good-2}. Note that neither $c_\vr(X,Y;0)$ nor the case $\varepsilon=0$ of the other two $G$-persistent indices ``see" the metric structure of $Y$: only its topology is taken into account.
\end{remark}

In \Cref{sec:lb-dght-pi} below we establish the lower bound for $\dht^G$ (and also for $\dgh^G$) via $c_p^G$ stated in the Main Theorem on page \pageref{page:main-theorem}. In \Cref{sec:rel-pis} we will establish several inequalities between  three different persistence indices $c_p^G$, $c_\vr^G$ and $c_\TS^G$. Regarding the former point, persistence indices, however,  provide useful lower bound for the $G$-Gromov-Hausdorff distance only in situations when the group $G$ acts freely on one of the spaces $X$ and $Y$.

\begin{remark}[Persistent indices and fixed points]\label{rem:fixed-points}
Assume that $x_0\in X$ is a fixed point for the $G$-action: $\alpha_X(g,x_0)=x_0$ for all $g\in G$. Then, the constant map $\psi:Y\to X$ defined by $y\mapsto x_0$ is a $G$-map. Also, for each $\varepsilon>0$, $\psi$ induces $G$-maps $\psi_\varepsilon:\vrm(Y;\varepsilon)\toG \vrm(X;\varepsilon)$ by sending every $\mu$ in the domain to $\delta_{x_0}$. As a consequence, $c_p(X,Y;\varepsilon)=0$ for all $\varepsilon>0$.  A similar argument also shows that $c_\vr^G(X,Y;\varepsilon)= c_\TS^G(X,Y;\varepsilon)=0$ in this case.
\end{remark}

\begin{remark}\label{rmk:GGtildepindex}
Assume a group homomorphism $h:\widetilde{G}\to G$ is given.  Recall that for any $G$-map $\varphi:X\toG Y$ between two $G$-spaces $X$ and $Y$ one can induce a $\widetilde{G}$-map $\varphi:X \stackrel{\widetilde{G}}{\longrightarrow} Y$ in a canonical manner (see  item (2) of Remark \ref{rmk:R-order}). With this observation, one  sees that
$$c_p^G(X,Y;\varepsilon)\geq c_p^{\widetilde{G}}(X,Y;\varepsilon),\,c_\vr^G(X,Y;\varepsilon)\geq c_\vr^{\widetilde{G}}(X,Y;\varepsilon),\text{ and }c_\TS^G(X,Y;\varepsilon)\geq c_\TS^{\widetilde{G}}(X,Y;\varepsilon).$$
\end{remark}

\subsection{A lower bound for $\dht^G$ via $G$-persistent indices.}\label{sec:lb-dght-pi}
The following proposition exhibits a lower bound for the $\dht^G$ distance between two filtrations of the form given by \Cref{prop:PfinGaction} via the $G$-persistent index $c_p^G$.

\begin{proposition}\label{prop:strongercGpXYlbdd}
For any two $G$-metric spaces $X,Y$ and $p\in[1,\infty]$, we have that
$$\dht^G\big((\Pfin(X),\diam_p^X),(\Pfin(Y),\diam_p^Y)\big)\geq \sup_{\varepsilon\geq 0}c_p^G(X,Y;\varepsilon).$$
\end{proposition}
\begin{proof}
First, fix an arbitrary $\varepsilon>0$. Next, choose an arbitrary $\eta>0$ such that

$\dht^G\big((\Pfin(X),\diam_p^X),(\Pfin(Y),\diam_p^Y)\big)<\eta$. Then, there exists a $(G,\eta)$-map
$$\Psi:(\Pfin(Y),\diam_p^Y)\toG(\Pfin(X),\diam_p^X).$$
Hence, we have the $G$-map $\Psi\vert_{\vrm_p(Y;\varepsilon)}:\vrm_p(Y;\varepsilon)\toG\vrm_p(X;\varepsilon+\eta)$ implying that $c_p^G(X,Y;\varepsilon)\leq\eta$. Since the choice of $\eta$ is arbitrary, we have that
$$\dht^G\big((\Pfin(X),\diam_p^X),(\Pfin(Y),\diam_p^Y)\big)\geq c_p^G(X,Y;\varepsilon).$$

Now, consider the case when $\varepsilon=0$. Choose $\eta$ and $\Psi$ as before. Recall that, for an arbitrary $\delta>0$, there is the  canonical isometric embedding $j:Y\hooktoG\vrm_p(Y;\delta)$ s.t. $y\mapsto\delta_y$. Then, there is the following $G$-map $\Psi\vert_{\vrm_p(Y;\delta)}\circ j:Y\toG\vrm_p(X;\delta+\eta)$. Hence, $c_p^G(X,Y;0)\leq\delta+\eta$. Since the choice of $\eta$ and $\delta$ are arbitrary, one can conclude that
$$\dht^G\big((\Pfin(X),\diam_p^X),(\Pfin(Y),\diam_p^Y)\big)\geq c_p^G(X,Y;0).$$
This completes the proof.
\end{proof}

It is natural to ask if  there are cases when $\sup_{\varepsilon\geq 0}c_p^G(X,Y,\varepsilon)$ strictly larger than $c_p^G(X,Y;0)$ for which we need to understand how the function $\varepsilon\mapsto c_p^G(X,Y;\varepsilon)$ behaves.

\begin{example}
Assume that $G=\Z_2$, $X=\Sp^1$, and $Y=\Sp^2$. Then, it is easy to verify that,  by \Cref{thm:Z2hausmannsphere}, for any $\varepsilon\in[0,\zeta_2]$ we have $c_\infty^{\Z_2}(\Sp^1,\Sp^2;\varepsilon)=\frac{2\pi}{3}-\varepsilon$.
\end{example}

The example above suggests the possibility that the function $\varepsilon\mapsto c_p^G(X,Y;\varepsilon)$ is always nonincreasing. However, the following examples show that this is not the case, in general; see also \Cref{q:spheres-mono} below.

\begin{example}\label{ex:eps-good-1}
Assume that $G=\Z_2$ and $X=\Sp^1$. Also, let $Y\subset\Sp^N$ ($N\gg 1$) be a finite $\Z_2$-invariant subset of the  sphere $\Sp^N$ such that there exists positive numbers $\eta,\eta',\delta$ s.t. $\eta'\geq \eta$ with the following properties:\footnote{Latschev's theorem \cite{latschev2001vietoris} guarantees that it suffices to choose $Y$ as a sufficiently dense finite sample of $\Sp^N$.} 
\begin{itemize}
\item[(1)] $\vrm_\infty(Y;\varepsilon)=Y$ if $\varepsilon\in (0,\eta]$,  and 
\item[(2)] $\vrm_\infty(Y;\varepsilon)$ is  $\Z_2$-homotopy equivalent to $\Sp^N$ if $\varepsilon\in (\eta',\eta'+\delta)$. 
\end{itemize}
Then, it follows that $c_\infty^{\Z_2}(\Sp^1,Y;\varepsilon)=0$ for any $\varepsilon\in(0,\eta]$ and $c_\infty^{\Z_2}(\Sp^1,Y;\varepsilon)\gg 0$ for any $\varepsilon\in(\eta',\eta'+\delta)$.
\end{example}

\begin{example}\label{ex:eps-good-2}
We also have an example involving a connected space $Y$. Fix a tiny positive real number $r\ll 1$. Now, let $G=\Z_2$, $X=\Sp^2$, and $Y=r\cdot\Sp^1$. First of all, if $\varepsilon\in\left(0,\frac{2r\pi}{3}\right]$, then $\vrm_\infty(Y;\varepsilon)$ is $\Z_2$-homotopy equivalent to $\Sp^1$ by \Cref{thm:Z2hausmannspherevrm}. Therefore,

$$c_\infty^{\Z_2}(\Sp^2,Y;\varepsilon)=0\text{ if }\varepsilon\in\left(0,\frac{2r\pi}{3}\right].$$

Also, by \cite[Theorem 7.4]{adamaszek2017vietoris} and \cite[Proposition 5.3.2]{matouvsek2003using}, there is a $\Z_2$-map from $\Sp^{2n-1}$ to $\vr(Y;\varepsilon)$ if $\varepsilon\in\left(\frac{2(n-1)r\pi}{2n-1},\frac{2nr\pi}{2n+1}\right]$. Hence, with the aid of \Cref{thm:Z2hausmannsphere} and  item 1.(b) of \Cref{prop:strongercGpXYlbddalernatives}, one can verify that
$$c_\infty^{\Z_2}(\Sp^2,Y;\varepsilon)\geq c_\vr^{\Z_2}(\Sp^2,Y;\varepsilon) > \zeta_2-\varepsilon\text{ if }\varepsilon>\frac{2r\pi}{3}.$$

Furthermore, one may increase the dimension of $Y$ by letting $Y=(r\cdot\Sp^1)^N$ for any positive integer $N$. Then, with the help of  \cite[Corollary 5.2]{adams2021operations}, one can still verify the analogous result for this $Y$, too. 
\end{example}

\Cref{ex:eps-good-1,ex:eps-good-2} above imply that, as a lower bound for $\dgh^G$, in the sense provided by the inequality $2\dgh^G(X,Y)\geq c_\infty^G(X,Y;0)$ stated in the main theorem (on page \pageref{page:main-theorem}), $c_\infty^G(X,Y;0)$ can be strictly worse than $\sup_{\varepsilon>0}c_\infty^G(X,Y;\varepsilon).$

\begin{question}\label{q:spheres-mono}
    Is it true that $\varepsilon\mapsto c^{\Z_2}_\infty(\Sp^n,\Sp^m;\varepsilon)$ is non-increasing for all positive integers $n,m$ such that $n>m$?
\end{question}

The following example compares the values of the persistent homology and the persistent index based lower bounds stated in the Main Theorem on page \pageref{page:main-theorem}. Recall, from \Cref{prop:Gdnmatter}, that in the case of spheres, viewed as $\Z_2$-metric spaces, for $n>m$, we have
$$\di^{\Z_2}\big(\h_m(\vert\vr(\Sp^m;\bullet)\vert),\h_m(\vert\vr(\Sp^n;\bullet)\vert)\big)=\frac{\zeta_m}{2},$$
which, by \cite[Corollary 5.10]{adams2024persistent}, also implies that for the $\infty$-Vietoris-Rips metric thickenings,
$$\di^{\Z_2}\big(\h_m(\vrm_\infty(\Sp^m;\bullet)),\h_m(\vrm_\infty(\Sp^n;\bullet))\big)=\frac{\zeta_m}{2}.$$
In contrast, item (1) of \Cref{prop:cforspheres} states that $$c_{\infty}^{\Z_2}(\Sp^m,\Sp^n;0)\geq\zeta_m.$$ That is, in this case, the lower bound on $\dgh^{\Z_2}(\Sp^m,\Sp^n)$ derived from persistent indices is twice as strong as the one obtained from persistent homology. The example below illustrates a situation where the reverse holds—i.e., persistent homology provides a sharper lower bound than persistent indices.

\begin{figure}
\begin{center}
\includegraphics[width = 0.8\linewidth]{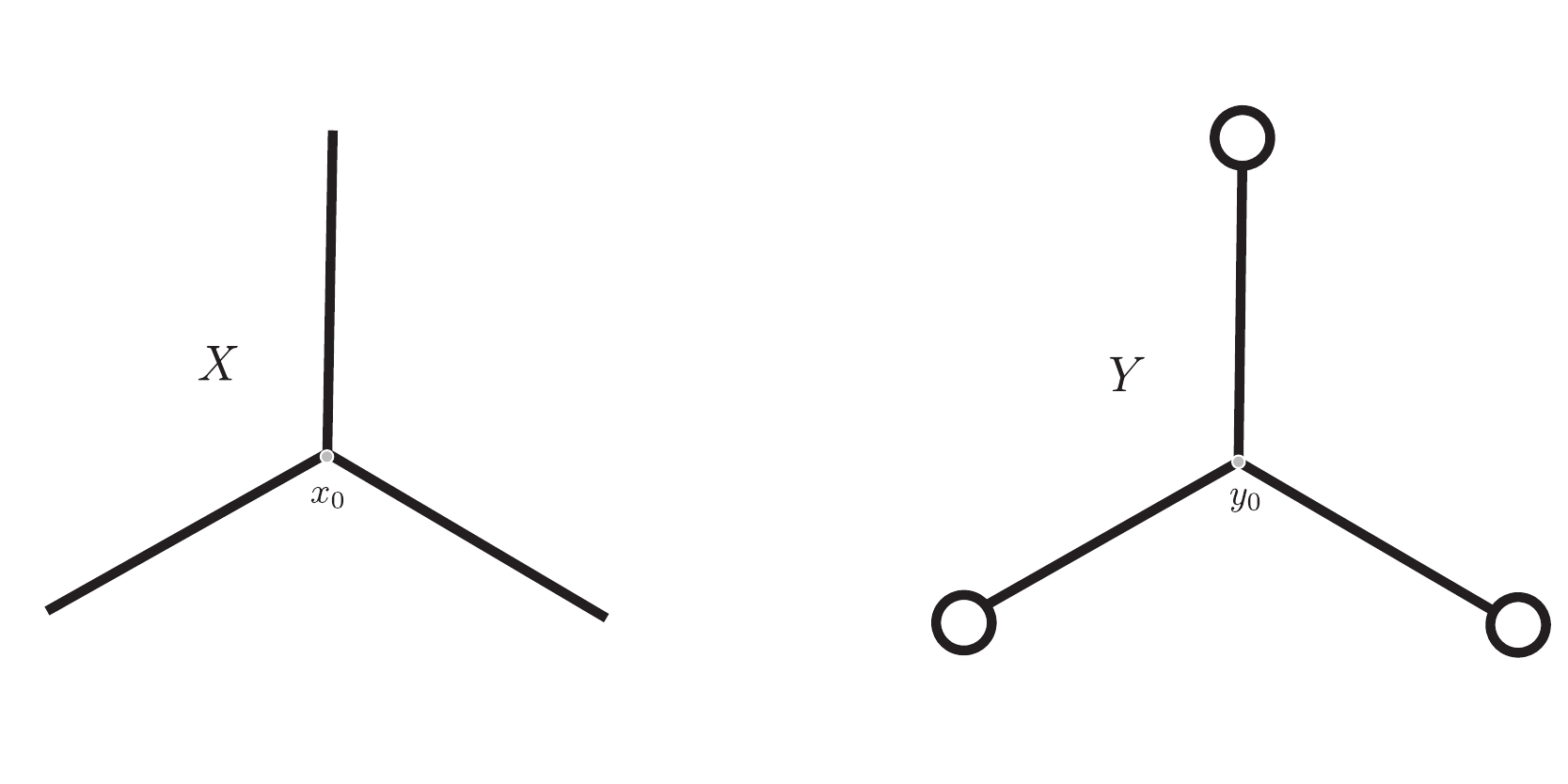}  
\end{center}
    \caption{The geodesic $\Z_3$-metric spaces $X$ and $Y$ from \Cref{ex:compatibility}. The $\Z_3$-action on each of these spaces exchanges different branches leaving the center points fixed.}
     \label{fig:compatibility}
\end{figure}
 
\begin{example}\label{ex:compatibility}
This  example shows that, for some choices of $G$, $X$, and $Y$, we can have
$$\di^G\big(\h_k(\vrm_p(X;\bullet)),\h_k(\vrm_p(Y;\bullet))\big)>\sup_{\varepsilon\geq 0}c_p^G(X,Y;\varepsilon).$$
Consider the 1-dimensional geodesic $\Z_3$-metric spaces $X$ and $Y$ with  $\Z_3$-actions shown in \Cref{fig:compatibility}. Note that $x_0$ and $y_0$ are fixed points w.r.t the respective $\Z_3$-actions. Then, by \Cref{rem:fixed-points},  one  concludes that $\sup_{\varepsilon\geq 0}c_\infty^{\Z_3}(X,Y;\varepsilon)=\sup_{\varepsilon\geq 0}c_\infty^{\Z_3}(Y,X;\varepsilon)=0$. However, note that
$$\di^{\Z_3}\big(\h_1(\vrm_\infty(X;\bullet)),\h_1(\vrm_\infty(Y;\bullet))\big)\geq\di\big(\h_1(\vrm_\infty(X;\bullet)),\h_1(\vrm_\infty(Y;\bullet))\big)>0.$$ 
Indeed, whereas $\h_1(\vrm_\infty(X;\bullet))=0$, $\h_1(\vrm_\infty(Y;\bullet))$ will not be trivial as it detects the presence and size of the 3 geodesic loops.\footnote{If $s$ is the common length of the 3 loops in $Y$, one can in fact prove that the interleaving distance above equals $\tfrac{s}{6}$.} This example can obviously be generalized to a case exhibiting $\Z_k$ symmetry, for every positive integer $k$, by considering versions of $X$ and $Y$ with $k$-branches.
\end{example}

\subsection{Relationship between different $G$-persistent indices.}\label{sec:rel-pis}

Now, we will compare the three $G$-persistent indices defined in \Cref{def:clowerbdd}. But, before that we need the following definition.

\begin{definition}[$G$-separation]
Let $G$ be a finite group and $X$ be a $G$-metric space. Then, we define the \emph{$G$-separation} of $X$ as follows:
$$\mathrm{Sep}^G(X):=\inf\{d_X(x,\alpha_X(g,x)):g\in G\text{ is nontrivial element and }x\in X\}.$$

If $G$ is a trivial group, then we define $\mathrm{Sep}^G(X):=+\infty$.
\nomenclature[19]{$\mathrm{Sep}^G(X)$}{$G$-separation}
\end{definition}

Note that if the $G$-action on $X$ is trivial, then its $G$-separation is zero. However, the following lemma shows that one can guarantee  the $G$-separation to be strictly positive when $X$ is compact and the $G$-action is free.

\begin{lemma}\label{lemma:posGsep}
Let $G$ be a finite group and $X$ be a compact $G$-metric space with a free action $\alpha_X$. Then, we always have $\mathrm{Sep}^G(X)>0$.
\end{lemma}
\begin{proof}
Let $\mathrm{Sep}^G_X:X\longrightarrow\R_{\geq 0}$ be the map such that $\mathrm{Sep}^G_X(x):=\min\{d_X(x,\alpha_X(g,x)):g\in G\text{ is nontrivial}\}$ for all $x\in X$. Since the $G$-action is free, it follows that $\mathrm{Sep}^G_X(x)>0\,\forall x\in X$. Given that $X$ is compact, we can obtain the claim by showing that $\mathrm{Sep}^G_X$ is continuous. We will prove a  stronger claim: $\mathrm{Sep}^G_X$ is $2$-Lipschitz. Indeed, fix any two points $x,x'\in X$. Then, 
\begin{align*}
\Big\vert \mathrm{Sep}_X^G(x)-\mathrm{Sep}_G^X(x')\big\vert &=
\Big\vert\min_{g\in G} d_X(x,\alpha_X(g,x))-\min_{g\in G}d_X(x',\alpha_X(g,x')) \Big\vert\\
&\leq \max_{g\in G}\big\vert d_X(x,\alpha_X(g,x))-d_X(x',\alpha_X(g,x')) \big\vert\\
&\leq \max_{g\in G} \big(d_X(x,x')+d_X(\alpha_X(g,x),\alpha_X(g,x'))\big)\\
&=2d_X(x,x')
\end{align*}
\end{proof}

\begin{example}\label{ex:G-sep}
Consider the following $\Z_2$-metric spaces $X$:
\begin{enumerate}
    \item $X=\Sp^n$ (see the item (2) of \Cref{ex:spheres-Z2}). Then, $\mathrm{Sep}^{\Z_2}(X)=\pi$ since $d_X(x,\alpha_X(-1,x))=\pi$ for all $x\in X$.

    \item $X=\square^n_\infty$ (see the item (3) of \Cref{ex:spheres-Z2}). Then, $\mathrm{Sep}^{\Z_2}(X)=2$ since $d_X(x,\alpha_X(-1,x))=2$ for all $x\in X$.

    \item $X=\Sp^n_\infty$ (see the item (3) of \Cref{ex:spheres-Z2}). Then, for each $x=(x_0,x_1,\cdots,x_n)\in X$, we have that $d_X(x,\alpha_X(-1,x))=2\Vert (x_0,x_1,\cdots,x_n) \Vert_{\ell^\infty}$. Therefore, one can easily verify that $\mathrm{Sep}^{\Z_2}(X)=\frac{2}{\sqrt{n+1}}$.
\end{enumerate}
\end{example}

\begin{proposition}\label{prop:strongercGpXYlbddalernatives}
Suppose that $G$ is a finite group. Then, for any two compact $G$-metric spaces $X$ and $Y$ the followings hold:
\begin{enumerate}
    \item
    \begin{itemize}
        \item[(a)] For $\varepsilon=0$, we have that $$c_{\infty}^G(X,Y;0)=c_{\vr}^G(X,Y;0).$$

        \item[(b)] For any $\varepsilon>0$, we have that
    $$c_{\vr}^G(X,Y;\varepsilon)\leq c_{\infty}^G(X,Y;\varepsilon)\leq\inf_{\delta>0}\big\{c_{\vr}^G(X,Y;\varepsilon+\delta)+\delta\big\}.$$
    \end{itemize}

    \item 
    \begin{itemize}
        \item[(a)] For $\varepsilon=0$, we have that $$c_{\vr}^G(X,Y;0)\leq 2c_\TS^G(X,Y;0).$$ Furthermore, if $\mathrm{Sep}^G(X)>0$ and either $c_{\vr}^G(X,Y;0)<\mathrm{Sep}^G(X)$ or $2c_\TS^G(X,Y;0)<\mathrm{Sep}^G(X)$, then we have that $$c_{\vr}^G(X,Y;0)=2c_\TS^G(X,Y;0).$$
        
        \item[(b)] Assume that $\mathrm{Sep}^G(X),\mathrm{Sep}^G(Y)>0$. For any $\varepsilon\in \left(0,\frac{\mathrm{Sep}^G(Y)}{2}\right]$, if either $$\text{$c_{\vr}^G(X,Y;2\varepsilon)<\mathrm{Sep}^G(X)-2\varepsilon$ or $2c_\TS^G\left(X,Y;\varepsilon\right)<\mathrm{Sep}^G(X)-2\varepsilon$,}$$ then we have that $$c_{\vr}^G(X,Y;2\varepsilon)=2c_\TS^G\left(X,Y;\varepsilon\right).$$
    \end{itemize}
\end{enumerate}
\end{proposition}

\begin{question}
\begin{itemize}
    \item About item 1 (b): Are there $G, X, Y$, and $\varepsilon>0$ such that $c_\infty^G(X,Y;\varepsilon)$ is strictly larger than $c_{\vr}(X,Y;\varepsilon)$? If so, is the upper bound $\inf_{\delta>0}\{c_{\vr}^G(X,Y;\varepsilon+\delta)+\delta\}$ tight?
    
    \item About item 2 (a): Can we find an example of $G$, $X$, and $Y$ such that $c_{\vr}^G(X,Y;0)$ is strictly smaller than $2c_\TS^G(X,Y;0)$? Note that this would necessarily require that $$\min\big(c_{\vr}^G(X,Y;0),2c_\TS^G(X,Y;0)\big)\geq\mathrm{Sep}^G(X).$$
\end{itemize}
\end{question}

\begin{example}\label{ex:cvrce}
\begin{enumerate}
    \item As already mentioned in \Cref{rmk:referpolymath}, the quantity $c_\vr^{\Z_2}(\Sp^m,\Sp^n;0)$ was already considered in \cite{adams2022gromov}. Furthermore, by \cite[Remark 5.6]{adams2022gromov}, we have that $c_\vr^{\Z_2}(\Sp^m,\Sp^n;0)<\pi=\mathrm{Sep}^{\Z_2}(\Sp^m)$ (see the item (1) of \Cref{ex:G-sep}) for all $0< m < n <\infty$. Hence, by \Cref{prop:strongercGpXYlbddalernatives}, we have that
    $$c_{\infty}^{\Z_2}(\Sp^m,\Sp^n;0)=c_{\vr}^{\Z_2}(\Sp^m,\Sp^n;0)=2c_\TS^{\Z_2}(\Sp^m,\Sp^n;0)$$
    for all $0< m < n <\infty$.

    \item Let's consider the case when $G=\Z_2$, $X=\square^m_\infty$, and $Y=\square^n_\infty$ for some $0< m < n <\infty$. Later, we will see that both of $c_{\vr}^{\Z_2}(\square^m_\infty,\square^n_\infty;0)$ and $2c_{\TS}^{\Z_2}(\square^m_\infty,\square^n_\infty;0)$ are greater than or equal to $2=\mathrm{Sep}^{\Z_2}(\square^m_\infty)$ (see the item (2) of \Cref{ex:G-sep}, the item (1) of \Cref{prop:cforsphereslinfty} and its proof). In particular, we cannot guarantee the equality between $c_{\vr}^{\Z_2}(\square^m_\infty,\square^n_\infty;0)$ and $2c_{\TS}^{\Z_2}(\square^m_\infty,\square^n_\infty;0)$ via \Cref{prop:strongercGpXYlbddalernatives}.

    \item Let's consider the case when $G=\Z_2$, $X=\Sp^m_\infty$, and $Y=\Sp^n_\infty$ for some $0< m < n <\infty$. Later, we will see that both of $c_{\vr}^{\Z_2}(\Sp^m_\infty,\Sp^n_\infty;0)$ and $2c_{\TS}^{\Z_2}(\Sp^m_\infty,\Sp^n_\infty;0)$ are greater than or equal to $\frac{2}{\sqrt{m+1}}=\mathrm{Sep}^{\Z_2}(\Sp^m_\infty)$ (see the item (3) of \Cref{ex:G-sep}, the item (2) of \Cref{prop:cforsphereslinfty} and its proof). In particular, we cannot guarantee the equality between $c_{\vr}^{\Z_2}(\Sp^m_\infty,\Sp^n_\infty;0)$ and $2c_{\TS}^{\Z_2}(\Sp^m_\infty,\Sp^n_\infty;0)$ via \Cref{prop:strongercGpXYlbddalernatives}.
\end{enumerate}
\end{example}

\subsection{Proof of \Cref{prop:strongercGpXYlbddalernatives}.}

We need a few technical elements --including a $G$-equivariant nerve lemma-- in order to prove \Cref{prop:strongercGpXYlbddalernatives}.

\begin{definition}[$G$-invariant cover]\label{def:Ginvcov}
Assume that a $G$-space $(X,\alpha_X)$ is given. Then, a cover $\mathcal{U}=\{U_\lambda\}_{\lambda\in\Lambda}$ of $X$ is said to be \emph{$G$-invariant} if there is a $G$-action $\alpha_{\Lambda}:G\times\Lambda\rightarrow\Lambda$ on the index set $\Lambda$ such that $U_{\alpha_\Lambda(g,\lambda)}=\alpha_X(g,U_\lambda)$ for all $g\in G$ and $\lambda\in \Lambda$.
\end{definition}

\begin{proposition}[$G$-nerve lemma]\label{prop:Gnerve}
Assume that a finite group $G$, a paracompact $G$-space $(X,\alpha_X)$, and a $G$-invariant open cover $\mathcal{U}:=\{U_\lambda\}_{\lambda\in\Lambda}$ of $X$ are given. Also, assume that $U_\sigma:=\bigcap_{\lambda\in\sigma} U_\lambda$ is $G_\sigma$-contractible for every simplex $\sigma\in\mathrm{N}\,\mathcal{U}$ where $\mathrm{N}\,\mathcal{U}$ is the nerve complex associated to $\mathcal{U}$ and $G_\sigma:=\{g\in G:\alpha_\Lambda(g,\sigma)=\sigma\}$ is the isotropy subgroup of $G$. Then, $X$ and $\vert\mathrm{N}\,\mathcal{U}\vert$ are $G$-homotopy equivalent. 
\end{proposition}

\begin{remark}\label{rmk:trivialcontractibiliy} If $G$ is the trivial group, then $G$-contractibility is just boils down the standard notion of contractibility.
\end{remark}

\begin{remark}
Various versions of the $G$-nerve lemma have been already studied in the following articles. Each version of $G$-nerve lemma holds under  different assumptions all of which are more restrictive than those in \Cref{prop:Gnerve}:

\begin{itemize}
    \item \cite[Lemma 2.5]{hess2013topology} requires the space $X$ to be a $G$-simplicial complex.

    \item \cite[Theorem 2.19]{yang2014equivariant} requires the space $X$ to be $G$-CW complex $X$ and the cover $\mathcal{U}$ to be locally finite and ``regular" (see \cite[Definition 2.7]{yang2014equivariant}).

    \item \cite[Theorem 4.6]{gonzalez2024equivariant} requires the cover $\mathcal{U}$ to be locally finite.
\end{itemize}
\end{remark}

The proof of the above proposition is relegated to \Cref{sec:Gnerve}. As an application, we obtain the following.

\begin{proposition}\label{prop:GvrTS}
Let $G$ be a finite group and $X$ be a $G$-metric space with $\mathrm{Sep}^G(X)>0$. Then, $\vert\vr(X;2r)\vert$ and $B_r(X,\TS(X))$ are $G$-homotopy equivalent for any $r\in \left(0,\frac{\mathrm{Sep}^G(X)}{2}\right]$.
\end{proposition}

\begin{remark}
That for a compact  metric space $(X,d_X)$, $|\vr(X;2r)|$ and $B_r(X,\TS(X))$ are homotopy equivalent \emph{for all $r>0$} is a  known fact; see \cite[Proposition 2.27]{lim2020vietoris}. 
\end{remark}

\begin{example}
    Note, for example, that when $G=\Z_2$ and  $X=\Sp^n$ with the canonical antipodal action,  \Cref{prop:GvrTS}  together with \Cref{ex:G-sep}, gives that $|\vr(\Sp^n;2r)|$ and $B_r(\Sp^n,\TS(\Sp^n))$ are $\Z_2$-homotopy equivalent for all $r\in(0,\tfrac{\pi}{2}]$, which is optimal.
\end{example}

\begin{remark}\label{remark:z3-contract}
Let $G=\Z_3:=\left\{0,\frac{2\pi}{3},\frac{4\pi}{3}\right\}$, $X=\Sp^1$, and the $\Z_3$-action $\alpha_{\Sp^1}$ is defined by $\alpha_{\Sp^1}(g,\theta):=\theta+g$ for any $g\in\Z_3$ and $\theta\in\Sp^1$. Then, it is easy to verify that $\mathrm{Sep}^{\Z_3}(\Sp^1)=\frac{2\pi}{3}$. Hence, by \Cref{prop:GvrTS}, $\vert\vr(\Sp^1;2r)\vert$ and $B_r(\Sp^1,\TS(\Sp^1))$ are $\Z_3$-homotopy equivalent for all $r\in \left(0,\frac{\pi}{3}\right]$.
\end{remark}

\begin{proof}[\textbf{Proof of \Cref{prop:GvrTS}}]

Note that for every $r>0$, $\mathcal{U}:=\{B_r(x,\TS(X))\}_{x\in X}$ is a $G$-invariant open cover of $B_r(X,\TS(X))$. Also, $\mathrm{N}\,\mathcal{U}=\vert\vr(X;2r)\vert$ by \cite[Proposition 2.25]{lim2024vietoris}. Assume  that $0<r\leq \tfrac{1}{2}\mathrm{Sep}^G(X)$. We claim that the isotropy group $G_\sigma$ is  trivial for every $\sigma\in\mathrm{N}\,\mathcal{U}$. Assuming this, since in this case $U_\sigma = \bigcap_{x\in \sigma} B_{r}(x,\TS(X))$ is contractible by \cite[Lemma 2.28]{lim2024vietoris}, it follows immediately from \Cref{prop:Gnerve}  and \Cref{rmk:trivialcontractibiliy} that $B_r(X,\TS(X))$ and $\vert\vr(X;2r)\vert$ are $G$-homotopy equivalent.

We now prove the claim made above. Suppose that there exists an $n$-dimensional simplex $\sigma:=\{x_0,\dots,x_n\}\in\mathrm{N}\,\mathcal{U}$ such that $\alpha_X(g,\sigma)=\sigma$ for some nontrivial $g\in G$. Then, this implies that there exists $f\in B_r(X,\TS(X))$ and $x_i$ such that $f\in B_r(d_X(x_i,\cdot),\TS(X))\cap B_r(d_X(\alpha_X(g,x_i),\cdot),\TS(X))$. By the triangle inequality, we have  $d_X(x_i,\alpha_X(g,x_i))<2r$, which  contradicts the assumption on $r$.
\end{proof}

\begin{question}
Does the $\Z_3$-homotopy equivalence between $\vert\vr(\Sp^1;2r)\vert$ and $B_r(\Sp^1,\TS(\Sp^1))$ alluded to in \Cref{remark:z3-contract} still hold for  scale value $r=\frac{\pi}{3}+\varepsilon$ for some small enough $\varepsilon>0$? In more generality, can the range of $r$ stipulated by \Cref{prop:GvrTS} be improved? A careful inspection of its proof (below) shows that this might be possible through establishing that $U_\sigma$ is $G_\sigma$ contractible even when $G_\sigma$ is not the trivial group.
\end{question}

One implication of the $G$-homotopy equivalence whose  existence is guaranteed by \Cref{prop:GvrTS} is that there are $G$-maps from $|\vr(X;2r)|$ to $B_r(X,\TS(X))$ and vice-versa, when $r\in \left(0,\frac{\mathrm{Sep}^G(X)}{2}\right]$. This fact will be useful in the proof of \Cref{prop:strongercGpXYlbddalernatives}. Moreover, by careful analysis on the proof of \Cref{prop:Gnerve} (see \Cref{sec:Gnerve}), one can conclude that, \emph{for every $r>0$},  there is a one-sided $G$-map from $B_r(X,\TS(X))$ to $\vert\vr(X;2r)\vert$, as \Cref{lemma:onesidedTStoVR} below shows.

\begin{lemma}\label{lemma:onesidedTStoVR}
Let $G$ be a finite group and $X$ be a $G$-metric space. Then, for every $r>0$ there is a $G$-map from $B_r(X,\TS(X))$ to $\vert\vr(X;2r)\vert$.
\end{lemma}
\begin{proof}
Consider the $G$-invariant open cover $\mathcal{U}:=\{B_r(x,\TS(X))\}_{x\in X}$ of $B_r(X,\TS(X))$ as in the proof of \Cref{prop:GvrTS}. Next, let $\{\psi_x\}_{x\in X}$ be a $G$-partition of unity subordinate to $\mathcal{U}$, whose existence is guaranteed by \Cref{lemma:Gptofuty}. Then, one can verify that the following map
$$f\mapsto (\psi_x(f))_{x\in X}$$
is a well defined $G$-map from $B_r(X,\TS(X))$ to $\vert\vr(X;2r)\vert$.
\end{proof}

\begin{lemma}\label{lemma:vrmtovr}
Let $G$ be a finite group, $X$ be a totally bounded $G$-metric space, and $p\in [1,\infty]$. Then, for any $\varepsilon,\delta>0$, there exists a $G$-map $\Theta:\vrm_\infty(X;\varepsilon)\toG\vert\vr(X;\varepsilon+\delta)\vert$.
\end{lemma}
\begin{proof}
By \Cref{lemma:Ginvdeltanetexst}, there exists a $G$-invariant $\frac{\delta}{2}$-net $U$ of $X$. Hence, by \Cref{lemma:G2deltamapexst}, there exists a $(G,\delta)$-map $\Phi:(\Pfin(X),\diam_\infty^X)\toG(\Pfin(U),\diam_\infty^U)$. Furthermore, since $U$ is finite, by \Cref{lemma:homeo}, the map $\mathrm{id}^{-1}:\vrm_\infty(U;\varepsilon+\delta)\to\vert\vr(U;\varepsilon+\delta)\vert$ s.t. $\sum_{i=0}^n u_i\delta_{x_i}\mapsto\sum_{i=0}^n u_ix_i$ is continuous and therefore is a $G$-map. Finally, with the canonical inclusion $\iota:\vert\vr(U;\varepsilon+\delta)\vert\hooktoG\vert\vr(X;\varepsilon+\delta)\vert$, one can build the following $G$-map
$$\Theta:=\iota\circ\mathrm{id}^{-1}\circ\Phi\vert_{\vrm_\infty(X;r)}:\vrm_\infty(X;\varepsilon)\toG\vert\vr(X;\varepsilon+\delta)\vert.$$
\end{proof}

\begin{proof}[\textbf{Proof of \Cref{prop:strongercGpXYlbddalernatives}}]

\begin{enumerate}
    \item 
    
    \begin{enumerate}
    \item The following diagram depicts two maps that we will construct in the proof of this item: 
    \small{$$\begin{tikzcd}
                          &  & \vert\vr(X;\eta)\vert \arrow[r, "\mathrm{id}"] & \vrm_\infty(X;\eta) \\
Y \arrow[rru] \arrow[rrd] &  &                  &   \\
                          &  & \vrm_\infty(X;\eta) \arrow[r, "\Theta"] & \vert\vr(X;\eta+\delta)\vert
\end{tikzcd}$$}
 Fix  arbitrary $\eta>c_\vr^G(X,Y;0)$. Then, there exists a $G$-map $Y\toG\vert\vr(X;\eta)\vert$. Furthermore, the canonical map $\mathrm{id}:\vert\vr(X;\eta)\vert\toG\vrm_\infty(X;\eta)$ given by \Cref{lemma:homeo}  is also a $G$-map. Hence, by composing these two maps, one can produce a $G$-map $Y\toG\vrm_\infty(X;\eta)$. This implies that $\eta\geq c_{\infty}^G(X,Y;0)$. Since the choice of $\eta$ is arbitrary, we have that $c_{\infty}^G(X,Y;0)\leq c_\vr^G(X,Y;0)$. Now, conversely, fix an arbitrary $\eta>c_{\infty}^G(X,Y;0)$. Then, there exists a $G$-map $Y\toG\vrm_\infty(X;\eta)$. Next, fix an arbitrary $\delta>0$. Then, by \Cref{lemma:vrmtovr}, there is a $G$-map $\Theta:\vrm_\infty(X;\eta)\toG\vert\vr(X;\eta+\delta)\vert$. Hence, by composing these two maps one can produce a $G$-map $Y\toG\vert\vr(X;\eta+\delta)\vert$. This implies that $\eta+\delta\geq c_\vr^G(X,Y;0)$. Since the choice of $\eta$ and $\delta$ are arbitrary, we have that $c_{\infty}^G(X,Y;0)\geq c_\vr^G(X,Y;0)$. Therefore,
    $$c_{\infty}^G(X,Y;0)=c_\vr^G(X,Y;0).$$

    \item Fix arbitrary $\varepsilon>0$ and $\eta>c_{\infty}^G(X,Y;\varepsilon)$ and $\delta>0$. We will use the following diagrams in the proof of this item: 
    $$\small{\begin{tikzcd}
\vrm_\infty(Y;\varepsilon) \arrow[r]      & \vrm_\infty(X;\varepsilon+\eta) \arrow[d, "\Theta"] & \vert\vr(Y;\varepsilon+\delta)\vert \arrow[r]      & \vert\vr(X;\varepsilon+\delta+\eta)\vert \arrow[d, "\mathrm{id}"] \\
\vert\vr(Y;\varepsilon)\vert \arrow[u, "\mathrm{id}"] & \vert\vr(X;\varepsilon+\eta+\delta)\vert   & \vrm_\infty(Y;\varepsilon) \arrow[u, "\Theta"] & \vrm_\infty(X;\varepsilon+\delta+\eta)               
\end{tikzcd}}$$

      There exists a $G$-map $\vrm_\infty(Y;\varepsilon)\toG\vrm_\infty(X;\varepsilon+\eta)$. Also, again by \Cref{lemma:homeo} and \Cref{lemma:vrmtovr}, there are $G$-maps $\mathrm{id}:\vert\vr(Y;\varepsilon)\vert\toG\vrm_\infty(Y;\varepsilon)$ and $\Theta:\vrm_\infty(X;\varepsilon+\eta)\toG\vert\vr(X;\varepsilon+\eta+\delta)\vert$. Hence, by composing these three maps one can produce a $G$-map $\vert\vr(Y;\varepsilon)\vert\toG\vert\vr(X;\varepsilon+\eta+\delta)\vert$. This implies that $\eta+\delta\geq c_\vr^G(X,Y;\varepsilon)$. Since the choice of $\eta$ and $\delta$ are arbitrary, we have that $c_\vr^G(X,Y;\varepsilon)\leq c_{\infty}^G(X,Y;\varepsilon)$. Now, conversely, fix arbitrary $\delta>0$ and $\eta>c_\vr^G(X,Y;\varepsilon+\delta)$. Then, there exists a $G$-map $\vert\vr(Y;\varepsilon+\delta)\vert\toG\vert\vr(X;\varepsilon+\delta+\eta)\vert$. Also,  again by \Cref{lemma:homeo} and \Cref{lemma:vrmtovr}, there are $G$-maps $\mathrm{id}:\vert\vr(X;\varepsilon+\delta+\eta)\vert\toG\vrm_\infty(X;\varepsilon+\delta+\eta)$ and $\Theta:\vrm_\infty(Y;\varepsilon)\toG\vert\vr(Y;\varepsilon+\delta)\vert$. Hence, by composing these three maps one can produce a $G$-map $\vrm_\infty(Y;\varepsilon)\toG\vrm_\infty(X;\varepsilon+\delta+\eta)$. This implies that $\eta+\delta\geq c_{\infty}^G(X,Y;\varepsilon)$. Since the choice of $\eta$ is arbitrary, we have that $c_{\infty}^G(X,Y;\varepsilon)\leq c_\vr^G(X,Y;\varepsilon+\delta)+\delta$. Therefore,
    $$c_{\vr}^G(X,Y;\varepsilon)\leq c_{\infty}^G(X,Y;\varepsilon)\leq\inf_{\delta>0}\big\{c_{\vr}^G(X,Y;\varepsilon+\delta)+\delta\big\}.$$

    \end{enumerate}

    \item
    
    \begin{enumerate}
    \item We will use the following diagram in the proof of this item: 
    $$\begin{tikzcd}
                          &  & \vert\vr(X;2\eta)\vert \arrow[dd, shift left=2] \\
Y \arrow[rru] \arrow[rrd] &  &                            \\
                          &  & B_\eta(X,\TS(X)) \arrow[uu, shift left=2]
\end{tikzcd}$$ 
     Fix an arbitrary $\eta> c_\TS^G(X,Y;0)$. Then, there exists a $G$-map $Y\toG B_\eta(X,\TS(X))$. Also, by \Cref{lemma:onesidedTStoVR}, we know that there is a $G$-map from $B_\eta(X,\TS(X))$ to $\vert\vr(X;2\eta)\vert$. Hence, by composing these two maps one can produce a $G$-map $Y\toG \vert\vr(X;2\eta)\vert$. This implies that $2\eta\geq c_\vr^G(X,Y;0)$. Since the choice of $\eta$ is arbitrary, we have that $c_\vr^G(X,Y;0)\leq 2c_\TS^G(X,Y;0)$. Next, assume that $\mathrm{Sep}^G(X)>0$. Consider the case when $c_\vr^G(X,Y;0)<\mathrm{Sep}^G(X)$. Fix an arbitrary $\eta\in\left(\frac{c_\vr^G(X,Y;0)}{2},\frac{\mathrm{Sep}^G(X)}{2}\right]$. Then, there exists a $G$-map $Y\toG\vert\vr(X;2\eta)\vert$. Also, by \Cref{prop:GvrTS}, $\vert\vr(X;2\eta)\vert$ and $B_\eta(X,\TS(X))$ are $G$-homotopy equivalent. In particular, there is a $G$-map from $\vert\vr(X;2\eta)\vert$ to $B_\eta(X,\TS(X))$. Hence, by composing these two maps one can produce a $G$-map $Y\toG B_\eta(X,\TS(X))$. This implies that $\eta\geq c_\TS^G(X,Y;0)$. Since the choice of $\eta$ is arbitrary, we have that $c_\vr^G(X,Y;0)\geq 2c_\TS^G(X,Y;0)$. Therefore,
    $$c_\vr^G(X,Y;0)= 2c_\TS^G(X,Y;0).$$
    Since the proof of the case when $2c_\TS^G(X,Y;0)<\mathrm{Sep}^G(X)$ is similar, we omit it.

    \item The following diagram will be used in the proof of this item:
    $$\begin{tikzcd}
\vert\vr(Y;2\varepsilon)\vert \arrow[r] \arrow[d, shift left=2] & \vert\vr(X;2\varepsilon+2\eta)\vert \arrow[d, shift left=2] \\
B_\varepsilon(Y,\TS(X)) \arrow[r] \arrow[u, shift left=2] & B_{\varepsilon+\eta}(X,\TS(X)) \arrow[u, shift left=2]
\end{tikzcd}$$
    
    Fix an arbitrary $\varepsilon\in\left(0,\frac{\mathrm{Sep}^G(Y)}{2}\right]$. Assume that $c_\vr^G(X,Y;2\varepsilon)<\mathrm{Sep}^G(X)-2\varepsilon$. Fix   $\eta\in\left(\frac{c_\vr^G(X,Y;2\varepsilon)}{2},\frac{\mathrm{Sep}^G(X)}{2}-\varepsilon\right]$. Then, there exists a $G$-map $\vert\vr(Y;2\varepsilon)\vert\toG\vert\vr(X;2\varepsilon+2\eta)\vert$.  Also, again by \Cref{prop:GvrTS}, there are $G$-maps from $B_{\varepsilon}(Y,\TS(Y))$ to $\vert\vr(Y;2\varepsilon)\vert$ and from $\vert\vr(X;2\varepsilon+2\eta)\vert$ to $B_{\varepsilon+\eta}(X,\TS(X))$. Hence, by composing these three maps one can produce a $G$-map $B_{\varepsilon}(Y,\TS(Y))\toG B_{\varepsilon+\eta}(X,\TS(X))$. This implies that $\eta\geq c_\TS^G\left(X,Y;\varepsilon\right)$. Since the choice of $\eta$ is arbitrary, we have that $c_\vr^G(X,Y;2\varepsilon)\geq 2c_\TS^G\left(X,Y;\varepsilon\right)$. 
    
    Conversely, fix   $\eta\in\left(c_\TS^G\left(X,Y;\varepsilon\right),\frac{\mathrm{Sep}^G(X)}{2}-\varepsilon\right]$. Then, there exists a $G$-map $$B_{\varepsilon}(Y,\TS(Y))\toG B_{\varepsilon+\eta}(X,\TS(X)).$$ Also, again by \Cref{prop:GvrTS}, there are $G$-maps from $\vert\vr(Y;2\varepsilon)\vert$ to $B_\varepsilon(Y,\TS(Y))$ and from $B_{\varepsilon+\eta}(X,\TS(X))$ to $\vert\vr(X;2\varepsilon+2\eta)\vert$ . Hence, by composing these three maps one can produce a $G$-map $\vert\vr(Y;2\varepsilon)\vert\toG \vert\vr(X;2\varepsilon+2\eta)\vert$. This implies that $2\eta\geq c_\vr^G(X,Y;2\varepsilon)$. Since the choice of $\eta$ is arbitrary, we have that $c_\vr^G(X,Y;2\varepsilon)\leq 2c_\TS^G\left(X,Y;\varepsilon\right)$. Therefore,
    $$c_\vr^G(X,Y;2\varepsilon)= 2c_\TS^G\left(X,Y;\varepsilon\right).$$
    Since the proof of the case  $2c_\TS^G\left(X,Y;\varepsilon\right)<\mathrm{Sep}^G(X)-2\varepsilon$ is similar, we omit it.
    \end{enumerate}
\end{enumerate}
\end{proof}

\subsection{Bounding  modulus of discontinuity via  $G$-persistent indices.}\label{sec:qbu}

Let $X$ be a metric space, $Y$ be a topological space, and $\psi:Y\longrightarrow X$ be a function. Importantly, $\psi$ is not necessarily continuous. Then, the \emph{modulus of discontinuity} $\delta(\psi)$ of a function $\psi$ is the infimum of all numbers $\delta>0$ such that every point in $Y$ has a neighborhood whose image has a diameter of at most $\delta$. This quantity $\delta(\psi)$ measures how much discontinuous $\psi$ is. In particular, $\delta(\psi)=0$ if and only if $\psi$ is continuous.

In \cite{dubins1981equidiscontinuity}, Dubins and Schwarz proved the following ``quantitative" version of the Borsuk-Ulam theorem:

\begin{quote}
\emph{If $0 < m < n <\infty$, then  any $\Z_2$-function $\psi$ from $\Sp^n$ to $\Sp^m$ (that is necessarily discontinuous by the classical Borsuk-Ulam theorem) must have  $\delta(\psi)\geq \zeta_m$.}
\end{quote}

This result was later employed in \cite{lim2023gromov} to prove that $2\,\dgh(\Sp^m,\Sp^n) \geq \zeta_m$ for all $n > m$.  Motivated by these ideas, a recent polymath project (cf. \cite[Theorem 1.3]{adams2022gromov}) established the following strengthened version of the quantitative Borsuk–Ulam theorem:

\begin{quote}
\emph{If $0 < m < n <\infty$, then any $\Z_2$-function $\psi$ from $\Sp^{n}$ to $\Sp^m$ must have  $$\delta(\psi)\geq c_{m,n}:=\inf\{\eta>0:\text{there is a }\Z_2\text{-map }\Sp^n\stackrel{\Z_2}{\longrightarrow} \vert\vr(\Sp^m;\eta)\vert\}.$$}
\end{quote}

Note that $c_{m,n}$ is equal to $c_{\vr}^{\Z^2}(\Sp^m,\Sp^n;0)$; see \Cref{def:clowerbdd}. Recalling that, by item 1(a) of Proposition \ref{prop:strongercGpXYlbddalernatives}, for an arbitrary finite group $G$ and  arbitrary compact $G$-metric spaces $X$ and $Y$, $c_{\infty}^G(X,Y;0)=c_{\vr}^G(X,Y;0)$, we  obtain  the following generalization to the $G$-equivariant version of the  quantitative Borsuk-Ulam  theorem recalled above.

\begin{proposition}[$G$-equivariant quantitative Borsuk-Ulam theorem]\label{prop:GquantBU}
Assume that a finite group $G$, a $G$-metric space $X$, a compact $G$-metric space $Y$, and a $G$-function $\psi:Y\toG X$ are given. Then, we have that $\delta(\psi)\geq c_\infty^G(X,Y;0)$.
\end{proposition}

\begin{remark}
A Vietoris–Rips-based version of \Cref{lemma:vrmmodcti} was obtained by Matt Superdock as early as August 2022, as part of the polymath project \cite{adams2022gromov}. Together with another technical ingredient established by Superdock, this lemma can be invoked to obtain a Vietoris–Rips-based version of \Cref{prop:GquantBU}, i.e., one that uses the persistent index $c_\vr^G(X,Y;0)$ in place of $c_\infty^G(X,Y;0)$. By item 1 (a) of \Cref{prop:strongercGpXYlbddalernatives}, this Vietoris–Rips-based formulation is equivalent to \Cref{prop:GquantBU}. See the complementary paper \cite{adams2025quantifying} for a detailed study of geometric applications of quantitative versions of the Borsuk–Ulam theorem, including a quantitative version of the topological Tverberg theorem.
\end{remark}

The following lemma is a generalization of \cite[Lemma 7.4]{adams2022gromov} to the $G$-equivariant setting.

\begin{lemma}\label{lemma:vrmmodcti}
Assume that a finite group $G$, a $G$-metric space $X$, a compact $G$-metric space $Y$, and a $G$-function $\psi:Y\toG X$ are given. Then, for any $\eta>0$ there exists $\alpha_\eta>0$ such that there is a $G$-map
$$\widetilde{\psi}_B:\vrm_\infty(B;\alpha_\eta)\toG\vrm_\infty(X;\delta(\psi)+\eta)$$
for any finite $G$-invariant subset $B \subseteq Y$.
\end{lemma}

Now we are ready to prove \Cref{prop:GquantBU}.

\begin{proof}[Proof of \Cref{prop:GquantBU}]
Fix an arbitrary $\eta>0$ and choose $\alpha_\eta$ as in \Cref{lemma:vrmmodcti}. By \Cref{lemma:Ginvdeltanetexst}, one can choose a finite $G$-invariant $\frac{\alpha_\eta}{2}$-net $B$ of $Y$. Then, by \Cref{lemma:G2deltamapexst}, we have a $G$-map $\iota:Y\toG\vrm_\infty(B;\alpha_\eta)$. Hence, the composition $\widetilde{\psi}_B\circ\iota$ is a $G$-map from $Y$ to $\vrm_\infty(X;\delta(\psi)+\eta)$. Therefore, we have that $c_\infty^G(X,Y;0)\leq\delta(\psi)+\eta$. Since the choice of $\eta$ is arbitrary, one can conclude that $c_\infty^G(X,Y;0)\leq\delta(\psi)$.
\end{proof}

\begin{proof}[Proof of \Cref{lemma:vrmmodcti}]
Fix an arbitrary $\eta>0$. Then, following the proof of \cite[Lemma 7.4]{adams2022gromov}, we know that there exists $\alpha_\eta>0$ such that $d_X(\psi(y),\psi(y'))<\delta(\psi)+\eta$ for any $y,y'\in Y$ with $d_Y(y,y')<\alpha_\eta$. We write the argument for the completeness. By the definition of the modulus of discontinuity, for each $y\in Y$ we have $r_y>0$ such that the open ball $B_{r_y}(y)$ satisfies $\diam\big(\psi(B_{r_y}(y))\big)<\delta(\psi)+\eta$. Then, $\{B_{r_y/2}(y)\}_{y\in Y}$ is an open covering of $Y$. Since $Y$ is compact, one can choose finitely many points $y_1,\dots,y_k$ such that $\{B_{r_{y_i}/2}(y_i)\}_{i=1}^k$ covers whole $Y$. Finally, we choose $\alpha_\eta:=\min\left\{\frac{r_{y_1}}{2},\dots,\frac{r_{y_k}}{2}\right\}$. Now, choose arbitrary two points $y,y'\in Y$ with $d_Y(y,y')<\alpha_\eta$. Then, there exists $y_i\in Y$ such that $y\in B_{r_{y_i}/2}(y_i)$. Hence, by the triangle inequality we have that $d_Y(y_i,y')\leq d_Y(y_i,y)+d_Y(y,y')\leq\frac{r_{y_i}}{2}+\alpha_\eta\leq r_{y_i}$. This implies that $y,y'\in B_{r_{y_i}}(y_i)$
so that $d_X(\psi(y),\psi(y'))<\delta(\psi)+\eta$ as we wanted.

Finally, for any finite $G$-invariant subset  $B\subseteq Y$, one can define
$$\widetilde{\psi}_B:\vrm_\infty(B;\alpha_\eta)\toG\vrm_\infty(X;\delta(\psi)+\eta)$$
as the composition of the following maps:
{\small$$\vrm_\infty(B;\alpha_\eta)\stackrel{\mathrm{id}^{-1}}{\longrightarrow}\vert\vr(B;\alpha_\eta)\vert\longrightarrow\vert\vr(\psi(B);\delta(\psi)+\eta)\vert\stackrel{\mathrm{id}}{\longrightarrow}\vrm_\infty(\psi(B);\delta(\psi)+\eta)\longhookrightarrow\vrm_\infty(X;\delta(\psi)+\eta)$$}
where the first and the third maps are the canonical identity maps in \Cref{lemma:homeo}, the fourth map is the canonical inclusion, and the second map is the canonical map induced by $\psi$. Note that $\widetilde{\psi}_B$ is well-defined and continuous because of the above choice of $\alpha_\eta$ and the finiteness of $B$.
\end{proof}


\section{Applications to determining $\dgh^{\Z_2}$ 
 between spheres.}\label{sec:GdGHresults-sph}

In this section we apply the material developed in previous sections to estimate and sometimes fully determining the value of the $\Z_2$-Gromov-Hausdorff distance between spheres.

\subsection{Results for $\Z_2$-persistence indices of spheres.}\label{sec:Spnresults}

Recall that $\Sp^n,\SpE^n,\Sp^n_\infty$, and $\square^n_\infty$ are $\Z_2$-metric spaces with  canonical $\Z_2$-actions; see items (2) and (3) of \Cref{ex:spheres-Z2}. Then, one can obtain the lower bounds below for the following $\Z_2$-persistence indices  (see \Cref{def:clowerbdd}):

\begin{proposition}\label{prop:cforspheres}
For all $0<m<n<\infty$, we have the following
\begin{enumerate}
    \item $c_{\infty}^{\Z_2}(\Sp^m,\Sp^n;0)\geq\zeta_m$. 

    \item $c_{\infty}^{\Z_2}(\SpE^m,\SpE^n;0)\geq\sqrt{\frac{2m+4}{m+1}}$.
\end{enumerate}
\end{proposition}

\begin{remark}\label{rmk:cinftyZ2SmSm+1}
When $n=m+1$, the inequality in item (1) of \Cref{prop:cforspheres} is actually tight. Indeed, by \Cref{prop:strongercGpXYlbddalernatives} and \cite[Theorem 5.2]{adams2022gromov}, we have that $c_{\infty}^{\Z_2}(\Sp^m,\Sp^{m+1};0)=\zeta_m$.
\end{remark}

\begin{remark} Consider odd dimensional spheres $\Sp^m$ and $\Sp^n$ equipped with their canonical $\Sp^1$-action (see the item (2b) of Example \ref{ex:spheres-Z2}). Then, by Remark \ref{rmk:GGtildepindex}, Proposition \ref{prop:strongercGpXYlbdd}, and Proposition \ref{prop:cforspheres}, one can conclude that
$$\dht^{\Sp^1}\big((\Pfin(\Sp^m),\diam_\infty^{\Sp^m}),(\Pfin(\Sp^n),\diam_\infty^{\Sp^n})\big)\geq c_{\infty}^{\Sp^1}(\Sp^m,\Sp^n;0) \geq c_{\infty}^{\Z_2}(\Sp^m,\Sp^n;0)\geq\zeta_m.$$

Note that, since $\Sp^1$ is not finite, Theorem \ref{Gdhtstability} cannot be applied to establish $$2\,\dgh^{\Sp^1}(\Sp^m,\Sp^n)\geq\dht^{\Sp^1}\big((\Pfin(\Sp^m),\diam_\infty^{\Sp^m}),(\Pfin(\Sp^n),\diam_\infty^{\Sp^n})\big)\geq\zeta_m.$$
However, for the special case when $m=1$ and $n=3$, we have that
$\dgh^{\Sp^1}(\Sp^1,\Sp^3)\geq\dgh(\Sp^1,\Sp^3)=\frac{\pi}{3}$ by  item (3) of \Cref{rmk:G-dGH} and \cite[Proposition 1.18]{lim2023gromov}. Actually, one can verify that $\dgh^{\Sp^1}(\Sp^1,\Sp^3)=\frac{\pi}{3}$ since the map $\phi_{3,1}:\Sp^3\twoheadrightarrow\Sp^1$ employed in the proof of \cite[Proposition 1.18]{lim2023gromov} is indeed a surjective $\Sp^1$-function with $\dis(\phi_{3,1})=\frac{2\pi}{3}$.
\end{remark}

\begin{proposition}\label{prop:cforsphereslinfty}
For all $0<m<n<\infty$, we have the following
\begin{enumerate}
    \item $c_{\infty}^{\Z_2}(\square^m_\infty,\square^n_\infty;0)\geq 2$. 

    \item $c_{\infty}^{\Z_2}(\Sp^m_\infty,\Sp^n_\infty;0)\geq\frac{2}{\sqrt{m+1}}$.
\end{enumerate}
\end{proposition}

\begin{remark}
    The inequality in item (1) of \Cref{prop:cforsphereslinfty} is an equality for all $n>m$ as it will follow from \Cref{thm:Z2dGHsqmsqn}.
\end{remark}

\subsubsection{Proofs of \Cref{prop:cforspheres,prop:cforsphereslinfty}.}

As a special case of \cite[Theorem 3.5]{hausmann1994vietoris}, it is known that there is a homotopy equivalence $T_r:\vert\vr(\Sp^n;r)\vert\rightarrow \Sp^n$ for all $r\in (0,\frac{\pi}{2}]$. Also, in \cite[Theorem 7.1]{lim2024vietoris}, we were able to improve  Hausmann's result by proving the existence of a homotopy equivalence $T_r:\vert\vr(\Sp^n;r)\vert\rightarrow \Sp^n$ for all $r\in (0,\zeta_n]$ (observe that $\zeta_n>\frac{\pi}{2}$).\footnote{The range $(0,\zeta_n]$ is tight for $n=1,2$ and it is believed to be tight in general; see \cite[Conjecture 7.8]{lim2024vietoris}.} Furthermore, one can make this result even stronger in the sense that $T_r$ can be chosen to be a \emph{$\Z_2$-homotopy equivalence}.

\begin{theorem}[$\Z_2$-equivariant Hausmann theorem for Vietoris-Rips complexes of spheres]\label{thm:Z2hausmannsphere}
For any $n\in\mathbb{Z}_{>0}$ and $r\in (0,\zeta_n]$, there is a $\Z_2$-homotopy equivalence
$T_r:\vert\vr(\Sp^n;r)\vert\toZtwo\Sp^n.$
\end{theorem}

The proof of this theorem is given in \Cref{app:proof-Z2-H}. Also, from \Cref{thm:vrmHausmann-G}, for the case when $M=\Sp^n$ and $G=\Z_2$, one can obtain an analogous result for the Vietoris-Rips metric thickening $\vrm_\infty(\Sp^n;r)$. However, since $r(\Sp^n)=\tfrac{\pi}{2}$, the range of values of $r$ such that $\vrm_\infty(\Sp^n;r)\stackrel{\Z_2}{\simeq}\Sp^n$ is smaller than the one guaranteed by \Cref{thm:Z2hausmannsphere} for the standard Vietoris-Rips filtration $\vr(\Sp^n;r).$ The following statement, giving a $\Z_2$-equivariant version of \cite[Proposition 5.3]{adamaszek2018metric}, provides an improved range of values of $r$.

\begin{theorem}[$\Z_2$-equivariant Hausmann theorem for Vietoris-Rips thickenings of spheres]\label{thm:Z2hausmannspherevrm}
For any $n\in\mathbb{Z}_{>0}$ and $r\in (0,\zeta_n]$, there is a $\Z_2$-homotopy equivalence
$p_r^n:\vrm_\infty(\Sp^n;r)\toZtwo\Sp^n.$
\end{theorem}
\begin{proof}
In \cite[Proposition 5.3]{adamaszek2018metric}, by invoking Jung's Theorem (\Cref{thm:Jung}), the authors proved that the following map is well-defined and a homotopy equivalence in the given  range of $r$:
\begin{align*}
    p_r^n:\vrm_\infty(\Sp^n;r)&\longrightarrow\Sp^n\\
    \sum_{i=0}^m u_i\delta_{x_i}&\longmapsto\frac{\sum_{i=0}^m u_i x_i}{\Vert \sum_{i=0}^m u_i x_i \Vert}.
\end{align*}
Also, the canonical inclusion $j_r:\Sp^n\longrightarrow\vrm_\infty(\Sp^n;r)$ s.t. $x\longmapsto\delta_x$ is a homotopy inverse of $p_r^n$ that satisfies (1) $p_r^n\circ j_r=\mathrm{id}_{\Sp^n}$ and (2) $j_r\circ p_r^n$ is homotopic to $\mathrm{id}_{\vrm_\infty(\Sp^n;r)}$. Furthermore, the homotopy between $j_r\circ p_r^n$ and $\mathrm{id}_{\vrm_\infty(\Sp^n;r)}$ is given by the following linear interpolation.
\begin{align*}
    H:\vrm_\infty(\Sp^n;r)\times [0,1]&\longrightarrow\vrm_\infty(\Sp^n;r)\\
    (\mu,t)&\longmapsto (1-t)\mu+tj_r\circ p_r^n(\mu).
\end{align*}

Finally, it is easy to verify that all of $p_r^n$, $j_r$, and $H$ are indeed a $\Z_2$-maps.
\end{proof}

\begin{remark}
Note that, in principle, \Cref{thm:Z2hausmannspherevrm} is a special case of \Cref{thm:vrmHausmann-G}. However, $r(\Sp^n)=\frac{\pi}{2}$ so that \Cref{thm:vrmHausmann-G} does not give the desired range of values of $r$. The construction of the map $p_r^n$ in the proof of \Cref{thm:Z2hausmannspherevrm} differs from that of the one used in the proof of \Cref{thm:vrmHausmann-G}. Indeed, $p_r^n$ is defined as a ``projected" euclidean center of mass, as opposed to the Riemannian center of mass construction invoked in the proof of  \Cref{thm:vrmHausmann-G}.
\end{remark}

\begin{proof}[\textbf{Proof of \Cref{prop:cforspheres}}]
\begin{enumerate}
    \item Apply item (1) of \Cref{prop:strongercGpXYlbddalernatives}, \Cref{thm:Z2hausmannsphere}, and the classical Borsuk-Ulam theorem.

    \item By item (1) of \Cref{prop:strongercGpXYlbddalernatives}, it is enough to show that $c_{\vr}^{\Z_2}(\SpE^m,\SpE^n;0)\geq\sqrt{\frac{2m+4}{m+1}}$. By \Cref{thm:Z2hausmannsphere}, there is a $\Z_2$-homotopy equivalence from $\vert\vr_r(\Sp_{\mathrm{E}}^m)\vert$ to $\Sp_{\mathrm{E}}^m$ for any $r\in \left(0,2\,\sin\left(\frac{\zeta_m}{2}\right)\right]$. Hence, if there is a $\Z_2$-map from $\Sp_\mathrm{E}^{n}$ to $\vert\vr_r(\Sp_\mathrm{E}^{m})\vert$ for some $r>0$, then $r$ must satisfy $r\geq 2\,\sin\left(\frac{\zeta_m}{2}\right)=\sqrt{\frac{2m+4}{m+1}}$ since otherwise there exists a $\Z_2$-map from $\Sp_\mathrm{E}^{n}$ to $\Sp_\mathrm{E}^m$ which contradicts to the Borsuk-Ulam Theorem. Hence, the proof is complete.    
\end{enumerate}
\end{proof}

\begin{proof}[\textbf{Proof of \Cref{prop:cforsphereslinfty}}]
\begin{enumerate}
\item By  item (1) of \Cref{prop:strongercGpXYlbddalernatives}, it is enough to show that $$c_\vr^{\Z_2}(\square^m_\infty,\square^n_\infty;0)\geq 2.$$ Suppose not. Then, since $2=\mathrm{Sep}^{\Z_2}(\square^m_\infty)$ (see  item (2) of \Cref{ex:G-sep}), by item (2) of \Cref{prop:strongercGpXYlbddalernatives} we have that $c_\vr^{\Z_2}(\square^m_\infty,\square^n_\infty;0)=2c_\TS^{\Z_2}(\square^m_\infty,\square^n_\infty;0)<2$. Then, there exists a $\Z_2$-map from $\square^n_\infty$ to $B_r(\square^m_\infty,\TS(\square^m_\infty))$ for some $r<1$. Note that, by \cite[Lemma 7.16]{lim2024vietoris}, $\TS(\square^m_\infty)=\blacksquare^{m+1}_\infty$ and
    $$B_r(\square^m_\infty,\blacksquare^{m+1}_\infty)=[-1,1]^{m+1}\backslash [-(1-r),1-r]^{m+1}.$$
    Hence, $B_r(\square^m_\infty,\blacksquare^{m+1}_\infty)$ is $\Z_2$-homotopy equivalent to $\Sp^m$.\footnote{First, one easily sees that $B_r(\square^m_\infty,\blacksquare^{m+1}_\infty)$ is $\Z_2$-homotopy equivalent to $\square^m_\infty$ via the canonical projection map. Then, $\square^m_\infty$ is itself $Z_2$-homotopy equivalent to $\Sp^m$ via the radial projection.} Therefore, there is a $\Z_2$-map from $\Sp^n$ to $\Sp^m$ which contradicts the classical Borsuk-Ulam theorem. This proves the claim.

\item The following diagram describes maps that will be constructed in the proof of this item:
\begin{center}
    \begin{tikzcd}\label{fig:Spinftypf}
 \Sp^n_\infty \arrow[r] & B_r(\Sp^m_\infty,\TS(\Sp^m_\infty)) \arrow[r] & B_r(\Sp^m_\infty,\R^{m+1}_\infty) \arrow[r] & \Sp^m_\infty.
\end{tikzcd}
\end{center}
By item (1) of \Cref{prop:strongercGpXYlbddalernatives}, it is enough to show that $c_\vr^{\Z_2}(\Sp^m_\infty,\Sp^n_\infty;0)\geq \frac{2}{\sqrt{m+1}}$. Suppose not. Then, since $\frac{2}{\sqrt{m+1}}=\mathrm{Sep}^{\Z_2}(\Sp^m_\infty)$ (see item (3) of \Cref{ex:G-sep}), by item (2) of \Cref{prop:strongercGpXYlbddalernatives} we have that $c_\vr^{\Z_2}(\Sp^m_\infty,\Sp^n_\infty;0)=2c_\TS^{\Z_2}(\Sp^m_\infty,\Sp^n_\infty;0)<\frac{2}{\sqrt{m+1}}$. Then, there exists a $\Z_2$-map from $\Sp^n_\infty$ to $B_r(\Sp^m_\infty,\TS(\Sp^m_\infty))$ for some $0<r<\frac{1}{\sqrt{m+1}}$. Now, we claim that there exists a $\Z_2$-map from $B_r(\Sp^m_\infty,\TS(\Sp^m_\infty))$ to $B_r(\Sp^m_\infty,\R^{m+1}_\infty)$. Since $\R^{m+1}_\infty$ is injective in the category of metric spaces (see \cite[pg. 303]{lang2013injective}), there exists a $1$-Lipschitz map from $\TS(\Sp^m_\infty)$ to $\R^{m+1}_\infty$. More precisely, it is well-known (see \cite[pg.301]{lang2013injective}) that this $1$-Lipschitz map can be given explicitly as
\small{$${f\longmapsto\left(\sup\limits_{(x_0,x_1,\dots,x_m)\in\Sp^m_\infty}(x_0-f(x_0,x_1,\dots,x_m)),,\dots,\sup\limits_{(x_0,x_1,\dots,x_m)\in\Sp^m_\infty}(x_m-f(x_0,x_1,\dots,x_m))\right).}$$}
It is easy to verify that this map is indeed $\Z_2$-map. Hence, there exists a $\Z_2$-map from $B_r(\Sp^m_\infty,\TS(\Sp^m_\infty))$ to $B_r(\Sp^m_\infty,\R^{m+1}_\infty)$ as we claimed. Furthermore, $B_r(\Sp^m_\infty,\R^{m+1}_\infty)$ is $\Z_2$-homotopy equivalent to $\Sp^m$ by the map given in \cite[Theorem 7.19]{lim2024vietoris}. Finally, by doing the composition of these three maps (see Figure \ref{fig:Spinftypf}) one can establish a $\Z_2$-map from $\Sp^n$ to $\Sp^m$ which contradicts the classical Borsuk-Ulam theorem. This completes the proof.
\end{enumerate}
\end{proof}


\subsection{Determining $\dgh^{\Z_2}(\Sp^m,\Sp^n)$,  $\dgh^{\Z_2}(\SpE^m,\SpE^n)$ and $\dgh^{\Z_2}(\SpE^m,\SpE^n)$.}\label{sec:GdGHresults}
In this section we apply the techniques developed in previous sections to the case of $\dgh^{\Z_2}(\Sp^m,\Sp^n)$ (\Cref{sec:precise-dgh-sph}), the case of $\dgh^{\Z_2}(\SpE^m,\SpE^n)$ (\Cref{sec:Z2dGHEsphere}), and the case of $\dgh^{\Z_2}(\square^m_\infty,\square^n_\infty)$ (\Cref{sec:Z2dGHinftysphere}).

\subsubsection{The case of $\dgh^{\Z_2}(\Sp^m,\Sp^n)$.}\label{sec:precise-dgh-sph}

We first prove that the $\Z^2$-Gromov-Hausdorff distance between spheres endowed with their geodesic distance and canonical involution is \emph{equal} to the standard  Gromov-Hausdorff.

\begin{proposition}\label{prop:Z2dGHSmSn}
For any $0\leq m<n<\infty$, we have that $$\dgh^{\Z_2}(\Sp^m,\Sp^n)=\dgh(\Sp^m,\Sp^n).$$
\end{proposition}

\begin{remark}\label{rmk:recovpolymath}
Note that one can recover \cite[Main Theorem]{adams2022gromov} by applying \Cref{Gdhtstability},  \Cref{prop:strongercGpXYlbddalernatives}, \Cref{prop:strongercGpXYlbdd}, and \Cref{prop:Z2dGHSmSn}.
\end{remark}

The following is an immediate generalization of the ``helmet trick" from \cite{lim2023gromov}; see \cite[Proposition 2.2]{martin2024some} for a proof.

\begin{lemma}\label{lemma:antipodalcorrespondence}
Suppose $0\leq m<n<\infty$. Then, for any correspondence $R\in\mathcal{R}(\Sp^m,\Sp^n)$, there is a $\Z_2$-correspondence $\widetilde{R}\in\mathcal{R}_{\Z_2}(\Sp^m,\Sp^n)$ such that $\dis(\widetilde{R})=\dis(R)$.
\end{lemma}

\begin{proof}[\textbf{Proof of \cref{prop:Z2dGHSmSn}}]
By  item (3) of \Cref{rmk:G-dGH}, it is enough to verify that $\dgh^{\Z_2}(\Sp^m,\Sp^n)\leq\dgh(\Sp^m,\Sp^n)$. Choose an arbitrary $R\in\mathcal{R}(\Sp^m,\Sp^n)$. Then, by Lemma \ref{lemma:antipodalcorrespondence}, there is a $\Z_2$-correspondence $\widetilde{R}\in\mathcal{R}_{\Z_2}(\Sp^m,\Sp^n)$ such that $\dis(\widetilde{R})=\dis(R)$. Since, the choice of $R$ is arbitrary, this implies that $\dgh^{\Z_2}(\Sp^m,\Sp^n)\leq\dgh(\Sp^m,\Sp^n)$ as we required.
\end{proof}

Therefore, by the Main Theorem (see pg.\pageref{page:main-theorem}) and \Cref{prop:Z2dGHSmSn}, the following inequalities hold:
\begin{align*}
    2\dgh(\Sp^m,\Sp^n)=2\dgh^{\Z^2}(\Sp^m,\Sp^n)&\geq\dht^{\Z^2}\big((\Pfin(\Sp^m),\diam_p^{\Sp^m}),(\Pfin(\Sp^n),\diam_p^{\Sp^n})\big)\\
    &\geq \di^{\Z^2}\big(\h_k(\vrm_p(\Sp^m;\bullet)),\h_k(\vrm_p(\Sp^n;\bullet))\big).
\end{align*}

However, the following corollary and remark show that the topological lower bound

$\dht^{\Z^2}\big((\Pfin(\Sp^m),\diam_p^{\Sp^m}),(\Pfin(\Sp^n),\diam_p^{\Sp^n})\big)$ is \emph{strictly better} than the algebraic lower bound $\di^{\Z^2}\big(\h_k(\vrm_p(\Sp^m;\bullet)),\h_k(\vrm_p(\Sp^n;\bullet))\big)$.

\begin{corollary}\label{cor:SmSnlbddvr}
For all $0<m<n<\infty$, we have the following
$$\dht^{\Z_2}\big((\Pfin(\Sp^m),\diam_\infty^{\Sp^m}),(\Pfin(\Sp^n),\diam_\infty^{\Sp^n})\big)\geq c_\infty^{\Z_2}(\Sp^m,\Sp^n;0)\geq\zeta_m.$$
\end{corollary}
\begin{proof}
Apply \Cref{prop:strongercGpXYlbdd} for $p=\infty$ and  item (1) of \Cref{prop:cforspheres}.
\end{proof}

\begin{remark}\label{rmk:vrminftysmsn}
By \Cref{cor:vrvrmdIGsame} and \Cref{prop:Gdnmatter}, we have that
$$\di^{\Z^2}\big(\h_k(\vrm_\infty(\Sp^m;\bullet)),\h_k(\vrm_\infty(\Sp^n;\bullet))\big)=\frac{\zeta_m}{2}.$$
\end{remark}

\subsubsection{The case of $\dgh^{\Z_2}(\Sp_\mathrm{E}^m,\Sp_\mathrm{E}^n)$.}\label{sec:Z2dGHEsphere}

Recall that $\Sp_{\mathrm{E}}^m$ denotes the $m$-sphere equipped with the Euclidean metric inherited from $\mathbb{R}^{m+1}$ (see page~\pageref{notations}). In \cite[Section 9]{lim2023gromov}, we computed several nontrivial bounds for the Gromov–Hausdorff distance $\dgh(\Sp_{\mathrm{E}}^m, \Sp_{\mathrm{E}}^n)$ for arbitrary dimensions $0 \leq m < n \leq \infty$. However, unlike the case of spheres equipped with the geodesic metric, we were unable to determine the exact values of $\dgh(\Sp_{\mathrm{E}}^1, \Sp_{\mathrm{E}}^2)$, $\dgh(\Sp_{\mathrm{E}}^1, \Sp_{\mathrm{E}}^3)$, and $\dgh(\Sp_{\mathrm{E}}^2, \Sp_{\mathrm{E}}^3)$. This was primarily due to the difficulty of establishing tight lower bounds for $\dgh(\Sp_{\mathrm{E}}^m, \Sp_{\mathrm{E}}^n)$—in contrast to the geodesic case $\dgh(\Sp^m, \Sp^n)$, where exact values could be computed (see~\cite[Theorem B]{lim2023gromov}) using the helmet trick described above. In other words, the helmet trick fundamentally relies on the structure of the geodesic distance on spheres and is therefore inapplicable when the spheres are equipped with their Euclidean metrics.

\medskip
In this section, however, we successfully compute the exact value of the $\Z_2$-Gromov–Hausdorff distances $\dgh^{\Z_2}(\Sp_{\mathrm{E}}^m,\Sp_{\mathrm{E}}^n)$ for the specific cases $(m,n) = (1,2), (1,3),$ and $(2,3)$, leveraging the lower bound established in~\Cref{sec:Gindices}.

\begin{theorem}\label{thm:SEmSEm+1lb}
For any positive integers $0 < m < n < \infty$, we have
$$\dgh^{\Z_2}(\Sp_\mathrm{E}^m,\Sp_\mathrm{E}^n)\geq\frac{1}{2}\,\sqrt{\frac{2m+4}{m+1}}.$$
\end{theorem}
\begin{proof}
Apply \Cref{Gdhtstability}, \Cref{prop:strongercGpXYlbdd} for $p=\infty$, and  item (2) of \Cref{prop:cforspheres}.
\end{proof}

It turns out that the lower bound given in \Cref{thm:SEmSEm+1lb} is tight in the following cases.

\begin{corollary}\label{cor:SEmnprecise}
We have the following equalities:
\begin{enumerate}
    \item $\dgh^{\Z_2}(\Sp_\mathrm{E}^1,\Sp_\mathrm{E}^{2})=\frac{\sqrt{3}}{2}$,
    
    \item $\dgh^{\Z_2}(\Sp_\mathrm{E}^1,\Sp_\mathrm{E}^{3})=\frac{\sqrt{3}}{2}$, and
    
    \item $\dgh^{\Z_2}(\Sp_\mathrm{E}^2,\Sp_\mathrm{E}^{3})=\sqrt{\frac{2}{3}}.$
\end{enumerate}
\end{corollary}
\begin{proof}
Apply \Cref{thm:SEmSEm+1lb}, \cite[Lemma 9.7]{lim2023gromov}, and the fact that the correspondences of \cite[Propositions 1.16,1.18,1.19]{lim2023gromov} are actually all $\Z_2$-correspondences.
\end{proof}

\Cref{thm:SEmSEm+1lb} and \Cref{cor:SEmnprecise} lead to formulating the following conjecture (cf. \cite[Conjecture 1]{lim2023gromov}):

\begin{conjecture}
For all positive integers $m$,  we have that $\dgh^{\Z_2}(\Sp_\mathrm{E}^m,\Sp_\mathrm{E}^{m+1})=\frac{1}{2}\,\sqrt{\frac{2m+4}{m+1}}$.
\end{conjecture}

By adapting a construction from \cite{lim2021gromov}, we can obtain the following upper bound (which is tight for $m=1$).

\begin{proposition}\label{prop:SEmSEm+1ub}
For any positive integers $m\geq 1$, we have
$$\dgh^{\Z_2}(\Sp_\mathrm{E}^m,\Sp_\mathrm{E}^{m+1})\leq\frac{1}{2}\eta_m^{\mathrm{E}}$$
where
$$\eta_m^{\mathrm{E}}:=\begin{cases}2\sqrt{\frac{m+2}{m+3}}&\text{for odd }m\\ \sqrt{2+2\sqrt{\frac{m}{m+4}}}&\text{for even }m.\end{cases}$$
\end{proposition}
\begin{proof}
Note that the correspondence between $\Sp^m$ and $\Sp^{m+1}$ employed in the proof of \cite[Proposition 1.20]{lim2023gromov} is actually a $\Z_2$-correspondence. Hence, if one uses the same correspondence and computes its distortion relative to $\Sp^n_\mathrm{E}$ and $\Sp^m_\mathrm{E}$, then one can establish the required inequality.
\end{proof}

\subsubsection{The case of $\dgh^{\Z_2}(\square^m_\infty,\square^n_\infty)$}\label{sec:Z2dGHinftysphere}

Recall that $\square^n_\infty$ and $\Sp^n_\infty$ have canonical $\Z_2$-actions; see  item (3) of \Cref{ex:spheres-Z2}. Surprisingly, we have the following.

\begin{theorem}\label{thm:Z2dGHsqmsqn}
For every $m<n$, we have
$$\dgh^{\Z_2}(\square^m_\infty,\square^n_\infty)=1.$$
\end{theorem}
\begin{proof}
Note that $\dgh^{\Z_2}(\square^m_\infty,\square^n_\infty)\leq\frac{1}{2}\max\{\diam(\square^m_\infty),\diam(\square^n_\infty)\}=1$ by item (1) of \Cref{rmk:G-dGH}. Furthermore, $\dgh^{\Z_2}(\square^m_\infty,\square^n_\infty)\geq 1$ by \Cref{Gdhtstability}, \Cref{prop:strongercGpXYlbdd} for $p=\infty$, and  item (3) of \Cref{prop:cforspheres}. This completes the proof. 
\end{proof}

\begin{corollary}
For any $m<n$, we have
$$\dgh^{\Z_2}(\Sp^m_\infty,\Sp^n_\infty)\geq\frac{1}{\sqrt{m+1}}.$$
\end{corollary}
\begin{proof}
Apply \Cref{Gdhtstability}, \Cref{prop:strongercGpXYlbdd} for $p=\infty$, and the item (4) of \Cref{prop:cforspheres}.
\end{proof}


\bibliography{main-arxiv-v3.bib}


\appendix

\section{Proofs from \Cref{sec:Gdistances}.}\label{app:proofs-GHG}

\begin{proof}[\textbf{Proof of \Cref{prop:GGHeqmaps}}]
First, let's prove that LHS$\geq$RHS. For that, fix an arbitrary $R\in\mathcal{R}_G$. For each orbit $[x]\in X\slash G$, choose a representative point $x\in X$ of $[x]$ and define $\varphi(x)$ as a point in $Y$ such that $(x,\varphi(x))\in R$. We can extend the domain of $\varphi$ to whole $X$ in a $G$-equivariant way. Then, we have $\varphi:X\toG Y$ such that $(x,\varphi(x))\in R$ for all $x\in X$, henceforth $\dis(\varphi)\leq\dis(R)$. In a similar way, one can also construct $\psi:Y\toG X$ such that $(\psi(y),y)\in R$ for all $y\in Y$, henceforth $\dis(\psi)\leq\dis(R)$. Also, it is easy to check that $\codis(\varphi,\psi)\leq\dis(R)$. Hence, $\max\{\dis(\varphi),\dis(\psi),\codis(\varphi,\psi)\}\leq\dis(R)$ and this implies that
$$\dgh^G(X,Y)\geq\frac{1}{2}\inf\Big\{\dis(R(\varphi,\psi)):\varphi:X\toG Y \text{and } \psi:Y\toG X\text{ are }G\text{-functions}\Big\}$$
since $R$ is arbitrary.

Now, in order to prove the reverse direction LHS $\leq$ RHS, fix arbitrary $G$-functions $\varphi:X\toG Y$ and $\psi:Y\toG X$. Then, consider the correspondence $R:=\{(x,\varphi(x))\}_{x\in X}\cup\{(\psi(y),y)\}_{y\in Y}$. It is easy to check that $G$ is indeed a $G$-correspondence since both of $\varphi$ and $\psi$ are $G$-functions. Moreover, it is easy to check that $\dis(R)=\max\{\dis(\varphi),\dis(\psi),\codis(\varphi,\psi)\}$. Hence,
$$\dgh^G(X,Y)\leq\frac{1}{2}\inf\Big\{\dis(R(\varphi,\psi)):\varphi:X\toG Y \text{and } \psi:Y\toG X\text{ are }G\text{-functions}\Big\}$$
since $\varphi,\psi$ are arbitrary. This completes the proof.
\end{proof}

\begin{proof}[\textbf{Proof of \Cref{prop:GGHeqembeds}}]
First, let's prove that LHS$\geq$RHS. Fix an arbitrary $G$-correspondence $R\in\mathcal{R}_G$. Let $\varepsilon:=\frac{1}{2}\dis(R)$. Then, one can define a metric $d_Z$ on $Z:=X\sqcup Y$ in the following way: For all $x\in X$ and $y\in Y$,
$$d_Z(x,y):=\inf_{(x',y')\in R}\big(d_X(x,x')+\varepsilon+d_Y(y,y')\big).$$
Moreover, since it is easy to check that $d_Z(\alpha_X(g,x),\alpha_Y(g,y))=d_Z(x,y)$ for all $x\in X$, $y\in Y$, and $g\in G$, indeed $(Z,d_Z,\alpha_X\sqcup \alpha_Y)$ is a $G$-metric space. Hence, if we consider the canonical embeddings $\iota_X:X\hooktoG Z$ and $\iota_Y:Y\hooktoG Z$, then $d_{\mathrm{H}}^Z(\iota_X(X),\iota_Y(Y))=\varepsilon=\frac{1}{2}\dis(R)$. Since the choice of $R$ is arbitrary, we have LHS$\geq$RHS as we wanted.

Now, in order to prove the reverse direction LHS$\leq$RHS, fix an arbitrary $G$-metric space $Z$ and $G$-isometric embeddings $\iota_X:X\hooktoG Z$ and $\iota_Y:Y\hooktoG Z$. Now, choose an arbitrary $\varepsilon>d_{\mathrm{H}}^Z(\iota_X(X),\iota_Y(Y))$. For each orbit $[x]\in X\slash G$, choose a representative point $x\in X$ of $[x]$ and define $\varphi(x)$ as a point in $Y$ such that $d_Z(\iota_X(x),\iota_Y(\varphi(x)))<\varepsilon$. We can extend the domain of $\varphi$ to whole $X$ in a $G$-equivariant way. Then, we have a $G$-function $\varphi:X\toG Y$ such that $d_Z(\iota_X(x),\iota_Y(\varphi(x)))<\varepsilon$ for all $x\in X$. In a similar way, one can also construct a $G$-function $\psi:Y\toG X$ such that $d_Z(\iota_X(\psi(y)),\iota_Y(y))<\varepsilon$ for all $y\in Y$.  Then, let's consider the correspondence $R:=\{(x,\varphi(x))\}_{x\in X}\cup\{(\psi(y),y)\}_{y\in Y}$. It is easy to verify that $R$ is a $G$-correspondence between $X$ and $Y$ and $\dis(R)\leq 2\varepsilon$. This implies that $\dgh^G(X,Y)\leq\varepsilon$. Since the choice of $\varepsilon$ is arbitrary, we have $\dgh^G(X,Y)\leq d_{\mathrm{H}}^Z(\iota_X(X),\iota_Y(Y))$. Finally, since the choice of $Z$,$\iota_X$, and $\iota_Y$ are arbitrary, we have LHS$\leq$RHS as we required. This completes the proof.
\end{proof}

\begin{proof}[\textbf{Proof of \Cref{thm:dghGmetric}}]
While our proof will be just a slight modification of \cite[Proof of Theorem 7.3.30]{burago2022course}, we write the proof here for the completeness.

Since all the other properties are trivial, it is enough to prove the triangle inequality and that $\dgh^G(X,Y)=0$ implies $X$ and $Y$ are $G$-isometric.

Let us prove the triangle inequality, first. Fix arbitrary $G$-spaces $X$, $Y$, and $Z$. We will prove
$$\dgh^G(X,Z)\leq\dgh^G(X,Y)+\dgh^G(Y,Z).$$
Now, fix an arbitrary $\varepsilon>0$ and choose $R_{XY}\in\mathcal{R}_G(X,Y)$ and $R_{YZ}\in\mathcal{R}_G(Y,Z)$ such that $\dis(R_{XY})<2\dgh^G(X,Y)+\varepsilon$ and $\dis(R_{YZ})<2\dgh^G(Y,Z)+\varepsilon$. Next, we define
$$R_{XZ}:=\{(x,z)\in X\times X:\exists \,y\in Y\text{ s.t. }(x,y)\in R_{XY}\text{ and }(y,z)\in R_{YZ}\}.$$
Note that it is easy to verify that $R_{XZ}\in\mathcal{R}_G(X,Z)$. Furthermore, since
\begin{align*}
    \vert d_X(x,x')-d_Z(z,z')\vert&\leq\vert d_X(x,x')-d_Y(y,y')\vert+\vert d_X(y,y')-d_Z(z,z')\vert\\
    &\leq \dis(R_{XY})+\dis(R_{YZ})
\end{align*}
for all $(x,z),(x',z')\in R_{XZ}$, we have that
$$2\dgh^G(X,Z)\leq\dis(R_{XZ})\leq\dis(R_{XY})+\dis(R_{YZ})<2\dgh^G(X,Y)+2\dgh^G(Y,Z)+2\varepsilon.$$
Since $\varepsilon$ is arbitrary, the proof of the triangle inequality is complete.

Finally, we need to show that if $\dgh^G(X,Y)=0$ then $X$ and $Y$ are $G$-isometric. Let $X$ and $Y$ be two compact $G$-metric spaces such that $\dgh^G(X,Y)=0$. By \Cref{prop:GGHeqmaps}, there exists a sequence of $G$-functions $\varphi_n:X\toG Y$ such that $\dis(\varphi_n)\stackrel{n\rightarrow\infty}{\longrightarrow}0$. Fix a countable dense set $S\subseteq X$. Furthermore, note that (under the assumption that $G$ is finite or countable) without loss of generality one can assume that $S$ is $G$-invariant. i.e., $\alpha_X(g,x)\in S$ for all $g\in G$ and $x\in S$. Using the Cantor diagonal procedure, one can choose a subsequence $\{\varphi_{n_k}\}$ of $\{\varphi_n\}$ such that for every $x\in S$ the sequence $\{\varphi_{n_k}(x)\}$ converges in $Y$. Without loss of generality we may assume that this holds for $\{\varphi_n\}$ itself. Then one can define a map $\varphi:S\toG Y$ as the limit of $\varphi_n$, namely, set $\varphi(x):=\lim\limits_{n\rightarrow\infty}\varphi_n(x)$ for every $x\in S$. Since $\vert d_X(x,x')-d_Y(\varphi_n(x),\varphi_n(x'))\vert\leq\dis(\varphi_n)\stackrel{n\rightarrow\infty}{\longrightarrow}0$, we have $d_Y(\varphi(x),\varphi(x'))=\lim\limits_{n\rightarrow\infty}d_Y(\varphi_n(x),\varphi_n(x'))=d_X(x,x')$ for all $x,x'\in S$. In other words, $\varphi$ is a $G$-equivariant distance preserving map from $S$ to $Y$. Then $\varphi$ can be extended to a distance preserving map from the entire $X$ to $Y$ by \cite[Proposition 1.5.9]{burago2022course}. Furthermore it is easy to verify that this extended $\varphi$ is also $G$-equivariant. In a similar way, one can also construct a $G$-equivariant distance preserving map $\psi:Y\toG X$. Then, $\psi\circ\varphi$ is a distance preserving map from $X$ to $X$. Since $X$ is compact, $\psi\circ\varphi$ is surjective by \cite[Proposition 1.6.14]{burago2022course}. Similarly $\varphi\circ\psi$ is also surjective and it follows that $X$ and $Y$ are $G$-isometric.
\end{proof}

\begin{proof}[\textbf{Proof of \Cref{prop:diGpsdmtr}}]
It is easy to verify that $$\mbox{$\di^G\big((V_\bullet,\nu_\bullet),(V_\bullet,\nu_\bullet)\big)=0$ and $\di^G\big((V_\bullet,\nu_\bullet),(W_\bullet,\omega_\bullet)\big)=\di^G\big((W_\bullet,\omega_\bullet),(V_\bullet,\nu_\bullet)\big)$}$$ from the definition. Hence, we only need to show the triangle inequality. Consider three arbitrary $G$-persistence modules $(U_\bullet,\upsilon_\bullet)$, $(V_\bullet,\nu_\bullet)$, and $(W_\bullet,\omega_\bullet)$. We need to show the following:
$$\di^G\big((U_\bullet,\upsilon_\bullet),(W_\bullet,\omega_\bullet)\big)\leq\di^G\big((U_\bullet,\upsilon_\bullet),(V_\bullet,\nu_\bullet)\big)+\di^G\big((V_\bullet,\nu_\bullet),(W_\bullet,\omega_\bullet)\big).$$
Fix an arbitrary $\varepsilon>\di^G\big((U_\bullet,\upsilon_\bullet),(V_\bullet,\nu_\bullet)\big)$ and $\varepsilon'>\di^G\big((V_\bullet,\nu_\bullet),(W_\bullet,\omega_\bullet)\big)$. Then, $(U_\bullet,\upsilon_\bullet)$ and $(V_\bullet,\nu_\bullet)$ are $(G,\varepsilon)$-interleaved and $(V_\bullet,\nu_\bullet)$ and $(W_\bullet,\omega_\bullet)$ are $(G,\varepsilon')$-interleaved. Hence, there are $G$-persistence morphisms $T_\bullet:U_\bullet \to V_{\bullet+\varepsilon}$, $T'_\bullet: V_\bullet \to U_{\bullet+\varepsilon}$, $S_\bullet:V_\bullet \to W_{\bullet+\varepsilon'}$, and $S'_\bullet: W_\bullet \to V_{\bullet+\varepsilon'}$ satisfying the conditions given in \Cref{def:gint-dist}. Finally, consider the following $G$-persistence morphisms $S_{\bullet+\varepsilon}\circ T_\bullet:U_\bullet \to W_{\bullet+\varepsilon+\varepsilon'}$ and $T'_{\bullet+\varepsilon'}\circ S'_\bullet: W_\bullet \to U_{\bullet+\varepsilon+\varepsilon'}$. We will verify that these families also satisfy the conditions given in \Cref{def:gint-dist}. Note that
\begin{align*}
    (T'_{r+\varepsilon+2\varepsilon'}\circ S'_{r+\varepsilon+\varepsilon'})\circ(S_{r+\varepsilon}\circ T_r)&=T'_{r+\varepsilon+2\varepsilon'}\circ (S'_{r+\varepsilon+\varepsilon'}\circ S_{r+\varepsilon})\circ T_r\\
    &=T'_{r+\varepsilon+2\varepsilon'}\circ v_{r+\varepsilon,r+\varepsilon+2\varepsilon'}\circ T_r\\
    &=u_{r+2\varepsilon,r+2\varepsilon+2\varepsilon'}\circ T'_{r+\varepsilon}\circ T_r\\
    &=u_{r+2\varepsilon,r+2\varepsilon+2\varepsilon'}\circ u_{r,r+2\varepsilon}=u_{r,r+2\varepsilon+2\varepsilon'}.
\end{align*}
for every $r\in\R$. In a similar way, one can also check that $(S_{r+2\varepsilon+\varepsilon'}\circ T_{r+\varepsilon+\varepsilon'})\circ(T'_{r+\varepsilon'}\circ S'_{r})=w_{r,r+2\varepsilon+2\varepsilon'}$ for every $r\in\R$.

Hence, $(U_\bullet,\upsilon_\bullet)$ and $(W_\bullet,\omega_\bullet)$ are $(G,\varepsilon+\varepsilon')$-interleaved so that $\di^G\big((U_\bullet,\upsilon_\bullet),(W_\bullet,\omega_\bullet)\big)\leq\varepsilon+\varepsilon'$. Since the choice of $\varepsilon$ and $\varepsilon'$ are arbitrary, one can conclude that
$$\di^G\big((U_\bullet,\upsilon_\bullet),(W_\bullet,\omega_\bullet)\big)\leq\di^G\big((U_\bullet,\upsilon_\bullet),(V_\bullet,\nu_\bullet)\big)+\di^G\big((V_\bullet,\nu_\bullet),(W_\bullet,\omega_\bullet)\big).$$
\end{proof}

\section{Proof of \Cref{prop:Gnerve}.}\label{sec:Gnerve}

Our proof of \Cref{prop:Gnerve} invokes many elements of \cite[Section 4.G]{h01} which provides a proof of the classical nerve lemma. We attempted to keep our arguments as elementary and explicit as possible in order to increase readability.

Let $\Gamma=(V(\Gamma),E(\Gamma))$ be a directed graph such that if $e=(u,v), e'=(v,w)\in E(\Gamma)$ then $e''=(u,w)\in E(\Gamma)$. i.e., directed edges of $\Gamma$ are commutative. A sequence $(v_0,v_1,\dots,v_n)$ is called as a directed path of $\Gamma$ if $(v_0,v_1),(v_1,v_2),\dots,(v_{n-1},v_n)$ are directed edges of $\Gamma$. For technical reasons, we will view a singleton sequence $(v_0)$ also as a path.

For a given directed graph $\Gamma$ with commutative edges, a \emph{complex of spaces}\footnote{Actually, this is just a special case of a more general definition of complex of spaces. For the general version of definition, see \cite[Section 4.G]{h01}.} associated to $\Gamma$ consists of a set of topological spaces $\{X_v\}_{v\in V(\Gamma)}$ and a set of embeddings $\{\iota_e\}_{e\in E(\Gamma)}$ such that $\iota_e:X_u\longhookrightarrow X_v$ if $e=(u,v)$. By an abuse of notation, both of a complex of spaces and its underlying directed graph will be denoted by $\Gamma$.

Next, for a complex of spaces $\Gamma$, one can consider the following topological space $\Delta\Gamma$ called the \emph{realization} of $\Gamma$: For any path $(v_0,v_1,\dots,v_n)$ in $\Gamma$, one can consider $X_{v_0}\times\Delta_n$ where $\Delta_n$ is the $n$-simplex with the vertices $\{v_0,v_1,\dots,v_n\}$. Furthermore, if $(i_0,i_1,\dots,i_{n'})$ is a subsequence $(0,1,\dots,n)$, then $(v_{i_0},v_{i_1},\dots,v_{i_{n'}})$ is also a path in $\Gamma$ by the commutativity. Hence, one can consider the identification between $(x,\iota_{n',n}(p))\in X_{v_0}\times\Delta_n$ and $(\iota_{(v_0,v_{i_0})}(x),p)\in X_{v_{i_0}}\times\Delta_{n'}$ for all $x\in X_{v_0},p\in\Delta_{n'}$ induced by the canonical inclusions $\iota_{(v_0,v_{i_0})}:X_{v_0}\longhookrightarrow X_{v_{i_0}}$ and $\iota_{n',n}:\Delta_{n'}\longhookrightarrow \Delta_n$. Finally, the realization $\Delta\Gamma$ is the quotient space of the disjoint union of all $X_{v_0}\times\Delta_n$ with the aforementioned type of identifications.
$$\Delta\Gamma:=\left(\bigsqcup_{(v_0,v_1,\dots,v_n)\text{ is a path}} (X_{v_0}\times\Delta_n)\right)\slash\sim.$$

One typical example of a complex of spaces is as follows: Let $X$ be a topological space and $\mathcal{U}=\{U_\lambda\}_{\lambda\in\Lambda}$ be an open cover of $X$ ($\Lambda$ is an arbitrary index set). For each $\sigma:=\{i_0,\dots,i_n\}\in\mathrm{N}\,\mathcal{U}$, the nonempty intersection $U_{i_0}\cap\dots \cap\,U_{i_n}$ is denoted by $U_\sigma$. Note that, if $\sigma'$ is a face of $\sigma$, then there is a canonical inclusion
$$\iota_{\sigma,\sigma'}:U_\sigma\longhookrightarrow U_{\sigma'}.$$
Hence, $\{U_\sigma\}_{\sigma\in\mathrm{N}\,\mathcal{U}}$ and $\{\iota_{\sigma,\sigma'}\}_{\sigma'\prec\sigma}$ form a complex of spaces over the $1$-skeleton of $\mathrm{sd}(\mathrm{N}\,\mathcal{U})$, the barycentric subdivision of $\mathrm{N}\,\mathcal{U}$. Recall that each $n$-simplex of $\mathrm{sd}(\mathrm{N}\,\mathcal{U})$ is $\{\sigma_0,\sigma_1,...,\sigma_n\}$ where $\sigma_i$'s are simplexes of $\mathrm{N}\,\mathcal{U}$ such that $\sigma_i\prec\sigma_{i+1}$ for all $i=0,...,n-1$. In this case, the associated realization is denoted by $\Delta X_\mathcal{U}$. More precisely,

$$\Delta X_\mathcal{U}:=\left(\bigsqcup_{\sigma_0\prec\sigma_1\cdots\prec\sigma_n} (U_{\sigma_n}\times\Delta_n)\right)\slash\sim$$

where $(x,\iota_{n',n}(p))\in U_{\sigma_n}\times\Delta_n$ and $(\iota_{\sigma_n,\sigma_{i_{n'}}}(x),p)\in U_{\sigma_{i_{n'}}}\times\Delta_{n'}$ are identified for all $x\in U_{\sigma_n},p\in\Delta_{n'}$ whenever $\{\sigma_{i_0},\sigma_{i_1},...,\sigma_{i_{n'}}\}$ is an $n'$-face of $\{\sigma_0,\sigma_1,...,\sigma_n\}$ in $\mathrm{sd}(\mathrm{N}\,\mathcal{U})$. Observe that one can view $\Delta X_\mathcal{U}$ as a certain thickening of $\vert\mathrm{sd}(\mathrm{N}\,\mathcal{U})\vert$ since one can express $\vert\mathrm{sd}(\mathrm{N}\,\mathcal{U})\vert$ as follows:
$$\vert\mathrm{sd}(\mathrm{N}\,\mathcal{U})\vert=\left(\bigsqcup_{\sigma_0\prec\sigma_1\cdots\prec\sigma_n} (\{*\}\times\Delta_n)\right)\slash\sim.$$

Also, if $(X,\alpha_X)$ is a $G$-space and $\mathcal{U}$ is a $G$-invariant open cover for a given group $G$ (see \Cref{def:Ginvcov}), then all of $\vert\mathrm{N}\,\mathcal{U}\vert$, $\vert\mathrm{sd}(\mathrm{N}\,\mathcal{U})\vert$, and $\Delta X_\mathcal{U}$ are $G$-spaces in the obvious way. We need the following \Cref{prop:4g1} and \Cref{prop:4g2}. Note that these propositions are improvements of \cite[Propositions 4G.1 and 4G.2]{h01}, respectively.

\begin{proposition}\label{prop:4g1}
Assume that a group $G$, a $G$-space $X$, and a $G$-invariant open cover $\mathcal{U}:=\{U_\lambda\}_{\lambda\in\Lambda}$ of $X$ are given. Also, assume that $U_\sigma:=\bigcap_{\lambda\in\sigma} U_\lambda$ is $G_\sigma$-contractible for every simplex $\sigma\in\mathrm{N}\,\mathcal{U}$ where $G_\sigma:=\{g\in G:\alpha_\Lambda(g,\sigma)=\sigma\}$ is the isotropy subgroup of $G$.  Then, $\Delta X_\mathcal{U}$ and $\vert\mathrm{N}\,\mathcal{U}\vert$ are $G$-homotopy equivalent.
\end{proposition}

In order to prove \Cref{prop:4g1}, we need the following technical preliminaries.\\

\begin{itemize}
    \item Let $X$, $Y$ be $G$-spaces and $\psi:X\toG Y$ be a $G$-map. Then, the \emph{mapping cylinder} $M\psi$ associated to $\psi$ (see \cite[pg.2]{h01}) is also a $G$-space with the obvious $G$-action.

    \item We say a pair $(X,A)$ as a \emph{$G$-pair} if $X$ is a $G$-space and $\emptyset\neq A\subseteq X$ is a $G$-invariant subset of $X$. 
    
    \item We say that a $G$-pair $(X,A)$ has the \emph{$G$-homotopy extension property} if the following holds: For any $G$-space $Y$ and $G$-map $\varphi:(X\cup\{0\})\bigcup (A\times [0,1])\toG Y$, we always have its extension $G$-map $\widetilde{\varphi}:X\times [0,1]\toG Y$.

    \item Assume that a $G$-pair $(X,A)$ is given. Then, a $G$-map $r:X\toG A$ is called as a \emph{$G$-retraction} if $r$ is a retraction and a $G$-map.
\end{itemize}

The following is a generalization of \cite[Example 0.15]{h01}

\begin{lemma}\label{lemma:Ghepexample}
Let $(X,A)$ be a $G$-pair. Suppose that there are $G$-invariant subsets $B,N\subseteq X$ such that:
\begin{enumerate}
    \item $A,B\subset N$ and $A\cap B=\emptyset$; and

    \item there are $G$-map $\psi:B\toG A$ and $G$-homeomorphism $\Psi:M\psi\toG N$ with $\Psi\vert_{A\cup B}=\mathrm{id}_{A\cup B}$.
\end{enumerate}
Then, $(X,A)$ has the $G$-homotopy extension property.
\end{lemma}
\begin{proof}
We modify \cite[Example 0.15]{h01}. First, observe that $[0,1]\times [0,1]$ retracts onto $([0,1]\times \{0\})\bigcup(\{0,1\}\times [0,1])$, hence $B\times [0,1]\times [0,1]$ $G$-retracts onto $(B\times[0,1]\times \{0\})\bigcup(B\times\{0,1\}\times [0,1])$, and this $G$-retraction induces a $G$-retraction of $M\psi\times [0,1]$ onto $(M\psi\times\{0\})\bigcup ((A\cup B)\times [0,1])$. Thus, $(M\psi,A\cup B)$ has the $G$-homotopy extension property. Hence so does the $G$-homeomorphic pair $(N,A\cup B)$. Now given a $G$-map $\varphi:(X\cup\{0\})\bigcup (A\times [0,1])\toG Y$, one can define its extension $\widetilde{\varphi}:X\times [0,1]\toG Y$ in the following manner: $\widetilde{\varphi}(x,t):=\varphi(x,0)$ if $(x,t)\in X\backslash(N\backslash B)\times[0,1]$ and use the $G$-homotopy extension property of $(N,A\cup B)$ for the other region.
\end{proof}

The next lemma is a $G$-equivariant version of \cite[Proposition 0.19 and Corollary 0.20]{h01}.

\begin{lemma}\label{lemma:Gdefretact}
Suppose that $(X,A)$ is a $G$-pair and has the $G$-homotopy extension property. If the canonical inclusion $\iota:A\hooktoG X$ and a $G$-map $\psi:X\toG A$ are $G$-homotopy equivalences, then there exists a $G$-deformation retraction $\widetilde{\psi}:X\toG A$.
\end{lemma}
\begin{proof}
Let $H:A\times [0,1]\longrightarrow A$ be a $G$-homotopy with $H(\cdot,0)=\psi\circ\iota$ and $H(\cdot,1)=\mathrm{id}_A$. Since $\psi\circ\iota=\psi\vert_A$, $H$ can be seen as a $G$-homotopy between $\psi\vert_A$ and $\mathrm{id}_A$. Then, since we assume $(X,A)$ has the $G$-homotopy extension property, we can extend $H$ to a $G$-homotopy $\widetilde{H}:X\times [0,1]\longrightarrow A$. Then, we definite the $G$-map $\widetilde{\psi}:=\widetilde{H}(\cdot,1)$. Note that $\widetilde{\psi}\circ\iota=\mathrm{id}_A$. One can show that $\widetilde{\psi}$ is indeed a $G$-deformation retraction from $X$ to $A$ following the same steps in the proof of \cite[Proposition 0.19]{h01}.
\end{proof}

\begin{lemma}\label{lemma:mcdr}
Let $G$ be a group and $(X,\alpha_X)$ be a $G$-contractible space with a $G$-homotopy equivalence $\varphi:X\toG\{*\}$. Then, there is a $G$-deformation retraction $\widetilde{\psi}:M\varphi\toG X$. In particular, the canonical inclusion $\iota:X\hooktoG M\varphi$ is a $G$-homotopy equivalence.
\end{lemma}
\begin{proof}
Let $\psi:\{*\}\toG X$ be a $G$-homotopy equivalence such that $\psi\circ\varphi\Gsimeq\mathrm{id}_X$. Let $H:X\times [0,1]\longrightarrow X$ be a $G$-homotopy between $\psi\circ\varphi$ and $\mathrm{id}_X$ such that $H(\cdot,0)=\mathrm{id}_X$ and $H(\cdot,1)=\psi\circ\varphi$.

Next, let $\widetilde{\psi}:M\varphi\toG X$ s.t. $(x,t)\mapsto H(x,t)$. Then, it is easy to verify that $\widetilde{\psi}\circ\iota=\mathrm{id}_X$. Finally, $\iota\circ\widetilde{\psi}\Gsimeq\mathrm{id}_{M\varphi}$ via the following $G$-homotopy
\begin{align*}
    \widetilde{H}:M\varphi\times [0,1]&\longmapsto M\varphi\\
    (x,t,s)&\longmapsto (H(x,st),(1-s)t).
\end{align*}
\end{proof}

\begin{lemma}\label{lemma:simplexmc}
Let $k\geq 0$ and $\Gamma=(V(\Gamma),E(\Gamma))$ be a directed graph such that $V(\Gamma):=\{v_0,v_1,\cdots,v_k\}$ and $E(\Gamma):=\{(v_i,v_j)\}_{0\leq i < j \leq k}$. Note that $\Gamma$ can be seen as the $1$-skeleton of the $k$-simplex. Furthermore, let $G$ be a group and consider a complex of spaces over $\Gamma$ such that $X_v$ is $G$-contractible for all $v\in V(\Gamma)$. Let $\varphi_v:X_v\rightarrow\{*\}$ be the $G$-homotopy equivalence and $M\varphi_v$ be its mapping cylinder for each $v\in V(\Gamma)$. Then, $\Delta M\Gamma$ $G$-deformation retracts onto $\Delta M^{(k-1)}\Gamma\cup\Delta\Gamma$ where
$$\Delta M\Gamma:=\left(\bigsqcup_{(v_{i_0},v_{i_1},\dots,v_{i_n})\text{ is a path}} (M\varphi_{v_{i_0}}\times\Delta_n)\right)\slash\sim$$
and
$$\Delta M^{(k-1)}\Gamma:=\left(\bigsqcup_{n\leq k-1, (v_{i_0},v_{i_1},\dots,v_{i_n})\text{ is a path}} (M\varphi_{v_{i_0}}\times\Delta_n)\right)\slash\sim.$$
\end{lemma}

\begin{figure}
\begin{center}
    \includegraphics[width=8cm]{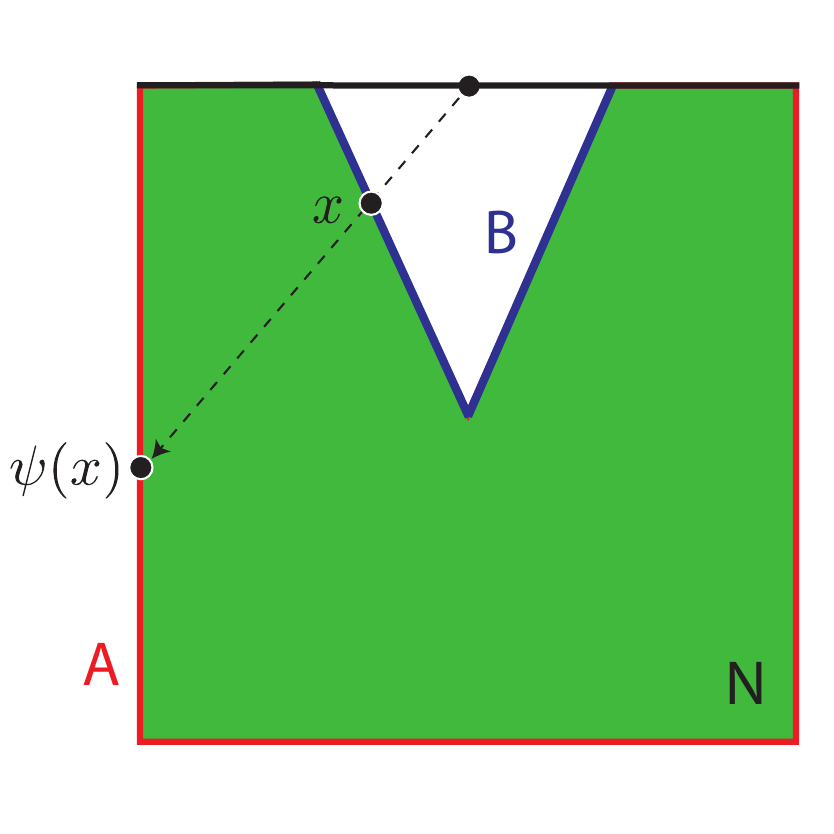}
\end{center}
    \caption{An example of how to build a mapping cylinder neighborhood for the pair $(\Delta_1\times [0,1],(\Delta_1\times{0})\bigcup(\partial\Delta_1\times [0,1]))$. The \textcolor{red}{red} region is $A=(\Delta_1\times{0})\bigcup(\partial\Delta_1\times [0,1])$, the \textcolor{green}{green} region is the closed neighborhood $N$, and the \textcolor{blue}{blue} region is $B$.}\label{fig:mcn}
\end{figure}

\begin{proof}
We will prove the lemma by induction on $k$. The claim holds for $k=0$ by \Cref{lemma:mcdr}. Assume that the claim holds up to $k-1$. Now, we need to verify the claim for dimension $k$.

By \Cref{lemma:Gdefretact}, it suffices to show that the canonical inclusion $\iota_{k-1,k}:\Delta M^{(k-1)}\Gamma\cup\Delta\Gamma\hooktoG\Delta M\Gamma$ is a $G$-homotopy equivalence and the pair $(\Delta M\Gamma,\Delta M^{(k-1)}\Gamma\cup\Delta\Gamma)$ satisfies the $G$-homotopy extension property. One can verify that the latter assertion is true by \Cref{lemma:Ghepexample} since the pair $\big(\Delta_n\times [0,1],(\Delta_n\times\{0\})\bigcup(\partial\Delta_n\times [0,1])\big)$ satisfies the conditions in \Cref{lemma:Ghepexample} (see \Cref{fig:mcn}). For the other condition, recall that we know $\Delta M^{(k-1)}\Gamma\cup\Delta\Gamma$ $G$-deformation retracts onto $\Delta\Gamma$ by the induction assumption. Hence, there is a $G$-map $\psi_{k-1}:\Delta M^{(k-1)}\Gamma\cup\Delta\Gamma\longrightarrow\Delta\Gamma$ s.t. $\iota_{k-1}\circ\psi_{k-1}\Gsimeq\mathrm{id}_{\Delta M^{(k-1)}\Gamma\cup\Delta\Gamma}$ and $\psi_{k-1}\circ\iota_{k-1}=\mathrm{id}_{\Delta\Gamma}$ where $\iota_{k-1}$ is the canonical inclusion from $\Delta\Gamma$ to $\Delta M^{(k-1)}\Gamma\cup\Delta\Gamma$. Furthermore, the canonical inclusion $\iota_k:\Delta\Gamma\hooktoG\Delta M\Gamma$ is a $G$-homotopy equivalence since it is equivalent to $\iota_{v_k}:X_{v_k}\hooktoG M\varphi_{v_k}$ as the following commutative diagram shows:
$$\begin{tikzcd}
\Delta\Gamma \arrow[r, hook, "\iota_k"] \arrow[d, shift left, "P_{X_{v_k}}"]
& \Delta M\Gamma \arrow[d, shift left, "P_{M\varphi_{v_k}}"]\\
X_{v_k}  \arrow[r, hook, "\iota_{v_k}"] \arrow[u, hook, shift left, "Q_{X_{v_k}}"] & M\varphi_{v_k} \arrow[u, hook, shift left, "Q_{M\varphi_{v_k}}"]
\end{tikzcd}$$
where $P_{X_{v_k}}:\Delta\Gamma\toG X_{v_k}$, $P_{M\varphi_{v_k}}:\Delta M\Gamma\toG M\varphi_{v_k}$ are the canonical projections such that $P_{X_{v_k}}(x,p):=x$ $\forall\,(x,p)\in\Delta\Gamma$, $P_{M\varphi_{v_k}}(x,t,p):=(x,t)$ $\forall\,(x,t,p)\in\Delta M\Gamma$ and $Q_{X_{v_k}}$, $Q_{M\varphi_{v_k}}$ are the canonical inclusions. It is easy to verify that $P_{X_{v_k}}$, $P_{M\varphi_{v_k}}$, $Q_{X_{v_k}}$, and $Q_{M\varphi_{v_k}}$ are $G$-homotopy equivalences (actually they are $G$-deformation retracts and their inverses) with the $G$-homotopies constructed via the linear interpolation. Since $\iota_{v_k}$ is a $G$-homotopy equivalence by \Cref{lemma:mcdr}, $\iota_k$ is also a $G$-homotopy equivalence. Now, let $\widetilde{\psi}_k:\Delta M\Gamma\toG\Delta\Gamma$ be a $G$-homotopy equivalence such that $\widetilde{\psi}_k\circ\iota_k\Gsimeq\mathrm{id}_{\Delta\Gamma}$ and $\iota_k\circ\widetilde{\psi}_k\Gsimeq\mathrm{id}_{\Delta M\Gamma}$.
$$\begin{tikzcd}
\Delta\Gamma \arrow[r, hook, shift left, "\iota_{k-1}"] \arrow[rr, bend left, "\iota_k"] & \Delta M^{(k-1)}\Gamma\cup\Delta\Gamma \arrow[r, hook, shift left, "\iota_{k-1,k}"] \arrow[l, shift left, "\psi_{k-1}"] & \Delta M\Gamma \arrow[l, shift left, "\psi_{k-1,k}"] \arrow[ll, bend left, "\widetilde{\psi}_k"]
\end{tikzcd}$$
Finally, let $\widetilde{\psi}_{k-1,k}:=\iota_{k-1}\circ\widetilde{\psi}_k$. Then, $\widetilde{\psi}_{k-1,k}\circ\iota_{k-1,k}=\iota_{k-1}\circ\widetilde{\psi}_k\circ\iota_{k-1,k}\Gsimeq\iota_{k-1}\circ\widetilde{\psi}_k\circ\iota_{k-1,k}\circ\iota_{k-1}\circ\psi_{k-1}=\iota_{k-1}\circ\widetilde{\psi}_k\circ\iota_k\circ\psi_{k-1}\Gsimeq\iota_{k-1}\circ\psi_{k-1}\Gsimeq\mathrm{id}_{\Delta M^{(k-1)}\Gamma\cup\Delta\Gamma}$. Also, $\iota_{k-1,k}\circ\widetilde{\psi}_{k-1,k}=\iota_{k-1,k}\circ\iota_{k-1}\circ\widetilde{\psi}_k=\iota_k\circ\widetilde{\psi}_k\Gsimeq\mathrm{id}_{\Delta M\Gamma}$. Hence, $\iota_{k-1,k}$ is a $G$-homotopy equivalence as we required. This completes the proof.
\end{proof}

Now we are ready to prove \Cref{prop:4g1}.

\begin{proof}[\textbf{Proof of \Cref{prop:4g1}}]
First of all, since $\vert\mathrm{N}\,\mathcal{U}\vert$ and $\vert\mathrm{sd}(\mathrm{N}\,\mathcal{U})\vert$ are $G$-homotopy equivalent in the obvious way, it is enough to show the $G$-homotopy equivalence between $\Delta X_\mathcal{U}$ and $\vert\mathrm{sd}(\mathrm{N}\,\mathcal{U})\vert$.

Since $U_\sigma$ is $G_\sigma$-contractible whenever $\sigma\in\mathrm{N}\,\mathcal{U}$, note that there is a $G_\sigma$-homotopy equivalence $\varphi_\sigma:U_\sigma\stackrel{G_\sigma}{\longrightarrow}\{*\}$ for all $\sigma\in\mathrm{N}\,\mathcal{U}$. Now, let $M\varphi_\sigma$ denotes the mapping cylinder associated to $\varphi_\sigma$. It is easy to verify that if $\sigma'$ is a face of $\sigma$ then $M\varphi_\sigma\subseteq M\varphi_{\sigma'}$. Hence, the set of all $M\varphi_\sigma$ and their inclusions also form a complex of diagram over the $1$-skeleton of $\mathrm{sd}(\mathrm{N}\,\mathcal{U})$, and its associated realization is as follows:
$$\Delta MX_\mathcal{U}:=\left(\bigsqcup_{\sigma_0\prec\sigma_1\cdots\prec\sigma_n} (M\varphi_{\sigma_n}\times\Delta_n)\right)\slash\sim$$

Observe that $\Delta MX_\mathcal{U}$ is also a $G$-space in a natural way and is actually the mapping cylinder associated to the $G$-map $\varphi$ s.t.
\begin{align*}
    \varphi:\Delta X_\mathcal{U}&\toG\vert\mathrm{sd}(\mathrm{N}\,\mathcal{U})\vert\\
    (x,p)&\longmapsto(*,p)
\end{align*}
Therefore, $\Delta MX_\mathcal{U}$ $G$-deformation retracts onto $\vert\mathrm{sd}(\mathrm{N}\,\mathcal{U})\vert$. Hence, we only need to show the $G$-homotopy equivalence between $\Delta X_\mathcal{U}$ and $\Delta MX_\mathcal{U}$. Now, for each nonnegative integer $k$, let $\Delta M^{(k)}X_\mathcal{U}$ be the $k$-skeleton part of $\Delta MX_\mathcal{U}$. More precisely,
$$\Delta M^{(k)}X_\mathcal{U}:=\left(\bigsqcup_{n\leq k, \sigma_0\prec\sigma_1\cdots\prec\sigma_n} (M\varphi_{\sigma_n}\times\Delta_n)\right)\slash\sim$$
We claim that $\Delta M^{(k)}X_\mathcal{U}\cup\Delta X_\mathcal{U}$ $G$-deformation retracts onto $\Delta M^{(k-1)}X_\mathcal{U}\cup\Delta X_\mathcal{U}$ for all $k\geq 0$. This assertion holds since a $G_\sigma$-deformation retract exists for each part of $\Delta M^{(k)}X_\mathcal{U}\cup\Delta X_\mathcal{U}$ over a $k$-simplex by \Cref{lemma:simplexmc}. More precisely, for each $k$-simplex $\hat{\tau}:=\{\tau_0,\tau_1,...,\tau_k\}$ of $\mathrm{sd}(\mathrm{N}\,\mathcal{U})$, there exist a $G_{\tau_k}$-deformation retraction $\psi_{\hat{\tau}}$ from $\left(\bigsqcup_{\tau_{i_0}\prec\tau_{i_1}\cdots\prec\tau_{i_n}} (M\varphi_{\tau_{i_n}}\times\Delta_n)\right)\slash\sim$ onto $\left(\left(\bigsqcup_{n\leq k-1, \tau_{i_0}\prec\tau_{i_1}\cdots\prec\tau_{i_n}} (M\varphi_{\tau_{i_n}}\times\Delta_n)\right)\slash\sim\right)\bigcup\left(\left(\bigsqcup_{\tau_{i_0}\prec\tau_{i_1}\cdots\prec\tau_{i_n}} (U_{\tau_{i_n}}\times\Delta_n)\right)\slash\sim\right)$ by \Cref{lemma:simplexmc}. Then, for each $k\geq 0$, one can define a map
$$\psi_{k-1,k}:\Delta M^{(k)}X_\mathcal{U}\cup\Delta X_\mathcal{U}\longrightarrow\Delta M^{(k-1)}X_\mathcal{U}\cup\Delta X_\mathcal{U}$$
as follows:
$$\psi_{k-1,k}(x,t,p):=\psi_{\hat{\tau}}(x,t,p)\text{ if }(x,t,p)\in\left(\bigsqcup_{\tau_{i_0}\prec\tau_{i_1}\cdots\prec\tau_{i_n}} (M\varphi_{\tau_{i_n}}\times\Delta_n)\right)\slash\sim$$
Now, let us verify that each $\psi_{k-1,k}:\Delta M^{(k)}X_\mathcal{U}\cup\Delta X_\mathcal{U}\longrightarrow\Delta M^{(k-1)}X_\mathcal{U}\cup\Delta X_\mathcal{U}$ is a $G$-map in order to finish the proof. However, one can easily make $\psi_{k-1,k}$ as $G$-equivariant as follows: fix an arbitrary $g\in G$ and $k$-simplex $\hat{\tau}:=\{\tau_0,\tau_1,...,\tau_k\}$ of $\mathrm{sd}(\mathrm{N}\,\mathcal{U})$. Now, if $\alpha_\Lambda(g,\hat{\tau})\neq\hat{\tau}$, then we define $\psi_{k-1,k}(g.(x,t,p)):=g.\psi_{\hat{\tau}}(x,t,p)$ for all $(x,t,p)\in\left(\bigsqcup_{\tau_{i_0}\prec\tau_{i_1}\cdots\prec\tau_{i_n}} (M\varphi_{\tau_{i_n}}\times\Delta_n)\right)\slash\sim$. Secondly, if $\alpha_\Lambda(g,\hat{\tau})=\hat{\tau}$, in particular this implies that $g\in G_{\tau_k}$. Hence, since $\psi_{\hat{\tau}}$ is $G_{\tau_k}$-contractible, we have that $\psi_{\hat{\tau}}(g.(x,t,p))=g.\psi_{\hat{\tau}}(x,t,p)$. This shows that $\psi_{k-1,k}$ is $G$-equivariant as we wanted.

Therefore, for each $k\geq 0$, we have $\psi_{k-1,k}\circ\iota_{k-1,k}=\mathrm{id}_{\Delta M^{(k-1)}X_\mathcal{U}\cup\Delta X_\mathcal{U}}$ and $\iota_{k-1,k}\circ\psi_{k-1,k}\Gsimeq\mathrm{id}_{\Delta M^{(k)}X_\mathcal{U}\cup\Delta X_\mathcal{U}}$ where $\iota_{k-1,k}:\Delta M^{(k-1)}X_\mathcal{U}\cup\Delta X_\mathcal{U}\hooktoG\Delta M^{(k)}X_\mathcal{U}\cup\Delta X_\mathcal{U}$ is the canonical inclusion. Letting $k$ vary, the infinite concatenation of these $G$-deformation retractions in the $t$-intervals $[1/2^{k+1},1/2^k]$ gives a $G$-deformation retraction of $\Delta MX_\mathcal{U}$ onto $\Delta X_\mathcal{U}$.
\end{proof}

Now we recall the notion of \emph{partition of unity}. Suppose a topological space $X$ and its open cover $\mathcal{U}:=\{U_\lambda\}_{\lambda\in \Lambda}$ are given. Then, a set of maps $\{\psi_\lambda:X\rightarrow [0,1]\}_{\lambda\in\Lambda}$ is called \emph{a partition of unity subordinate to $\mathcal{U}$} if these maps  satisfy the following conditions:
\begin{enumerate}[label=(\roman*)]
    \item $0\leq\psi_\lambda(x)\leq 1$ for all $x\in X$.
    
    \item $\supp[\psi_\lambda]\subseteq U_\lambda$ for all $\lambda\in \Lambda$.
    
    \item The family of supports $\{\supp[\psi_\lambda]\}_{\lambda\in\Lambda}$ is locally finite.
    
    \item $\sum_{i\in I}\psi_\lambda(x)=1$ for all $x\in X$.
\end{enumerate}

Next, assume that $(X,\alpha_X)$ is a $G$-space and $\mathcal{U}$ is a $G$-invariant open cover. Then, finally, a partition of unity $\{\psi_\lambda\}_{\lambda\in\Lambda}$ subordinate to $\mathcal{U}$ is said to be \emph{$G$-partition of unity} if they further satisfy the following fifth condition:

\begin{itemize}
\item[(v)] $\psi_{\alpha_\Lambda(g,\lambda)}(\alpha_X(g,x))=\psi_\lambda(x)$ for all $\lambda\in\Lambda$, $g\in G$, and $x\in X$. 
\end{itemize}

Then, we have the following lemma.

\begin{lemma}\label{lemma:Gptofuty}
Suppose a finite group $G$, a $G$-space $(X,\alpha_X)$ and a $G$-invariant open cover $\mathcal{U}:=\{U_\lambda\}_{\lambda\in \Lambda}$ are given. If $\{\psi_\lambda\}_{\lambda\in\Lambda}$ is a partition of unity subordinate to $\mathcal{U}$, then $\{\widetilde{\psi}_\lambda\}_{\lambda\in\Lambda}$ becomes a $G$-partition of unity subordinate to $\mathcal{U}$ where, for all $\lambda\in \Lambda$, $\widetilde{\psi}_\lambda$ is defined by
\begin{align*}
    \widetilde{\psi}_\lambda:X&\longrightarrow [0,1]\\
    x&\longmapsto\frac{1}{\vert G \vert}\sum_{g\in G}\psi_{\alpha_\Lambda(g,\lambda)}(\alpha_X(g,x))
\end{align*}
\end{lemma}
\begin{proof}
It is easy to verify that $\{\widetilde{\psi}_\lambda\}_{\lambda\in\Lambda}$ satisfy the conditions (i), (iv), and (v) for $G$-partition of unity subordinate to $\mathcal{U}$ by the construction. In order to verify conditions (ii) and (iii), first observe the following:
{\small $$\supp[\widetilde{\psi}_\lambda]=\overline{\{x\in X:\widetilde{\psi}_\lambda(x)>0\}}=\bigcup_{g\in G}\overline{\{x\in X:\psi_{\alpha_\Lambda(g,\lambda)}(\alpha_X(g,x))>0\}}=\bigcup_{g\in G}\alpha_\Lambda(g^{-1},\supp[\psi_{\alpha_X(g,\lambda)}]).$$}
Hence, since $\supp[\psi_{\alpha_\Lambda(g,\lambda)}]\subseteq U_{\alpha_\Lambda(g,\lambda)}=\alpha_X(g,U_\lambda)$, one can conclude that $\supp[\widetilde{\psi}_\lambda]\subseteq U_\lambda$. This proves condition (ii).

Next, for each $x\in X$ choose an open neighborhood $V_x\subseteq X$ of $x$ such that $V_x\bigcap\supp[\psi_\lambda]\neq\emptyset$ for only finitely many $\lambda$'s. Now, let $\widetilde{V}_x:=\bigcap_{g\in G}\alpha_X(g^{-1},V_{\alpha_X(g,x)})$. Then, if $\widetilde{V}_x\bigcap\supp[\widetilde{\psi}_\lambda]\neq\emptyset$ for some $\lambda\in\Lambda$, then there exist $g\in G$ such that $\alpha_X(g^{-1},\supp[\psi_{\alpha_\Lambda(g,\lambda)}])\bigcap \alpha_X(g^{-1},V_{\alpha_X(g,x)})\neq\emptyset$ which implies that $\supp[\psi_{\alpha_\Lambda(g,\lambda)}]\bigcap V_{\alpha_X(g,x)}\neq\emptyset$. By the finiteness of $G$ and the construction of each $V_x$, one can conclude that $\widetilde{V}_x\bigcap\supp[\widetilde{\psi}_\lambda]\neq\emptyset$ only for finitely many $\lambda$'s. This proves condition (iii) and completes the proof.
\end{proof}

\begin{proposition}\label{prop:4g2}
Let $G$ be a finite group, $X$ be a $G$-paracompact space, and $\mathcal{U}:=\{U_\lambda\}_{\lambda\in \Lambda}$ be a $G$-invariant open cover of $X$. Then, $X$ and $\Delta X_\mathcal{U}$ are $G$-homotopy equivalent.
\end{proposition}
\begin{proof}
Recall that
$$\Delta X_\mathcal{U}:=\left(\bigsqcup_{\sigma_0\prec\sigma_1\cdots\prec\sigma_n} (U_{\sigma_n}\times\Delta_n)\right)\slash\sim.$$
Also, since $X$ is paracompact, there is a partition of unity $\{\psi_\lambda:U_\lambda\rightarrow [0,1]\}_{\lambda\in\Lambda}$ subordinate to $\mathcal{U}$. Hence, by \Cref{lemma:Gptofuty}, there exists a $G$-partition of unity $\{\widetilde{\psi}_\lambda\}_{\lambda\in\Lambda}$ subordinate to $\mathcal{U}$.

Then, we consider the following $G$-maps.
\begin{align*}
    \pi:\Delta X_\mathcal{U}&\toG X\\
    (x,p)&\longmapsto x
\end{align*}

and

\begin{align*}
    \eta:X&\toG\Delta X_\mathcal{U}\\
    x&\longmapsto \big(x,(\widetilde{\psi}_\lambda(x))_{\lambda\in\Lambda}\big).
\end{align*}

Then, obviously we have $\pi\circ\eta=\mathrm{id}_X$ and it is easy to verify that $\eta\circ\pi\Gsimeq\mathrm{id}_{\Delta X_\mathcal{U}}$ by considering the linear interpolation homotopy.
\end{proof}

\begin{proof}[\textbf{Proof of \Cref{prop:Gnerve}}]
Combine \Cref{prop:4g1} and \Cref{prop:4g2}.
\end{proof}

\section{Proof of \Cref{thm:Z2hausmannsphere}.}

First we need to recall:
\begin{theorem}[{Jung's theorem \cite[Lemma 2]{katz1983filling}}]\label{thm:Jung}
Assume that a subset $A$ of $\Sp^m$ satisfies $\diam(A)<\zeta_m:=\arccos\left(\frac{-1}{m+1}\right)$. Then, $A$ is contained in an open hemisphere.
\end{theorem}

\begin{proof}[\textbf{Proof of \Cref{thm:Z2hausmannsphere}}]\label{app:proof-Z2-H}
Let's consider the case for $n>1$, first. In \cite[Section A.4]{lim2024vietoris}, motivated by Hausmann's original proof \cite[Theorem 3.5]{hausmann1994vietoris}, the authors constructed a certain map $T_r:\vert\vr_r(\Sp^n)\vert\rightarrow\Sp^n$ and verified that $T_r$ is indeed a homotopy equivalence. Here is a reminder of how the map $T_r$ is defined:

Choose a total order on the points of $\Sp^n$. From now on, whenever we describe a finite subset of $\Sp^n$ by $\{x_0,\dots,x_q\}$, we suppose that $x_0<x_1<\cdots<x_q$. Let $r\in (0,\zeta_n]$. We shall associate to each $q$-simplex $\sigma:=\{x_0,\dots,x_q\}\in \vr_r(\Sp^n)$ a singular $q$-simplex $T_\sigma:\Delta_q\longrightarrow\Sp^n$. Recall that the standard Euclidean $q$-simplex $\Delta_q$ is defined in the following way:
$$\Delta_q:=\left\{\sum_{i=0}^qt_ie_i:t_i\in [0,1]\text{ and }\sum_{i=0}^q t_i=1\right\}.$$
This map $T_\sigma$ is defined inductively as follows: set $T_\sigma(e_0)=x_0$. Suppose that $T_\sigma(z)$ is defined for $y=\sum_{i=0}^{p-1}s_ie_i$. Let $z:=\sum_{i=0}^p t_ie_i$. If $t_p=1$, we pose $T_\sigma(z)=x_p$. Otherwise, let
$$x:=T_\sigma\left(\frac{1}{1-t_p}\sum_{i=0}^{p-1}t_i e_i\right).$$
We define $T_\sigma(z)$ as the point on the unique shortest geodesic joining $x$ to $x_p$ with $d_{\Sp^n}(x,T_\sigma(z))=t_p\cdot d_{\Sp^n}(x,x_p)$ (the unique shortest geodesic exists since the geodesic convex hull of $\{x_0,\dots,x_q\}$ must be contained in some open ball of radius smaller than $\frac{\pi}{2}$ by \Cref{thm:Jung}). To sum up, $T_\sigma$ is defined inductively on $\Delta_p$ for $p\leq q$ as the \emph{geodesic join} of $T_\sigma(\Delta_{p-1})$ with $x_p$.

If $\sigma'$ is a a face of $\sigma$ of dimension $p$, we form the euclidean sub $p$-simplex $\Delta'$ of $\Delta_q$ formed by the points $\sum_{i=0}^q t_i e_i\in\Delta_q$ with $t_i=0$ if $x_i\notin\sigma'$. One can check by induction on $\dim\sigma'$ that
\begin{equation}\label{eq:Hausmannmap}
    T_{\sigma'}=T_\sigma\vert_{\Delta'}.
\end{equation}

By (\ref{eq:Hausmannmap}), the correspondence $\sigma\mapsto T_\sigma$ gives rise to a map
$$T_r:\vert\vr_r(\Sp^n)\vert\longrightarrow\Sp^n.$$

Hence, it is enough to modify the map $T_r$ to a $\Z_2$-map, which can be done as follows. Instead of choosing a total order on the points of $\Sp^n$, one can choose a total order on the set of all antipodal pairs $\{\{x,-x\}:x\in\Sp^n\}$ i.e. we choose a total order on $\R\mathbb{P}^n$. Note that if a subset $\sigma:=\{x_0,\dots,x_q\}$ satisfies $\diam(\sigma)<\zeta_n$, then $\sigma$ cannot contain an antipodal pair. Hence, one can construct the map $T_r$ as above and it is easy to verify that $T_r$ is a $\Z_2$-map.

Finally, consider the case for $n=1$. By \Cref{prop:Gnerve}, we know that $\vert\vr_r(\Sp^1)\vert$ and $B_{r/2}(\Sp^1,E(\Sp^1))$ are $\Z_2$-homotopy equivalent. Hence, it is enough to show that the $\Z_2$-homotopy equivalence between $B_{r/2}(\Sp^1,E(\Sp^1))$ and $\Sp^1$. However, in \cite{lim2021some}, the authors established a homotopy equivalence $m_r:B_{r/2}(\Sp^1,E(\Sp^1))\longrightarrow\Sp^1$ (see \cite[Theorem 3.25]{lim2021some}). Finally, since $m_r$ is a certain ``center of mass" type map, it is simple to check that this map $m_r$ is indeed a $\Z_2$-map. This completes the proof.
\end{proof}

\section{Proofs from \Cref{sec:finiteness}.}\label{app:proofs-finiteness}

\begin{proof}[\textbf{Proof of \Cref{thm:Gequivprcpt}}]
Fix an arbitrary sequence $\{X_n\}_{n\geq 1}\subseteq\mathcal{F}$. Since $\mathcal{F}$ is uniformly $G$-totally bounded, for each $n,k\geq 1$ one can choose a $G$-invariant $\frac{1}{k}$-net $U_{n,k}=\{x_{n,k,1},\dots,x_{n,k,i_{n,k}}\}$ of $X_n$ such that $\vert U_{n,k} \vert=i_{n,k}\leq N_{\mathcal{F}}^G\left(\frac{1}{k}\right)$. Furthermore, one can induce a natural group action $\omega_{n,k}$ on $[i_{n,k}]=\{1,\dots,i_{n,k}\}$ such that $\alpha_X(g,x_{n,k,j})=x_{n,k,\omega_{n,k}(g,j)}$ for any $g\in G$ and $j\in[i_{n,k}]$.

For $k=1$, by the pigeon hole principle, one may choose a strictly increasing function $f^{(1)}:\Z_{>0}\longrightarrow\Z_{>0}$ such that the subsequence $\left\{X_{f^{(1)}(n)}\right\}_{n\geq 1}\subseteq\{X_n\}_{n\geq 1}$ satisfies the following: for any $n,n'\geq 1$, we have that $\vert U_{f^{(1)}(n),1} \vert=i_{f^{(1)}(n),1}=i_{f^{(1)}(n'),1}=\vert U_{f^{(1)}(n'),1} \vert=:N_1\leq N_{\mathcal{F}}^G(1)$ and $\omega_{f^{(1)}(n),1}=\omega_{f^{(1)}(n'),1}=:\omega_1$.

Inductively, for each $k\geq 1$ one can verify that there is a a strictly increasing function $f^{(k)}:\Z_{>0}\longrightarrow f^{(k-1)}(\Z_{>0})$ such that the sequence $\left\{X_{f^{(k)}(n)}\right\}_{n\geq 1}$ satisfies the following: for any $n,n'\geq 1$ and $1\leq l \leq k$, we have that $\vert U_{f^{(k)}(n),l} \vert=i_{f^{(k)}(n),l}=i_{f^{(k)}(n'),l}=\vert U_{f^{(k)}(n'),l} \vert=:N_l\leq N_{\mathcal{F}}^G\left(\frac{1}{l}\right)$ and $\omega_{f^{(k)}(n),l}=\omega_{f^{(k)}(n'),l}=:\omega_l$.

Hence, by the Cantor's diagonal argument, there exists a strictly increasing function $f:\Z_{>0}\longrightarrow\Z_{>0}$ such that for each fixed $k\geq 1$ and any $n\geq k$, we have that $\vert U_{f(n),k} \vert=i_{f(n),k}=i_{f(n'),k}=\vert U_{f(n'),k} \vert=:N_k\leq N_{\mathcal{F}}^G\left(\frac{1}{k}\right)$ and $\omega_{f(n),k}=\omega_{f(n'),k}=:\omega_k$. Without loss of generality, one may assume that $f=\mathrm{id_{\Z_{>0}}}$. So, from now on, for any $n\geq k$ one can assume that $U_{n,k}=\{x_{n,k,1},x_{n,k,2},\dots,x_{n,k,N_k}\}$.

We will apply the Cantor's diagonal argument once more. Recall that, since $\mathcal{F}$ is uniformly $G$-totally bounded, there exists $D>0$ such that $\diam(X_n)\leq D$ for all $n\geq 1$. Hence, for each fixed $n\geq k,k'\geq 1$, $i\in\{1,\dots,N_k\}$, and $i'\in\{1,\dots,N_{k'}\}$, we have that $\{d_{X_n}(x_{n,k,i},x_{n,k',i'})\}$ is a a subset of a compact space $[0,D]$. Therefore, again by the Cantor's diagonal argument, there exists a strictly increasing function $h:\Z_{>0}\longrightarrow\Z_{>0}$ such that
$$\lim_{n\rightarrow\infty}d_{X_{h(n)}}(x_{{h(n)},k,i},x_{{h(n)},k',i'})$$
exists for any $k,k'\geq 1$, $i\in\{1,\dots,N_k\}$, and $i'\in\{1,\dots,N_{k'}\}$. Again, without loss of generality, one may assume that $h=\mathrm{id_{\Z_{>0}}}$.

Finally, we define $X:=\bigcup_{k\geq 1}\{x_{k,i}\}_{i=1}^{N_k}$ and $d_X(x_{k,i},x_{k',i'}):=\lim_{n\rightarrow\infty} d_{X_n}(x_{n,k,i},x_{x,k',i'})$
for any $k,k'\geq 1$, $i\in\{1,\dots,N_k\}$, and $i'\in\{1,\dots,N_{k'}\}$. Furthermore, we define the $G$-action on $X$ as follows. $\alpha_X:G\times X\longrightarrow X$ such that $\alpha_X(g,x_{k,i}):=x_{k,\omega_k(g,i)}$. Then, we have that

\begin{align*}
    d_X\big(\alpha_X(g,x_{k,i}),\alpha_X(g,x_{k',i'})\big)&=d_X(x_{k,\omega_k(g,i)},x_{k',\omega_{k'}(g,i')})=\lim_{n\rightarrow\infty} d_{X_n}(x_{n,k,\omega_k(g,i)},x_{n,k',\omega_{k'}(g,i')})\\
    &=\lim_{n\rightarrow\infty} d_{X_n}\big(\alpha_X(g,x_{n,k,i}),\alpha_X(g,x_{n,k',i'})\big)=\lim_{n\rightarrow\infty} d_{X_n}(x_{n,k,i},x_{n,k',i'})\\
    &=d_X(x_{k,i},x_{k',i'}).
\end{align*}

for any $g\in G$, $k,k'\geq 1$, $i\in\{1,\dots,N_k\}$, and $i'\in\{1,\dots,N_{k'}\}$. Hence, $(X,d_X,\alpha_X)$ is indeed a $G$-metric space.

Next, we claim that $\{x_{k,i}\}_{i=1}^{N_k}$ is a $G$-invariant $\frac{1}{k}$-net of $X$ for each $k\geq 1$. Fix an arbitrary point $x_{k',i'}\in X$. Then, for each $n\geq 1$, there exists $k\geq 1$ and $i\in\{1,\dots,N_k\}$ such that $d_{X_n}(x_{n,k,i},x_{n,k',i'})<\frac{1}{k}$. Since $N_k$ is finite and does not depend on $n$, there exists $k\geq 1$ and $i\in\{1,\dots,N_k\}$ such that  $d_{X_n}(x_{n,k,i},x_{n,k',i'})<\frac{1}{k}$ for infinitely many $n$. Therefore, from the definition we have that $d_X(x_{k,i},x_{k',i'})<\frac{1}{k}$. Hence, $\{x_{k,i}\}_{i=1}^{N_k}$ is a $G$-invariant $\frac{1}{k}$-net of $X$ as we required.

Finally, observe that
$$\dgh^G(X,X_n)\leq\dgh^G(X,\{x_{k,i}\}_{i=1}^{N_k})+\dgh^G(\{x_{k,i}\}_{i=1}^{N_k},\{x_{n,k,i}\}_{i=1}^{N_k})+\dgh^G(X_n,\{x_{n,k,i}\}_{i=1}^{N_k})$$
by the triangle inequality. Since $\dgh^G(X,\{x_{k,i}\}_{i=1}^{N_k}),\dgh^G(X_n,\{x_{n,k,i}\}_{i=1}^{N_k})\leq\frac{1}{k}$ and

$\lim_{n\rightarrow\infty}\dgh^G(\{x_{k,i}\}_{i=1}^{N_k},\{x_{n,k,i}\}_{i=1}^{N_k})=0$, one can conclude that $\lim_{n\rightarrow\infty}\dgh^G(X,X_n)=0$ as we required. This completes the proof.
\end{proof}

\begin{proof}[\textbf{Proof of \Cref{lemma:prcptcoverbdd}}]
Our proof adapts the strategy of \cite[Lemma 6.4]{memoli2024persistent}, incorporating the necessary modifications to accommodate the $G$-equivariant setting.

Let $\varepsilon':=\varepsilon/3$, $N:=N_{\mathcal{F}}^G(\varepsilon')$, $j\in\{1,\dots,N\}$, and $\mathcal{F}_j:=\{X\in\mathcal{F}:N_X^G(\varepsilon')=j\}$. We will show that $\mathcal{F}_j$ can be covered by $j^{\frac{j^2-j}{2}}\cdot\vert\hom(G,\mathcal{S}_j)\vert$ many $\varepsilon$-balls.

For each $X\in\mathcal{F}_j$, one can choose a $G$-invariant $\varepsilon'$-net $U_X=\{x_1,\dots,x_j\}$ of $X$ with the cardinality $j$. We finish the proof via the following four steps.

\noindent
\textbf{Step 1.} For each $X\in\mathcal{F}_j$, there is a vector $v_X\in[0,2\varepsilon'j]^{\frac{j^2-j}{2}}$ associated to $X$.

Since $X$ is connected, following the same idea in the proof of \cite[Lemma 6.4]{memoli2019persistent}, one can show that $d_X(x_i,x_{i'})\leq 2\varepsilon'j$ for any $x_i,x_{i'}\in U_X$. Hence, one can choose the vector $v_X:=\big(d_X(x_i,x_{i'})\big)_{i< i'}\in [0,2\varepsilon'j]^{\frac{j^2-j}{2}}$.

\smallskip
\noindent
\textbf{Step 2.} For each $X\in\mathcal{F}_j$, there is a $G$-action $\omega_X\in\hom(G,\mathcal{S}_j)$ on $[j]$ associated to $X$.

One can define $\omega_X(g,i):=i'$ if $x_{i'}=\alpha_X(g,x_i)$ for all $g\in G$ and $x_i,x_{i'}\in U_X$.

Hence, by steps 1 and 2, we know that for each $X\in\mathcal{F}_j$ one can choose an element $(v_X,\omega_X)\in [0,2\varepsilon'j]^{\frac{j^2-j}{2}}\times\hom(G,\mathcal{S}_j)$.

\smallskip
\noindent
\textbf{Step 3.} Let $X,Y\in\mathcal{F}_j$. If $\Vert v_X-v_Y \Vert_\infty<2\varepsilon'$ and $\omega_X=\omega_Y$, then $\dgh^G(X,Y)<\varepsilon$.

Since $\Vert v_X-v_Y \Vert_\infty<\varepsilon'$ and $\omega_X=\omega_Y$, there is a bijective $G$-map $\psi:U_X\toG U_Y$ such that $\vert d_X(x_i,x_{i'})-d_Y(\psi(x_i),\psi(x_{i'}))\vert<2\varepsilon'$ for all $x_i,x_{i'}\in U_X$. Therefore, $\dgh^G(U_X,U_Y)\leq\frac{1}{2}\dis(\psi)<\varepsilon'$. Hence, by the triangle inequality,
$$\dgh^G(X,Y)\leq \dgh^G(X,U_X)+\dgh^G(U_X,U_Y)+\dgh^G(Y,U_Y)<3\varepsilon'=\varepsilon.$$

\noindent
\textbf{Step 4.} $\mathcal{F}_j$ can be covered by $j^{\frac{j^2-j}{2}}\cdot\vert\hom(G,\mathcal{S}_j)\vert$ many $\varepsilon$-balls.

Obviously, $[0,2\varepsilon'j]^{\frac{j^2-j}{2}}$ can be partitioned into $j^{\frac{j^2-j}{2}}$ many cells of all equal size. Hence, $[0,2\varepsilon'j]^{\frac{j^2-j}{2}}\times\hom(G,\mathcal{S}_j)$ can be partitioned into $j^{\frac{j^2-j}{2}}\cdot\vert\hom(G,\mathcal{S}_j)\vert$ many cells. For each cell, choose a $X\in\mathcal{F}_j$ such that $(v_X,\omega_X)$ belongs to the cell. Then, if $Y\in\mathcal{F}_j$ is in the same cell as $X$, then $\dgh^G(X,Y)<\varepsilon$ by the step 3. This completes the proof.
\end{proof}
\end{document}